\documentclass[]{amsart}
\usepackage[margin=1.5in]{geometry}
\usepackage{amsmath}
\usepackage{amscd}
\usepackage{amssymb}
\usepackage{amsthm}
\usepackage{ascmac}
\usepackage{url}
\usepackage{textcomp}
\usepackage{mathtools}
\usepackage{tikz-cd}
\usetikzlibrary{cd}
\usepackage{mathrsfs}
\usepackage{hyperref}
\hypersetup{breaklinks=true} 
\usepackage[all]{xy} 
\usepackage{color}
\usepackage{enumitem}
\setenumerate{
label=(\arabic*),
font=\normalfont
}
\setlist[enumerate,2]{label=(\alph*),ref=(\alph*)}
\theoremstyle{definition}
\newtheorem{thm}{Theorem}[section]
\newtheorem{lem}[thm]{Lemma}
\newtheorem{cor}[thm]{Corollary}
\newtheorem{prop}[thm]{Proposition}
\newtheorem{defn-prop}[thm]{Definition-Proposition}
\newtheorem{defn-thm}[thm]{Definition-Theorem}

\newtheorem{rem}[thm]{Remark}
\newtheorem{claim}[thm]{Claim}
\newtheorem{defn}[thm]{Definition}

\newtheorem{construction}[thm]{Construction}

\newcommand{\rank}{\mathop{\mathrm{rank}}}

\def\p{{\mathfrak p}}

\def\F{{\mathbb F}}
\def\G{{\mathbb G}}

\def\P{{\mathbb P}}
\def\Q{{\mathbb Q}}
\def\C{{\mathbb C}}
\def\Z{{\mathbb Z}}

\def\cO{\mathcal{O}}
\newcommand{\cC}{\mathcal{C}}

\def\red{\mathrm{red}}

\DeclareMathOperator\SL{\mathrm{SL}}

\DeclareMathOperator\Aut{\mathrm{Aut}}

\DeclareMathOperator\Gal{\mathrm{Gal}}

\DeclareMathOperator\Pic{\mathrm{Pic}}
\DeclareMathOperator\Sym{\mathrm{Sym}}
\DeclareMathOperator\Frac{\mathrm{Frac}}
\DeclareMathOperator\Spec{\mathrm{Spec}}
\DeclareMathOperator\Proj{\mathrm{Proj}}
\DeclareMathOperator\Shaf{\mathrm{Shaf}}

\DeclareMathOperator\Hom{\mathrm{Hom}}

\DeclareMathOperator\Gr{\mathrm{Gr}}
\DeclareMathOperator\Fl{\mathrm{Fl}}

\DeclareMathOperator\Br{\mathrm{Br}}

\DeclareMathOperator\diag{\mathrm{diag}}
\DeclareMathOperator{\PGL}{PGL}
\DeclareMathOperator{\PSL}{PSL}
\DeclareMathOperator{\spl}{split}
\DeclareMathOperator{\id}{id}
\DeclareMathOperator{\Alt}{Alt}
\DeclareMathOperator{\Pf}{Pf}

\newcommand{\sV}{\mathscr{V}}
\newcommand{\sA}{\mathscr{A}}
\newcommand{\sS}{\mathscr{S}}

\newcommand{\trans}[1]{{^t}\!{#1}}

\usepackage[
backend=biber,
bibstyle=ieee-alphabetic,
citestyle=ieee-alphabetic,
sorting=nyt,
isbn=false,
url=false,
doi=false,
giveninits=true,
maxnames=10,
labelalpha=true,
maxalphanames=4,
dashed=false,
]{biblatex}
\DeclareCaseLangs{}
\DeclareFieldFormat*{date}{\mkbibparens{#1}}
\DeclareFieldFormat*{volume}{\mkbibbold{#1}}
\DeclareFieldFormat*{title}{\mkbibemph{#1\isdot}}
\DeclareFieldFormat*{journaltitle}{#1}
\DeclareFieldFormat*{pages}{{#1}}

\AtEveryBibitem{%
  \ifentrytype{online}{%
    \DeclareFieldFormat*{date}{#1}%
  }{%
  }%
}

\title{Quintic del Pezzo threefolds in positive and mixed characteristic}

\begin{document}

\numberwithin{equation}{subsection}

\author{Tetsushi Ito}
\address{
Department of Mathematics, Faculty of Science, Kyoto University
Kyoto, 606--8502, Japan}
\email{tetsushi@math.kyoto-u.ac.jp}
\thanks{Supported by JSPS KAKENHI Grant Numbers
JP21K18577, JP23K20204, JP23K20786, JP24K21512 and JP25K00905.}

\author{Akihiro Kanemitsu}
\address{Department of Mathematical Sciences, Graduate School of Science, Tokyo Metropolitan University, 1-1 Minami-Osawa, Hachioji-shi, Tokyo 192-0397, Japan}
\email{kanemitsu@tmu.ac.jp}
\thanks{Supported by JSPS KAKENHI Grant Numbers 23K12948.}

\author{Teppei Takamatsu}
\address{Department of Mathematics (Hakubi Center),
Kyoto University,
Kitashirakawa, Oiwake-cho, Sakyo-ku,
Kyoto 606-8502, JAPAN}
\email{teppeitakamatsu.math@gmail.com}
\thanks{Supported by JSPS KAKENHI Grant Numbers JP22KJ1780 and JP25K17228.}

\author{Yuuji Tanaka}
\address{Beijing Institute of Mathematical Sciences and Applications (BIMSA), No. 544, Hefangkou Village, Huaibei Town, Huairou District, Beijing 101408, China}
\email{ytanaka@bimsa.cn}
\thanks{Supported by JSPS KAKENHI Grant Numbers JP21H00973, JP21K03246, and Startup Grant at BIMSA}

\begin{abstract}
We show that smooth quintic del Pezzo threefolds over arbitrary base schemes are classified by non-degenerate ternary symmetric bilinear forms.
Then we describe the automorphism group schemes, the Hilbert schemes of lines and the orbit structures of quintic del Pezzo threefolds, and we find several new phenomena in characteristic two.
As arithmetic applications,
we prove a refinement of the Shafarevich conjecture, and prove that there are exactly two isomorphism classes of quintic del Pezzo threefolds over the ring of rational integers.
\end{abstract}


\maketitle

\section{Introduction}
A \emph{quintic del Pezzo threefold} $Y$, also called a \emph{$V_5$-variety} $Y$, is a smooth Fano variety $Y$ with index two and $(-K_Y/2)^3=5$.
Among Fano threefolds, quintic del Pezzo threefolds form one of the most fundamental, but nontrivial, classes after the projective space $\P^3$ and the hyperquadric $Q^3$,
and hence such varieties have been a subject of
many studies of Fano threefolds.
Especially, the following three descriptions of $V_5$-varieties $Y$ over algebraically closed fields of characteristic zero play an important role in the study of their geometry \cite{Iskovskih1, Fuj90, Muk92, Ili94}
(see also \cite{DK19} for $V_5$ varieties over non-closed fields of characteristic zero):
\begin{enumerate}
 \item \label{description1} \textbf{(Projection from a line)} The projection from a line maps $Y$ to a three-dimensional quadric $Q^3$.
 \item \label{description2} \textbf{(Grassmann embedding)} $Y$ is a codimension $3$ linear section of the six-dimensional Grassmann variety $\Gr(5,3) \subset \P^9$,
 where $\Gr(5,3)$ is embedded in $\P^9$ via the Pl\"ucker embedding.
 \item \label{description3} \textbf{(Varieties of sums of powers)}
 $Y$ is the variety of sums of powers of a non-degenerate ternary quadratic form.
 Equivalently, it is the variety of trisecant lines to the projected Veronese surface $\P^2 \subset \P^4$.
\end{enumerate}

The descriptions \ref{description2} and \ref{description3} show that, over an algebraically closed field of characteristic zero, $V_5$-varieties are classified by \emph{nets of quinary alternating forms} and \emph{ternary quadratic forms}, respectively.
These descriptions have several important consequences for the geometry of $V_5$-varieties.
For example, over an algebraically closed field $k$ of characteristic zero, it follows:
\begin{enumerate}\setcounter{enumi}{3}
\item \label{application1} There exists a unique $k$-isomorphism class of $V_5$-varieties over $k$.
\item \label{application2} The Hilbert scheme of lines on a $V_5$-variety $Y$ is isomorphic to $\P^2$.
\item \label{application3} The automorphism group of a $V_5$-variety $Y$ is isomorphic to $\PGL_2$,
and $Y$ has three orbits under the action of $\Aut(Y) \simeq \PGL_2$.
\end{enumerate}

Contrary to the above case in characteristic zero,
the geometry of $V_5$-varieties in positive or mixed characteristic,
especially in characteristic two,
has not been much studied.
The descriptions as in \ref{description1} or \ref{description2} were announced in \cite{Fuj90,Meg98},
but the detailed proofs were not given (see Remark~\ref{rem:W5_p=2}).
The description \ref{description3} or the applications \ref{application1}--\ref{application3} does not seem to have been studied before.
More generally, there are very few studies on
$V_5$-schemes over general base schemes,
despite their importance within moduli theoretic
and number theoretic applications, which we have in mind.

The purpose of this paper is to generalize the above descriptions and applications to
all \emph{$V_5$-schemes} $Y$ over \emph{arbitrary} base schemes $B$.
The main results of this paper are summarized as follows:
\begin{enumerate}
\item \textbf{(Classification theorem)}
For an arbitrary scheme $B$, we show that $B$-isomorphism classes of $V_5$-schemes over $B$
are in natural bijection with $B$-similarity classes of non-degenerate ternary symmetric bilinear forms over $B$
(see Theorem \ref{thm:classification_W5}
for the precise statement).

\item \textbf{(Geometric applications)}
We describe the automorphism group schemes, the Hilbert schemes of lines, and the orbit structures of $V_5$-varieties explicitly (see Theorems \ref{thm:hilb_of_lines} and \ref{thm:splitV5_over_field}).

\item \textbf{(Arithmetic finiteness results)}
We study good reduction of $V_5$-varieties over $p$-adic fields,
and prove a refinement of the Shafarevich conjecture.
We also prove that there are exactly two isomorphism classes of $V_5$-schemes over $\mathbb{Z}$.
See Subsection~\ref{subsection:arithmetic_finiteness_results}
and
Section~\ref{section:arithmetic}
for details.
\end{enumerate}

The use of ternary symmetric bilinear forms,
rather than ternary quadratic forms,
is essential in this paper.
Note that, when $2$ is not invertible on the base scheme,
the theory of quadratic forms and that of symmetric bilinear forms are not equivalent,
while the latter provides the correct counterpart of $V_5$-schemes in our classification theorem.
This, in turn, leads to finding of several new interesting properties of $V_5$-varieties  in characteristic two.

\subsection{Classification of $V_5$-schemes over arbitrary schemes}

We give the precise statements of
our classification theorem for $V_5$-schemes.
To state it, we give the following definition (see Definitions~\ref{defn:W5} and \ref{defn:forms} for details).

\begin{defn}
 Let $B$ be a scheme.
\begin{enumerate}
\item A smooth projective scheme $Y$ over $B$
is called a \emph{$V_5$-scheme} over $B$
if, for every point $b \in B$, the fiber $Y_b$ is a $V_5$-variety.

\item A \emph{net of quinary alternating forms of rank $4$ over $B$} is a triple
\[ (U_5,N_3,\nu), \]
where $U_5$ and $N_3$ are locally free sheaves on $B$ of rank $5$ and $3$ respectively,
and $\nu \colon \wedge^2 U_5 \to N_3$ is a surjective map that locally defines a three-dimensional family of alternating forms of rank $4$.

\item Similarly, a \emph{non-degenerate ternary symmetric bilinear form over $B$} is a triple
\[ (N_3,L,\varphi), \]
where $N_3$ and $L$ are locally free sheaves on $B$ of rank $3$ and $1$ respectively,
and $\varphi \colon S^2 N_3 \to L$ is a surjective map that locally defines a non-degenerate symmetric bilinear form.
\end{enumerate}
Two nets of quinary alternating forms, or two symmetric bilinear forms, are said to be \emph{$B$-similar} (denoted by $\sim_B$) if they are isomorphic modulo twisting with line bundles.
\end{defn}

The following is the first main result of this paper.

\begin{thm}[Classification theorem, see Theorem~\ref{thm:W5description}]
\label{thm:classification_W5}
Let $B$ be a scheme.
There are natural bijections among the following three sets:
\begin{itemize}
\item The set of $B$-isomorphism classes of $V_5$-schemes over $B$:
\[ \{ Y \mid \text{$Y$ is a $V_5$-scheme over $B$} \} /\simeq_B. \]

\item The set of $B$-similarity classes of
nets of quinary alternating forms of rank $4$ over $B$:
\[
\{ (U_5,N_3,\nu)  \mid  \text{$(U_5,N_3,\nu)$ is a net of quinary alternating forms of rank $4$}
\}/\sim_B.
\]

\item The set of $B$-similarity classes of
non-degenerate ternary symmetric bilinear forms over $B$:
\[
\{ (N_3,L,\varphi) \mid
\text{$(N_3,L,\varphi)$ is a non-degenerate ternary symmetric bilinear form}
\}/\sim_B.
\]
\end{itemize}

For the construction of the correspondence, see Construction~\ref{construction:W5}.
\end{thm}

\subsection{Split $V_5$-schemes}

Our classification theorem implies the existence of a \emph{ubiquitous} model of $V_5$-schemes,
which we call the \emph{split} $V_5$-scheme $Y_{\spl, B}$;
it corresponds to the split symmetric bilinear form
\[ (N_3,L,\varphi)=(\cO_B^3, \cO_B, \varphi_{\spl}), \]
where 
$\varphi_{\spl} \colon S^2 \cO_B^3 \to \cO_B$
is the symmetric bilinear form whose matrix is
\[
\begin{pmatrix}
 0&-1&0\\-1&0&0\\0&0&1
\end{pmatrix}.
\]
The split $V_5$-scheme $Y_{\spl,B}$ is defined by the following $5$ quadratic equations in $\P^6_B$ (see Definition-Proposition~\ref{defn-prop:splitform} (3)):
\[
\begin{cases}
a_0 a_4 - a_1 a_3 + a_2^2,\\
a_0 a_5 - a_1 a_4 + a_2 a_3,\\
a_0 a_6 - a_2 a_4 + a_3^2,\\
a_1 a_6 - a_2 a_5 + a_3 a_4,\\
a_2 a_6 - a_3 a_5 + a_4^2.
\end{cases}
\]
Thus the split $V_5$-scheme $Y_{\spl,B}$ is defined over $\Z$, i.e.,\ we have an isomorphism
$Y_{\spl,B} \simeq Y_{\spl,\Z} \times_{\Spec \Z} B$.

Over an algebraically closed field or a finite field,
we can show that every $V_5$ variety is split.

\begin{thm}[see Corollary~\ref{cor:W5eqn}]
 Let $k$ be an algebraically closed field or a finite field.
 Then every $V_5$-variety over $k$ is isomorphic to the split $V_5$-variety $Y_{\spl,k}$.
\end{thm}

\subsection{Geometric applications}

Theorem \ref{thm:classification_W5} provides various fundamental geometric properties of $V_5$-schemes.

\subsubsection{Automorphism group schemes and Hilbert schemes of lines}

First, we describe the automorphism group schemes and the Hilbert schemes of lines: 

\begin{thm}[see Propositions~\ref{prop:aut_split_W5} and \ref{proposition:Hilb_of_lines}]\label{thm:hilb_of_lines}
Let $B$ be a scheme, $Y$ a $V_5$-scheme over $B$, and $(N_3,L,\varphi)$ the corresponding non-degenerate ternary symmetric bilinear form.
Then the following hold:
\begin{enumerate}
 \item \label{theorem:Aut_Hilb1}The automorphism group scheme $\Aut_{Y/B}$ is isomorphic to the subgroup scheme of $\Aut_{\P(N_3)/B}$ preserving $(N_3,L,\varphi)$.
  \item The Hilbert scheme of lines $\Sigma(Y/B)$ is isomorphic to $\P(N_3)$. Moreover, this isomorphism is equivariant with respect to the identification in \ref{theorem:Aut_Hilb1}.
\end{enumerate}
\end{thm}

Note that the theory of symmetric bilinear forms needs special treatment in characteristic two.
This, in turn, implies that $V_5$-varieties exhibit special behavior in characteristic two. 
This is summarized as follows:

\begin{thm}[=Proposition~\ref{prop:aut_split_W5}, Corollary~\ref{cor:lines_split_W5}, and Propositions~\ref{prop:generalW5} and \ref{prop:char2act}]\label{thm:splitV5_over_field}
Let $k$ be a field of characteristic $p \geq 0$, and $Y= Y_{\spl}$ be the split $V_5$-variety over $k$.
(When $k$ is an algebraically closed field or a finite field, then any $V_5$-variety is split.)
We identify $\Sigma(Y)$ with  $\P(N_3) = \P^2$ whose coordinates are given by $(x,y,z)$.
We have the following:
\begin{enumerate}
 \item
 $\Aut_{Y/k}$ is isomorphic to
 \[
\begin{cases}
 \PGL_{2,k} & (p\neq 2),\\
 \text{a non-reduced group scheme $G$ with $G_{\red} \simeq \SL_{2,k}$} & (p=2).
\end{cases}
 \]
 Also, we have
 \[
\begin{cases}
 h^0(T_Y) =3, \quad h^1(T_Y)=0  & (p\neq 2),\\
 h^0(T_Y) =5, \quad h^1(T_Y)=2 & (p=2).
\end{cases}
 \]
 See also Section~\ref{section:characteristic_two} for the details on the non-reduced automorphism group and non-trivial first-order deformations in characteristic two.
 \item
The Hilbert scheme of lines $\Sigma(Y)=\P^2$ is decomposed into $\Aut_{Y/k,\red}$-orbits as follows:
\[
\begin{cases}
\text{(open orbit)} \sqcup V(z^2-2xy) & (p\neq 2),\\
\text{(open orbit)} \sqcup V(z) \sqcup   \{(0,0,1)\} & (p=2).
\end{cases}
\]
In other words, there are two types of lines when $p\neq 2$, but three types when $p=2$.
 \item
 $Y$ is decomposed into $\Aut_{Y,k,\red}$-orbits as follows:
\begin{itemize}
 \item When $p \neq 2$, then there are three orbits $O_3^{(p)}$, $O_2^{(p)}$ and $O_1^{(p)}$.
 \item When $p = 2$, then there are four orbits $O_3^{(2)}$, $O_2^{(2)}$, $O_1^{(2)}$ and $O_1^{(2)'}$.
\end{itemize}
Each orbit $O_i^{(p)}$ or $O_i^{(p)'}$ is an $i$-dimensional orbit.
\end{enumerate}
\end{thm}

Moreover, we describe the interplay between lines and orbits on $V_5$-varieties.
As is shown in the above theorem, the Hilbert scheme of lines $\Sigma(Y) \simeq \P^2$ admits the orbit decomposition
\[
\Sigma(Y) \simeq \P^2 
=
\begin{cases}
\text{(open orbit)} \sqcup V(z^2-2xy) & (p\neq 2),\\
\text{(open orbit)} \sqcup V(z) \sqcup \{(0,0,1)\} & (p=2).
\end{cases}
\]
A line belonging to the open orbit is called \emph{ordinary}, and that belonging to the one-dimensional orbit is called \emph{special}.
The last one corresponding to the point $(0,0,1)$ when $p=2$ is called \emph{exceptional}.
In characteristic zero, Furushima--Nakayama (\cite{FN89}) studied the lines on $V_5$,
and described the lines passing through a fixed point.
In Section~\ref{section:orbit}, we investigate a similar description and prove:
\begin{thm}[=Propositions~\ref{prop:numberoflinep=2} and \ref{prop:numberoflinep>2}]
 Let $Y$ be a quintic del Pezzo threefold over an algebraically closed field $k$.
 Then we have the following:
\begin{enumerate}
 \item For a point $P \in O_3^{(p)}$, there are exactly three lines passing through $P$. These lines are ordinary.
 \item For a point $P \in O_2^{(p)}$, there are exactly two lines passing through $P$. One of these lines is ordinary, and the other is special.
 \item For a point $P \in O_1^{(p)}$, there is a unique line passing through $P$, which is special.
 \item For a point $P \in O_1^{(2)'}$ ($p=2$), there are exactly two lines passing through $P$. One of these lines is special, and the other is exceptional.
\end{enumerate}
\end{thm}

\subsubsection{The variety of trisecant lines to a projected Veronese variety}

The above theorem enables us to interpret $Y$ as the variety of trisecant lines to a projected Veronese variety.
More generally, we will prove that a $V_5$-scheme $Y$ is a variety of trisecant lines to a projected Veronese surface over arbitrary \emph{reduced} base schemes $B$:
\begin{thm}[=Theorem~\ref{thm:trisecant}]
Let $B$ be a reduced scheme, $Y$ a $V_5$-scheme over $B$, and $(N_3,L,\varphi)$ the corresponding symmetric bilinear form.
Let $\P(N_3) \subset \P(\ker \varphi)$ be the projected Veronese surface, that is, the projection of the Veronese embedding $\P(N_3) \subset \P(S^2N_3)$ from the point $[\varphi] \in \P(S^2N_3)$.
 Then $Y$ is isomorphic to the Zariski closure of
 \[
 \{ \, [L] \in \Gr(\ker \varphi, 2) \mid
 \text{\rm $\P(L) \cap \P(N_3)$ consists of three points} 
 \} \, \subset \, \Gr(\ker \varphi, 2)
\]
with its reduced structure.
\end{thm}
In other words, $Y$ is recovered from the corresponding symmetric bilinear form $(N_3,L,\varphi)$ as the \emph{variety of orthogonal bases}.

\subsubsection{Remarks on the proofs}

Note that the interpretation of a $V_5$-variety $Y$
as a variety of trisecant lines to a projected Veronese surface works only when $B$ is reduced, since it only describes $Y$ as a set.
Possibly, Theorem~\ref{thm:classification_W5} for reduced base schemes $B$ can be proved more geometrically by using this interpretation.
In our approach, however, we will take a more algebraic construction to deal with the case of non-reduced base schemes.
More precisely, instead of constructing $V_5$-schemes directly from symmetric bilinear forms, we will establish correspondences
\begin{enumerate}[label=(\alph*)]
    \item \label{remark:proof1} between $V_5$-schemes and nets of alternating forms, and
    \item \label{remark:proof2} between nets of alternating forms and symmetric bilinear forms. 
\end{enumerate}
The former correspondence is, in essence, obtained by constructing the relative Grassmann embedding of $V_5$-schemes $Y$ over an arbitrary base scheme (See \cite{Kuz25} for a treatment in the characteristic zero case).

More difficult—and indeed the most technical core part of our proof—is in the latter stage: the analysis of symmetric bilinear forms constructed from nets of alternating forms. 
For example, in the course of the proof, we will prove:
\begin{prop}(cf.\ Remark~\ref{remark:discriminant_square})
Let $R$ be a ring, $(R^{\oplus 5}, R^{\oplus 3}, \nu)$ a net of symmetric bilinear forms on $\Spec R$, and $(R^{\oplus 3},R,  \varphi)$ the associated symmetric bilinear form.
Then the invariant obtained by taking the discriminant of $\varphi$ is a complete square. 
\end{prop}
This proposition is connected with the relative invariant of the prehomogeneous space $\wedge^2 R^{\oplus 5} \otimes R^{\oplus3}$ and, at least in characteristic zero, it follows from the computation of the degree of the relative invariant \cite[Proposition 14]{SK77}. In fact, we require more refined versions of this proposition (see Claims \ref{claim:invariant} and \ref{claim:discriminant_square}).
The proof of Claim \ref{claim:invariant} involves the invariant theory of nets of alternating forms from \cite[Section 3]{Och97} (cf.\ \cite[Subsection 3.2]{Arithmeticinvariant2} and \cite[Section 2]{Gyo90}), while that of Claim \ref{claim:discriminant_square}, which a priori needs rather intricate computations, is done by reducing the situation to a simpler case.

\begin{rem}
\begin{enumerate}
\item For 1-dimensional linear sections of $Y$ (over an algebraically closed field),  an analogous construction as in \ref{remark:proof1} was proved in \cite[Proposition 0.2]{Mukaicurve}.
See \cite[Subsection 4.5]{Bhargava-Ho} for the discussion over a $\Z[1/30]$-scheme.
\item In this paper, we only study smooth $V_5$-schemes.
It seems an interesting problem to extend our results to $V_5$-varieties with singularities.
Such generalizations, if possible,
would be regarded as a higher dimensional generalization of Bhargava's classification of quintic rings
(see \cite{Bhargava-HighercompositionIV}), which is related to (the dual of) $2$-dimensional linear sections of $Y$.
\item Our construction of nets of alternating forms from symmetric bilinear forms in \ref{remark:proof2} can be considered as a generalization of a similar construction from conics defined over a field of characteristic zero in \cite{Muk92} (see also \cite{Dol12}).
Note that this generalization itself also requires a subtle treatment since we need to deal with arbitrary base schemes, including the characteristic two case.
\end{enumerate}
\end{rem}

\subsection{Arithmetic finiteness results}
\label{subsection:arithmetic_finiteness_results}

As arithmetic applications of our main result,
we obtain several finiteness results for $V_5$-varieties over number fields.
See Section \ref{section:arithmetic} for more details.

Recall that the Shafarevich conjecture, which was originally studied by Shafarevich, Faltings, and Zarhin for abelian varieties,
addresses the finiteness of varieties within certain classes of varieties over a number field that admit good reduction outside
a fixed finite set of finiteness places.
It has been studied for Fano varieties; for example, Scholl \cite{Scholl}, Javanpeykar-Loughran \cite{Javanpeykar-Loughran:GoodReductionFano}, Licht \cite{Licht}, and we \cite{IKTT} proved Shafarevich-type finiteness results for del Pezzo surfaces and
certain Fano threefolds,
which include the case of $V_5$-varieties
(see also references in \cite{IKTT} for results on other classes of varieties).

As an application of
Theorem~\ref{thm:classification_W5},
we explicitly calculate the number of isomorphism classes of $V_5$-varieties admitting good reduction outside a fixed finite set of places.

\begin{thm}[see Theorem~\ref{thm:explicitShafarevichconjecture}]
 Let $F$ be a number field, $S$ a finite subset of finite places of $F$, $T$ the set of finite places of $F$ dividing $2$, and $U$ the set of real places of $F$.
 We put $r \coloneqq  \#(S \cup T \cup U)$ (the union is taken within the set of all places of $F$).
 Then the number of elements of the set
\[
\bigg\{
\,Y \ \bigg| \ 
\begin{array}{l}
\text{\rm $Y$ is a $V_5$-variety over $F$ admitting} \\
\text{\rm good reduction at $\p$ for any  $\p \notin S$}
\end{array}
\bigg\}
/\simeq_F
\]
 is equal to $2^{r-1}$.
 \end{thm}

Javanpeykar--Loughran proved the finiteness of the above set, but did not calculate the cardinality explicitly
(see \cite[Proposition 4.10]{Javanpeykar-Loughran:GoodReductionFano}).
Note that the set $T$ of finite places dividing $2$ plays a special role, because every $V_5$-varieties over a $2$-adic field always admit good reduction
(see Proposition \ref{prop:goodreductioncriteriaW5odd}
 (3)).

We also prove finiteness results on integral models of $V_5$-varieties over the ring of integers.

\begin{prop}
[see Corollary \ref{cor:classification_over_z}]
There are exactly two isomorphism classes of $V_5$-schemes over $\Z$.
\end{prop}

\begin{prop}
[see 
Theorem~\ref{thm:finitenessofintegral models}]
\label{cor:rings_of_integers1}
Let $F$ be a number field, and $S$ a finite set of finite places of $F$.
Let $\cO_{F,S}$ be the ring of $S$-integers in $F$.
Then the number of elements in the set
\[
\{ \mathcal{Y} \mid  
\text{\rm $\mathcal{Y}$ is a $V_5$-scheme over $\cO_{F,S}$}
\}
/ \simeq_{\cO_{F,S}}
\]
is finite.
\end{prop}

\subsection{Organization of the paper}

This paper is organized as follows.

In Section~\ref{section:preliminaries}, we summarize basic properties of $V_5$-schemes.
We show $V_5$-schemes are relative linear sections of six-dimensional Grassmann bundles over $B$ (=Definition-Proposition~\ref{prop:relativeW5}).
Along the course of the proof, we also sketch the proof of this fact, as well as the uniqueness of a $V_5$-variety, over a fixed algebraically closed field, whose proof is missing in the literature.
In Section~\ref{section:correspondence},  we will prove the classification of $V_5$-schemes by nets of alternating forms or symmetric bilinear forms.
In Section~\ref{section:automorphisms}, we provide the description of the automorphism group schemes, and describe its action on $Y$ explicitly when $Y$ is split.
In Section~\ref{section:hilbert_scheme_of_lines}, we describe the Hilbert scheme $\Sigma(Y)$ of lines as the $\P^2$-bundle on $B$.
Then we describe the action of the automorphism group scheme on $\Sigma(Y)$ or the universal family of lines.
In Section~\ref{section:orbit}, we describe the orbit decomposition of $Y$, and discuss the geometry of lines and orbits in more detail, which implies that $Y$ is the variety of trisecant lines to the projected Veronese surface.
In Section~\ref{section:arithmetic}, we study several arithmetic results for $V_5$-varieties.
Finally, in Section~\ref{section:characteristic_two},
we briefly discuss some geometric consequences of the non-reduced automorphism group in characteristic two.
We will describe the automorphism group and its action on $Y$ (=Proposition~\ref{prop:non-reduced_action}).
Then, we will construct a specific non-reduced subgroup in the automorphism group (=Proposition~\ref{prop:non-reduced_subgroup}), and describe the quotient $Y$ by it (=Remark~\ref{rem:quotient_V10}).
Further, we show that $H^1(T_Y)$ is two-dimensional, and describe the corresponding two-dimensional deformations of $Y$ in Subsection \ref{subsection:deformation_p=2}.

\subsection{Further application to prime Fano threefolds of genus $12$}
$V_5$-varieties are related to other class of Fano threefolds called the \emph{prime Fano threefolds of genus $12$}.
In our forthcoming paper (\cite{V22}), as an application of the results in the present article, we will study prime Fano threefolds of genus $12$ in positive and mixed characteristic.
For instance, a complete classification of prime Fano threefolds of genus $12$ with infinite automorphism groups will be given there.
See \cite{V22} for details.

\subsection*{Acknowledgments}
The authors, especially the second author, wish to express their gratitude to Professor Shigeru Mukai for various insightful discussions, comments, and suggestions.
We are also grateful to Professors Adrien Dubouloz, Kento Fujita, Takashi Kishimoto, Tatsuro Kawakami, Hiromu Tanaka, Yuji Odaka, and Laurent Manivel for helpful discussions.
Along this project, we are aided by computer algebra systems like \texttt{Macaulay2} and \texttt{SageMath} for several computations and making predictions (e.g.,\ Claim~\ref{claim:discriminant_square}).

\subsection*{Notations and Convention}
\begin{itemize}
\item A \emph{variety} $X$ over a field $k$ is a geometrically integral, separated scheme $X$ of finite type over $k$.
A \emph{family of projective varieties} $X/B$ is a flat projective morphism $X \to B$ of finite type whose fibers are varieties.

\item For a family of projective varieties $X/B$, $\Pic_{X/B}$ denotes the Picard functor, and $\Aut_{X/B}$ is the relative automorphism group scheme (when it is representable).

\item For a locally free sheaf $E$ on a scheme $X$, we denote by $\P(E)=\Proj(\Sym E)$ the Grothendieck projectivization of $E$, which parametrizes rank one locally free quotients of $E$. Similarly, $\Gr(E,m)$ denotes the \emph{Grassmann scheme} that parametrizes rank $m$ locally free quotient of $E$, and $\Fl(E;a_1, \dots, a_k)$ the \emph{flag scheme} parameterizing the flags of locally free quotients of rank $a_1$, \dots, $a_k$.

\item For a locally free sheaf $E$ on a scheme $X$, $c_1(E)$ denotes the class of $\det E \in \Pic(X)$ (but with additive notation, e.g., $c_1(L_1 \otimes L_2)= c_1(L_1) + c_1(L_2)$ for line bundles $L_i$).

\item For a number field $F$ or a $p$-adic field $K$, we denote by $\cO_F$ or $\cO_K$ its ring of integers.
For a finite subset $S$ of finite places of $F$, $\cO_{F,S}$ is the ring of $S$-integers of $F$.
\item For a number field $F$ and its finite place $\p$, we denote by $F_{\p}$ the $\p$-adic completion of $F$, and $\cO_{F,\p}$ its ring of integers.
$\cO_{F,(\p)}$ denotes the localization of $\cO_F$ at $\p$.

\item 
In this paper, we only consider \emph{left} group actions.
Thus, if a group $G$ acts on a variety $X$ and $F$ is a contravariant functor, then the induced $G$-action on $F(X)$ is given by $F(g^{-1})$.
In particular, $G$ acts on the Hilbert scheme of $X$ by the left action given by $[Z] \to [gZ]$. 
\end{itemize}
Note that a version of the cohomology and base change theorem holds over an arbitrary base scheme, which we will use freely in the remaining part of the paper (see, e.g., \cite[Proposition 4.37]{FGA}).

\section{Preliminaries on $V_5$-varieties}\label{section:preliminaries}

In this section, we provide several fundamental definitions and properties of $V_5$-schemes.
In particular, we will describe $V_5$-schemes as relative linear sections of $\Gr(5,3)$-bundles (see \cite[Section 4.2]{Kuz25} for the case of characteristic zero).

Let $k$ be a field, and $\overline{k}$ an algebraic closure of $k$.
A \emph{Fano variety} $X$ over $k$ is,
by definition, a smooth projective variety over $k$ with ample anti-canonical divisor $-K_X$.
We put $X_{\overline{k}} \coloneqq X \otimes_k \overline{k}$.
The Picard rank of $X_{\overline{k}}$ is denoted by 
$\rho(X_{\overline{k}}) \coloneqq \rank \Pic(X_{\overline{k}})$,
and the index of $X_{\overline{k}}$ is defined by
\[
i(X_{\overline{k}}) \coloneqq \max\{r \in \Z \mid \text{$-K_X$ is divisible by $r$ in $\Pic(X_{\overline{k}})$}\}.
\]

\begin{defn}\label{defn:W5}
\begin{enumerate}
 \item A smooth projective variety $Y$ over $k$ is called a \emph{$V_5$-variety}, or a \emph{quintic del Pezzo threefold},
 if it is a Fano variety with $\rho(Y_{\overline k})=1$, $i(Y_{\overline k})=2$ and $(-K_Y/2)^3=5$.
 \item Let $B$ ba a scheme.
 A $B$-scheme $\pi \colon Y\to B$ is called a \emph{$V_5$-scheme} if $\pi$ is a smooth projective morphism and, for every point $b \in B$, the fiber $X_b$ over $b$ is a $V_5$-variety.
\end{enumerate}
\end{defn}

We use the following description of $V_5$-varieties over algebraically closed fields:
\begin{thm}[\cite{Fuj90,Meg98}]\label{theorem:W5_over_closed_field}
Let $k$ be an algebraically closed field, $\Gr(5,3)$ the Grassmann variety parametrizing three-dimensional quotients of a fixed five-dimensional space over $k$,
and $\Gr(5,3) \subset \P^9$ its Pl\"ucker embedding.
Then the smooth codimension $3$ linear section of $\Gr(5,3) \subset \P^9$ is a $V_5$-variety.

Conversely, every $V_5$-variety $Y$ over $k$ is obtained in this way.
Moreover, we have the following:
\begin{enumerate}
 \item\label{theorem:W5_over_closed_field1} The ample generator $H \in \Pic(Y)$ defines an embedding $Y \to \P^6$, which is the restriction of the above Pl\"ucker embedding.
 \item\label{theorem:W5_over_closed_field2} $h^0(H) =7$ and $h^i(H)=0$ for $i >0$.
 \item\label{theorem:W5_over_closed_field3} The isomorphism class of $Y$ is unique.
\end{enumerate}
\end{thm}

\begin{rem}[Proof of the above theorem]\label{rem:W5_p=2}
Since there is no detailed proof of the above theorem (especially in characteristic $2$, cf.\ \cite[(8.19)]{Fuj90}), we sketch a proof of it.

The first part of the theorem is the same as the case of characteristic zero.
We sketch the proof of the remaining parts.
Take a $V_5$-variety $Y$ over $k$.
By the Kodaira vanishing 
(cf.\ \cite[Remark 2.8 (1)]{Kawakami-Tanaka-weak}),
we have \ref{theorem:W5_over_closed_field2}.
Then, as is detailed in \cite{Meg98}, $Y$ is a quintic del Pezzo $3$-fold in the sense of \cite{Fuj90}.
Thus $H$ embedds $Y \subset \P^6$ and $Y$ is defined by $5$ quadrics.

Take a line $l \subset Y$ (a line exists since it does on a hyperplane section).
Then the blow-up $\widetilde Y$ of $Y$ along $l$ is a Fano variety of Picard rank $2$, whose second extremal contraction $\widetilde Y \to Q$ satisfies
\begin{enumerate}
 \item $Q \subset \P^4$ is a smooth three-dimensional quadric,
 \item the contraction $\widetilde Y \to Q$ is the blow-up of $Q$ along a twisted cubic $C$ contained in $Q$,
 \item The linear span $\langle C\rangle $ of $C$ cuts out a hyperplane section $H$ of $Q$, which is the strict transform of the exceptional divisor of $\widetilde Y \to Y$.
\end{enumerate}
Conversely, a twisted cubic contained in a smooth hyperquadric determines a $V_5$-variety, and hence it is enough to classify twisted cubics $C$ in a smooth hyperquadric $Q$.
Let $H$ be the hyperplane section $\langle C\rangle \cap Q$. 
We can prove that there are three types of twisted cubics in a smooth hyperquadric:
\begin{itemize}
 \item \textbf{(Ordinary twisted cubic)} $H$ is the smooth quadric surface $\P^1\times \P^1$, and $C$ is the graph of a separable double cover $\P^1 \to \P^1$.
 \item \textbf{(Exceptional twisted cubic, $p=2$ only)} $H$ is the smooth quadric surface $\P^1\times \P^1$, and $C$ is the graph of the Frobenius morphism $\P^1 \to \P^1$.
 \item \textbf{(Special twisted cubic)} $H$ is the quadric cone, and $C$ is a twisted cubic passing through the vertex.
\end{itemize}

In the latter part of this paper (Corollary \ref{cor:lines_split_W5} and Proposition~\ref{prop:line_stab_action}), we will see that, for a specific choice $Y_{\spl}$ of a smooth codimension $3$ linear section of $\Gr(5,3) \subset \P^9$, there are three types of lines $l_1$, $l_2$ and $l_3$.
We can check that each pair $(Q,C)$ as above corresponds to one of the three types of lines $l_1$, $l_2$, and $l_3$.
Thus $Y$ is isomorphic to $Y_{\spl}$.
\end{rem}

\begin{lem}
\label{lem:relativeW5}
Let $B$ be a scheme, and $ \pi \colon Y \rightarrow B$ a $V_5$-scheme over $B$.
Then the following hold:
\begin{enumerate}
 \item The relatively ample generator $H' \in \Pic_{Y/B} (B)$ is a line bundle and relatively very ample. 
 \item $H'$ defines an embedding 
\[
Y \hookrightarrow \P_{B} (\pi_{\ast}H')
\]
into a six-dimensional projective bundle $f \colon \P_{B} (\pi_{\ast}H') \to B$.
\end{enumerate}
 \end{lem}
 
 \begin{proof}
Let $\beta \in \Br (B)$ be the element associated with $H'$.
Note that we have
\[
H'^{\otimes 2} \simeq \omega_{Y/B}^{-1} \in \Pic_{Y/B}(B),
\]
where $\omega_{Y/B}^{-1}$ comes from the element in $\Pic (Y)$.
Therefore, we obtain $\beta^2 =1$.
For any point $s$ on $B$, we have $H^i (Y_s,H'_s) =0$ for $i>0$, and $H^0 (Y_s, H'_s) =7$.
This implies $f_\ast H'$ is a $\beta$-twisted locally free sheaf of rank $7$, and we obtain $\beta^7=1$.
This means $\beta =1$, and $H'$ is an actual line bundle.
By the cohomology and base change theorem, $H'$ is relatively very ample, and we obtain the desired embedding.
\end{proof}

\begin{defn}
The above line bundle $H'$ is uniquely determined modulo $\Pic(B)$.
If we replace $H'$ with $H' \otimes \pi^*M$ for $M \in \Pic(B)$, then the embedding $Y \hookrightarrow \P_{B} (\pi_{\ast}H')$ is identified with $Y \hookrightarrow \P_{B} (\pi_{\ast}(H' \otimes M))$ via the canonical isomorphism $\P_{B} (\pi_{\ast}(H' \otimes \pi^*M)) \simeq \P_{B} (\pi_{\ast}H' \otimes M) \simeq \P_{B} (\pi_{\ast}H')$.
We denote by $i \colon Y \to \P_B$ this canonical inclusion.
\end{defn}

Let $\Gr(5,3)$ be the six-dimensional Grassmann variety and $\Gr(5,3) \subset \P^9$  the Pl\"ucker embedding.
It is known that the twisted conormal bundle $\cC_{\Gr(5,3)/\P^9}(2)$ is the universal quotient bundle of this Grassmann variety.
Thus, for a $V_5$-variety $Y \subset \P^6$, which is a smooth linear section of $\Gr(5,3)$, the twisted conormal bundle $\cC_{Y/\P^6}(2)$ is the restriction of the universal quotient bundle.
In fact, since $Y$ is defined by $5$ quadrics in $\P^6$, $H^0(I_{Y/\P^6}(2))$ is a $5$-dimensional vector space, and the surjective evaluation map $H^0(I_{Y/\P^6}(2)) \otimes \cO_Y \to \cC_{Y/\P^6} (2)$ determines the embedding $Y \subset \Gr(5,3)$ (cf. \cite[Proposition 0.2]{Mukaicurve} and \cite[\S 4.4 and 4.5]{Bhargava-Ho} for the case of quintic elliptic curves).
The following generalizes this fact to the relative settings:
\begin{defn-prop}
\label{prop:relativeW5}
Let $B$ be a scheme, $ \pi \colon Y \rightarrow B$ a $V_5$-scheme over $B$, and $i \colon Y \to \P_B$ the relative embedding to the $\P^6$-bundle.
Let $\cC_i$ be the conormal bundle of the embedding $i$.

Set $Q_Y \coloneqq \cC_i(-K_{\pi})$ and $\cO(H) \coloneqq \det(\cC_i(-K_{\pi}))$, which gives a relative ample generator class of $\Pic_{Y/B}(B)$.
Note that these are canonically determined by the $V_5$-scheme $Y$.
Further we set $A_7 \coloneqq \pi_*(\cO(H))$ and $U_5 \coloneqq \pi_*Q_Y$.
Then the following hold:
\begin{enumerate}
\item \label{prop:relativeW51} $A_7$ is a locally free sheaf of rank $7$.
\item \label{prop:relativeW52} $U_5$ is a locally free sheaf of rank $5$.
\item \label{prop:relativeW53} The natural map $\pi^*U_5 \to Q_Y$ is surjective.
\item \label{prop:relativeW54}The above map induces a surjective map $\pi^* \wedge ^3 U_5 \to \cO(H)$.
\item \label{prop:relativeW55} The direct image of the map in \ref{prop:relativeW54} induces a surjective map $\wedge^3  U_5\to A_7$.
\item \label{prop:relativeW56} From the above maps, we have the following Cartesian diagram:
\[
 \xymatrix{
 Y \ar[r] \ar[d] & \P_B = \P_B(A_7) \ar[d] \\
 \Gr(U_5,3) \ar[r] &  \P_B(\wedge^3 U_5).
 }
\]
\end{enumerate}
By taking the dual and a twist, we have the following exact sequence with a locally free sheaf $N_3$ of rank $3$:
\[
0 \to A_7^\vee(c_1(U_5)) \to \wedge^2 U_5 \to N_3 \to 0.
\]
We call the map $\nu \colon \wedge^2 U_5 \to N_3$ \emph{the associated net of alternating forms of $Y$}.
 \end{defn-prop}

\begin{proof}
 We have already seen \ref{prop:relativeW51}.
 
For \ref{prop:relativeW52}, it is enough to check $H^i (Q_Y) =0$ ($i>0$) for a $V_5$-variety $Y$ over an algebraically closed field.
Take a $V_5$-variety $Y$ over an algebraically closed field, which is a smooth codimension $3$ linear section of $\Gr(5,3)$.
Then, by the Koszul complex $\wedge^\bullet \cO(-1)^3 \to \cO_{\Gr(5,3) } \to \cO_Y \to 0$, we are reduced to proving
\[
\begin{cases}
 H^i(Q(-m)) = 0  & (m=1,2,3), \\
 H^i(Q) = 0 & (i >0),
\end{cases}
\]
where $Q$ is the universal quotient bundle on $\Gr(5,3)$.
Consider the projectivization $\P(Q)$.
Then this variety $\P(Q)$ is the flag variety $\Fl(5;3,1)$.
By the Kodaira vanishing theorem for flag varieties (see, e.g.,  \cite[Theorem 3.1.1]{Brion-Kumar}), we have  $H^i(Q(-m)) = 0$ for $m \leq 3$ and $i>0$.
Also $H^0(Q(-m)) = 0$ for $m>0$.
Thus, the assertion follows.

By the cohomology and base change theorem, the remaining assertions follow from the case where $B$ is a spectrum of an algebraically closed field.
\end{proof}

\begin{lem}\label{lem:net_of_bivectors_from_V5}
Let the notations be as above.
Then we have:
\begin{enumerate}
 \item $c_1(A_7)  =4 c_1(U_5)$ and $c_1(N_3)=c_1(U_5)$.
 \item $\P_B(N_3) \cap \Gr(U_5,2) = \emptyset$ in $\P_B(\wedge^2 U_5)$.
\end{enumerate}
\end{lem}

\begin{proof}
 It is known that $\cC_{\Gr(U_5,3)/\P(\wedge ^3U_5)} \simeq Q(p^*c_1(U_5) -2 c_1(Q))$, where $p \colon \Gr(U_5,3) \to B$ is the natural projection.
Further we have
\[
K_\pi = -2H + \pi^*(c_1(A_7)-3c_1(U_5)).
\]
Thus
\[
Q_Y=\cC_i(-K_\pi) \simeq Q_Y(\pi^*(4c_1(U_5)-c_1(A_7))).
\]
This implies $3(4c_1(U_5)-c_1(A_7))=0$, since $\rank Q_Y =3$.
By taking the direct image by $\pi$, we also have
\[
U_5 \simeq U_5(4c_1(U_5)-c_1(A_7)).
\]
Thus $5(4c_1(U_5)-c_1(A_7))=0$, and hence $4c_1(U_5)-c_1(A_7)=0$.
By the definition of $N_3$ and the fact $4c_1(U_5)=c_1(A_7)$, we have $c_1(N_3) =c_1(U_5)$.

The latter assertion follows from the same argument to \cite[Lemma 3.3]{Kuznetsovspinor} (see also \cite[Definition-Proposition 3.11]{IKTT}).
\end{proof}

\section{$V_5$-schemes, nets of alternating forms, and symmetric bilinear forms}\label{section:correspondence}

\subsection{Statement of the classification theorem}

\begin{defn}\label{defn:forms}
 Let $B$ be a scheme.
\begin{enumerate}
 \item A \emph{net of quinary alternating forms of rank $4$} is a triple
 \[ (U_5,N_3,\nu), \]
 where
 \begin{itemize}
 \item $U_5$ is a locally free sheaf of rank $5$,
 \item $N_3$ is a locally free sheaf of rank $3$,
 \item $\nu \colon \wedge^2 U_5 \rightarrow N_3$ is a surjective map such that, for any open subset $U \subset B$ and any locally free quotient $N_3 \to M$ of rank $1$ on $U$, the corresponding map
 \[
 U_5 \to U_5^\vee \otimes M
 \]
 has rank $4$, i.e.,\ the cokernel of this map is locally free of rank $1$.
\end{itemize}
 \item A \emph{non-degenerate ternary symmetric bilinear form} is a triple
 \[ (N_3,L,\varphi), \]
 where
  \begin{itemize}
 \item $N_3$ is a locally free sheaf of rank $3$,
 \item $L$ is a line bundle,
 \item $\varphi \colon S^2 N_3 \rightarrow L$ is a surjective map whose corresponding map
 \[
 N_3 \to N_3^\vee \otimes L
 \]
 is an isomorphism.
\end{itemize}
 \item We say $(N_{3,1}, L_1, \varphi_1)$ and $(N_{3,2}, L_2, \varphi_2)$ are $B$-\emph{similar} (denoted as $(N_{3,1}, L_1, \varphi_1) \sim_B (N_{3,2}, L_2, \varphi_2)$) when there exist an invertible sheaf $M$ on $B$, isomorphisms $f \colon N_{3,1} \simeq N_{3,2} \otimes M$ and $g \colon L_1 \simeq L_2 \otimes M^{\otimes 2}$ such that
\[
g \circ \varphi_1 = \varphi_{2,M} \circ S^2f,
\]
where, $\varphi_{2,M}$ is the morphism
\[
S^2 (N_{3,2} \otimes M) \rightarrow L_2 \otimes M^{\otimes 2}
\]
induced by $\varphi_2$.
$B$-similarity for nets of alternating forms is defined similarly. 
\end{enumerate}
\end{defn}

\begin{rem}\label{remark:maximal_rank}
Let $U_5$ and $N_3$ be locally free sheaves of rank $5$ and $3$ respectively, and $\nu \colon \wedge^2 U_5 \rightarrow N_3$ a surjective map.

Note that, on the projectivization $g \colon \P_B(N_3) \to B$, we have the universal rank  $1$ locally free quotient $g^*N_3 \to \cO(\xi)$.
Thus, the composite $g^*(\wedge^2 U_5) \to g^* N_3 \to \cO(\xi)$ represents the alternating forms belonging to the net.

Note also that the projectivization $\P_B(\wedge^2 U_5)$ parametrizes (non-zero) alternating forms on $U_5$.
Since ranks of alternating forms are even, a general point corresponds to an alternating form of rank $4$.
The Grassmann variety $\Gr(U_5,2) \subset \P_B(\wedge^2 U_5)$ parametrizes alternating forms of rank $2$. 
Therefore, the following conditions are equivalent:
\begin{enumerate}
 \item $\nu \colon \wedge^2 U_5 \rightarrow N_3$ is a net of alternating form of rank 4.
 \item For any point $b \in B$, the corresponding map $\nu_b \colon \wedge^2 U_{5,b} \rightarrow N_{3,b}$ is of rank 4 as a net of alternating forms over the residue field $k(b)$.
 \item $\P_B(N_3) \cap \Gr(U_5,2) = \emptyset$ in $\P_B(\wedge^2 U_5)$.
 \item On $\P(N_3)$, the induced alternating form $g^*(\wedge^2U_5) \to \cO(\xi)$ is of rank $4$, i.e.,\ the cokernel of the induced map $g^* U_5 \to g^*U_5^\vee \otimes \cO(\xi) $ is locally free of rank $1$.
\end{enumerate}
\end{rem}

The following theorem is one of the main theorems of this paper.

\begin{thm}[Classification of $V_5$-schemes]
\label{thm:W5description}
Let $B$ be a scheme.
Then we have natural bijections between the following sets:

\begin{itemize}
\item \textbf{$V_5$-schemes:}
\[
\sV\coloneqq \big\{ \,
f \colon X \rightarrow B \ \big| \ \text{\rm $X$ is a $V_5$-scheme over $B$}
\, \big\}/ \simeq_B.
\]

\item 
\textbf{Normalized nets of quinary alternating forms:}
\[
\sA\coloneqq
\bigg\{
(U_5,N_3,\nu) \ \bigg| \ 
\begin{array}{l}
\text{\rm $(U_5,N_3,\nu)$ is a net of quinary alternating} \\
\text{\rm forms of rank $4$ with $c_1(U_5)=c_1(N_3) =0$} \\
\end{array}
\bigg\}/\sim_B.
\]

 \item 
 \textbf{Nets of quinary alternating forms:}
\[
\sA'\coloneqq
\bigg\{
(U_5,N_3,\nu) \ \bigg| \ 
\begin{array}{l}
\text{\rm $(U_5,N_3,\nu)$ is a net of quinary alternating} \\
\text{\rm forms of rank $4$ over $B$}
\end{array}
\bigg\}/\sim_B.
\]

\item 
\textbf{Normalized ternary symmetric bilinear forms:}
\[
\sS \coloneqq 
\bigg\{
(N_3,L,\varphi) \ \bigg| \  
\begin{array}{l}
\text{\rm $(N_3,L,\varphi)$ is a non-degenerate ternary symmetric} \\
\text{\rm bilinear form with $c_1(N_3)=c_1(L)=0$} \\
\end{array}
\bigg\}/\sim_B.
\]

\item 
\textbf{Ternary symmetric bilinear forms:}
\[
\sS' \coloneqq 
\bigg\{
(N_3,L,\varphi) \ \bigg| \ 
\begin{array}{l}
\text{\rm $(N_3,L,\varphi)$ is a non-degenerate ternary} \\
\text{\rm symmetric bilinear form}
\end{array}
\bigg\}/\sim_B.
\]
\end{itemize}
\end{thm}

\begin{rem}
When $(N_{3,1}, L_1, \varphi_1)$ and $(N_{3,2}, L_2, \varphi_2)$ are in $\sS$, then these are $B$-similar if and only if there exist isomorphisms $f \colon N_{3,1} \simeq N_{3,2}$ and $g \colon L_1 \simeq L_2$ such that
\[
g \circ \varphi_1 = \varphi_{2} \circ S^2f.
\]
The same also holds for nets of alternating forms. 
\end{rem}

\begin{rem}
\label{rem:stacky}
By taking projectivization, we have
\[\sS' \simeq
\overline{\sS'} \coloneqq 
\bigg\{
\P(L) \subset \P(S^2N_3)
\ \bigg| \ 
\begin{array}{l}
\text{\rm $(N_3,L,\varphi) $ is a non-degenerate} \\
\text{\rm ternary symmetric bilinear form}
\end{array}
\bigg\}/ \simeq_B.
\]
and
\[
\sA' \simeq \overline{\sA'}\coloneqq
\bigg\{
\P(N_3) \subset \P(\wedge^2 U_5) \ \bigg| \ 
\begin{array}{l}
\text{\rm $(U_5,N_3,\nu)$ is a net of quinary} \\
\text{\rm alternating forms of rank $4$}
\end{array}
\bigg\}/ \simeq_B.
\]
Here, we say $\P(L_1) \subset \P(S^2N_{3,1})$ and $\P(L_2) \subset \P(S^2N_{3,2})$ are $B$-\emph{isomorphic} when there exist isomorphisms $\P(N_{3,1}) \simeq \P(N_{3,2})$ and $\P(L_1) \simeq \P(L_2)$ which induces an isomorphism $[\P(L_1) \subset \P(S^2N_{3,1})] \simeq [\P(L_2) \subset \P(S^2N_{3,2})]$.
$B$-isomorphisms for nets of alternating forms are defined similarly.

The proof below indeed shows that $\sV$, $\overline {\sA'}$, and $\overline{\sS'}$ are isomorphic to each other as moduli stacks. 
\end{rem}

\subsection{Immediate consequences of the classification theorem}

As corollaries of Theorem \ref{thm:W5description}, we can classify $V_5$-schemes over $\Z$ and $V_5$-varieties over finite fields.
\begin{cor}
\label{cor:classification_over_z}
There are exactly two isomorphism classes of $V_5$-schemes over $\Z$.
\end{cor}

\begin{proof}
By \cite[Chapter 2, Theorem 2.2]{Milnor-Husemoller},
every non-degenerate symmetric bilinear form over $\Z$ is a direct sum of
copies of $\langle 1 \rangle$ and $\langle -1 \rangle$,
hence one of the following forms
\[
\langle 1 \rangle^{\oplus 3}, \qquad
\langle 1 \rangle^{\oplus 2} \oplus \langle -1 \rangle, \qquad
\langle 1 \rangle \oplus \langle -1     \rangle^{\oplus 2}, \qquad
\langle -1 \rangle^{\oplus 3}.
\]
Here, $\langle 1 \rangle^{\oplus 3}$, $\langle -1 \rangle^{\oplus 3}$
are $\Z$-similar, and
$\langle 1 \rangle^{\oplus 2} \oplus \langle -1 \rangle$,
$\langle 1 \rangle^{\oplus 2} \oplus \langle -1 \rangle$
are $\Z$-similar.
Moreover, $\langle 1 \rangle^{\oplus 3}$ is not $\Z$-similar to
$\langle 1 \rangle^{\oplus 2} \oplus \langle -1 \rangle$
because the former is (positive) definite, but the latter is indefinite.
The assertion follows from Theorem \ref{thm:W5description}.
\end{proof}

\begin{rem}
Corollary \ref{cor:classification_over_z}
is in contrast to the fact that
a smooth conic over $\Z$ is unique.
Symmetric bilinear forms and quadratic forms behave differently over $\Z$.
The above proof shows there exists two $\Z$-similarity classes of non-degenerate ternary symmetric bilinear forms over $\Z$.
On the other hand, there exists a unique $\Z$-similarity class of ternary quadratic forms over $\Z$ defining smooth conics,
i.e.,\ $q(x,y,z) = x^2 + yz$.
\end{rem}

\begin{prop}
\label{prop:char2}
Let $k$ be a perfect field of characteristic $2$.
Then, there is a unique $k$-isomorphism class of $V_5$-varieties over $k$.
In particular, every $V_5$-variety over $k$ is isomorphic to the split $V_5$-variety $Y_{\spl,k}$ defined later in Definition-Proposition \ref{defn-prop:splitform}.
\end{prop}

\begin{proof}
By Theorem \ref{thm:W5description},
it is enough to prove that, for every positive odd integer $n \geq 1$, there is a unique $k$-isomorphism class of non-degenerate $n$-dimensional symmetric bilinear forms
over $k$.
By \cite[p.62]{Milnor1971},
every non-degenerate $n$-dimensional  symmetric bilinear form over $k$ is isomorphic to
$\langle a_1 \rangle \oplus \cdots \oplus \langle a_r \rangle \oplus H^{\oplus s}$
with $r + 2s = n$ for some $a_1,\ldots,a_r \in k^{\times}$.
Here $H$ is a hyperbolic plane.
Since $k$ is perfect of characteristic $2$, every $a_i$ is a square in $k^{\times}$,
and it is isomorphic to
$\langle 1 \rangle^{\oplus r} \oplus H^{\oplus s}$ over $k$.
Since $n$ is odd, we have $r \geq 1$.
By the same calculation as in the proof of \cite[Chapter 2, Theorem 2.2]{Milnor-Husemoller},
we have an isomorphism
$\langle 1 \rangle^{\oplus r} \oplus H^{\oplus s} \simeq \langle 1 \rangle^{\oplus n}$ over $k$;
its $k$-isomorphism class is unique.
We note that
$\langle 1 \rangle = \langle -1 \rangle$
since the characteristic of $k$ is $2$ here.
\end{proof}

\begin{cor}
\label{cor:W5eqn}
Let $k$ be an algebraically closed field or a finite field.
Then, there is a unique $k$-isomorphism class of $V_5$-varieties over $k$.
In particular, every $V_5$-variety over $k$ is isomorphic to the split $V_5$-variety $Y_{\spl,k}$ defined later in Definition-Proposition \ref{defn-prop:splitform}.
\end{cor}

\begin{proof}
Every ternary quadratic form over $k$ is isotropic (see \cite[Lemma 12.3.5]{Voight21}).
Hence the assertion follows from Theorem \ref{thm:W5description} 
if $k$ is algebraically closed
or $k$ is a finite field of odd characteristic (see also Proposition \ref{prop:field}).
If $k$ is a finite field of even characteristic, it follows from Proposition \ref{prop:char2}.
\end{proof}

\subsection{The construction}
In the rest of this section, we will prove Theorem~\ref{thm:W5description}.
Along the course of the proof, we will describe one model $Y_{\spl, \Z}$ of $V_5$-schemes over $\Z$, which plays important roles not only in the proof of the theorem, but also in the rest of this paper.
In this subsection, we explain the constructions of the bijections among the above sets.
\begin{construction}\label{construction:W5}
\hfill
\begin{enumerate}
 \item \label{construction:W5_1}
($\sV \to \sA'$)
 For a $V_5$-scheme $Y/B \in \sV$, we have the associated net of quinary alternating forms as in Definition-Proposition~\ref{prop:relativeW5}.
 This net of alternating forms has rank $4$ by Lemma~\ref{lem:net_of_bivectors_from_V5} and Remark~\ref{remark:maximal_rank}.
 \item \label{construction:W5_2} ($\sA' \to \sV$)
 A net of quinary alternating forms $\wedge^2 U_5\to N_3$ defines a subvariety $Y \subset \Gr(U_5,3)$ as the relative linear section.
 The rank condition ensures that $Y$ is a $V_5$-scheme (See Lemma~\ref{lemma:correspondence_V5_net_of_bivectors}).
 \item \label{construction:W5_3}($\sA \simeq \sA'$) We will prove the natural inclusion map $\sA \to \sA'$ is an isomorphism (Lemma~\ref{lemma:A_A'}).
 \item \label{construction:W5_4} ($\sS \simeq \sS'$) The natural inclusion map $\sS \to \sS'$ is an isomorphism (Lemma~\ref{lemma:S_S'}).

 \item \label{construction:W5_5} ($\Phi' \colon \sA' \to \sS'$ and $\Phi \colon \sA \to \sS$)
 We have the following natural transformation between the functors $\wedge^4$ and $S^2 \wedge^2$:
 \[
 a\wedge b \wedge c \wedge d \mapsto (a \wedge b)(c \wedge d)-(a \wedge c)(b \wedge d)+(a \wedge d)(b \wedge c).
 \]
 By composing this transformation with the symmetric power of $\nu \colon \wedge^2 U_5 \to N_3$, we have the following composite of the two maps:
 \[
G_\nu \colon \wedge^4 U_5 \to S^2 (\wedge^2 U_5) \to S^2 N_3.
 \]
 We will prove that the cokernel of $G_\nu$ is a line bundle $L$ (with $c_1(L) = 4(c_1(N_3) -c_1(U_5))$).
 Moreover the symmetric bilinear form $\varphi_\nu \colon S^2N_3 \to L$ is non-degenerate (Corollary~\ref{cor:correspondence}).
 
 \item \label{construction:W5_6} ($\Psi' \colon \sS' \to \sA'$ and $ \Psi \colon \sS \to \sA$)
 Let $\varphi \colon S^2N_3 \to L$ be a non-degenerate symmetric bilinear form, and $K$ the kernel of $\varphi$.
We denote by $\mu_\varphi \colon N_3 \to N_3^\vee \otimes L$ the map induced from $\varphi$.

For two symmetric bilinear forms $p$, $r \in \Hom(S^2N_3, L) \simeq (S^2N_3)^\vee \otimes L$, we have the following composite of maps:
\[
\mu_p  \circ \mu_\varphi^{-1} \circ  \mu_r  \in \Hom(N_3  , N_3^\vee \otimes L) \simeq N_3^\vee \otimes N_3^\vee \otimes L.
\]
We denote by $F_\varphi(p,r) \in \wedge^2 (N_3^\vee) \otimes L$ the image of this element in $\wedge^2 (N_3^\vee) \otimes L$.
When $p=r$, we have $F_\varphi(p,p)=0$ since $p$ and $\varphi$ are symmetric.
Thus, we have the induced map
\[
F_\varphi \colon \wedge^2((S^2N_3)^\vee \otimes L) \to \wedge^2 (N_3^\vee) \otimes L.
\]
Moreover, if one of the forms $p$ or $r$ is isomorphic to $\varphi$, then $F_\varphi(p,r)=0$.
Thus, we have the induced map
\[
\nu_\varphi \colon \wedge^2 (K^\vee\otimes L) \to \wedge^2 (N_3^\vee) \otimes L,
\]
which gives a net of alternating forms $(K^\vee \otimes L, \wedge^2 (N_3^\vee) \otimes L, \nu_\varphi)$ (Corollary~\ref{cor:correspondence}).
\end{enumerate}
\end{construction}

\begin{rem}
It is not hard to prove that $\sA'$ is isomorphic to
 \[
 \bigg\{
(U_5,N_3,\nu) 
\ \bigg| \ 
\begin{array}{l}
\text{\rm $(U_5,N_3,\nu)$ is a net of quinary alternating forms of} \\
\text{\rm rank $4$ such that $c_1(U_5) = c_1(N_3)$ is $2$-torsion} \\
\end{array}
\bigg\}/ \sim_B.
\]
To deal with the $2$-torsion element $c_1(U_5) = c_1(N_3)$, we will check that an invariant of nets of quinary alternating forms is a complete square (Remark~\ref{remark:discriminant_square}).
\end{rem}

\subsection{Split $V_5$-schemes}

\begin{defn-prop}
\label{defn-prop:splitform}
Let $B$ be a scheme,  $U_5 =\cO_B^{\oplus 5}$ a free $\cO_B$-module of rank $5$ with standard basis $\{e_{1},e_{2},e_{3},e_{4},e_{5}\}$, and $N_3= \cO_B^{\oplus 3}$ a free $\cO_B$-module of rank $3$ with standard basis $\{\alpha,\beta,\gamma\}$.
\begin{enumerate}
 \item The net of alternating form  $\nu_{\spl} \colon \wedge^2 U_5 \to N_3$ that corresponds to
 \[
\begin{pmatrix}
 0&0&0&0&0\\
 0&0&0&0&-1\\
 0&0&0&1&0\\
 0&0&-1&0&0\\
 0&1&0&0&0\\
\end{pmatrix},
\begin{pmatrix}
 0&0&0&-1&0\\
 0&0&1&0&0\\
 0&-1&0&0&0\\
 1&0&0&0&0\\
 0&0&0&0&0\\
\end{pmatrix},
\begin{pmatrix}
 0&0&0&0&-1\\
 0&0&0&1&0\\
 0&0&0&0&0\\
 0&-1&0&0&0\\
 1&0&0&0&0\\
\end{pmatrix}
 \]
 is called \emph{split}.
In other words, $\nu_{\spl}$ sends $e_2 \wedge e_5 \mapsto -\alpha$, $e_3 \wedge e_4 \mapsto \alpha$, $e_1 \wedge e_4 \mapsto -\beta$, $e_2 \wedge e_3 \mapsto \beta$, $e_1 \wedge e_5 \mapsto -\gamma$, $e_2 \wedge e_4 \mapsto \gamma$, and $e_i\wedge e_j \mapsto 0$ ($i+j \neq 5$, $6$, $7$).
\item The symmetric bilinear form  $\varphi_{\spl} \colon S^2 N_3 \to \cO_B$
 that corresponds to the matrix
\begin{equation}
\label{eqn:splitmatrix}
\begin{pmatrix}
 0&-1&0\\
 -1&0&0\\
 0&0&1
\end{pmatrix}
\end{equation}
is called \emph{split}.
Thus $\varphi_{\spl}$ sends $\alpha\beta \mapsto -1$, $\gamma^2 \mapsto 1$ and other monomials to $0$.
 \item The above two forms correspond to each other by Construction~\ref{construction:W5}, and define a $V_5$-scheme $Y_{\spl}$ over $\Spec R$.
The embedding $Y_{\spl} \subset \P^6_B$ is defined by
\begin{equation}
\label{eqn:W5}
\begin{cases}
a_0 a_4 - a_1 a_3 + a_2^2,\\
a_0 a_5 - a_1 a_4 + a_2 a_3,\\
a_0 a_6 - a_2 a_4 + a_3^2,\\
a_1 a_6 - a_2 a_5 + a_3 a_4,\\
a_2 a_6 - a_3 a_5 + a_4^2,
\end{cases}
\end{equation}
in $\P^6_B$ with coordinate $a_0, a_1, \dots, a_6$.
\end{enumerate}
In particular, if $k$ is an algebraically closed field, then a $V_5$-variety over $k$ is isomorphic to $Y_{\spl}$.
\end{defn-prop}
\begin{proof}
By the definition, we see  $\nu_{\spl} \in \sA$ and $\varphi_{\spl} \in \sS$.
The assertions follow from a straightforward calculation as follows:
The map $G_{\nu_{\spl}}$ sends
\begin{align*}
& 2 \wedge 3 \wedge 4 \wedge 5 \mapsto -\alpha^2,\\
&1 \wedge  3 \wedge 4 \wedge 5 \mapsto -\gamma\alpha, \\
&1 \wedge 2 \wedge 4 \wedge 5 \mapsto -\alpha\beta-\gamma^2,\\
&1 \wedge 2 \wedge 3 \wedge  5 \mapsto -\beta\gamma,\\
&1 \wedge 2 \wedge 3 \wedge 4  \mapsto -\beta^2,
\end{align*}
and hence the cokernel corresponds to $\varphi_{\spl}$.
Conversely, we identify the cokernel of $\varphi_{\spl}^\vee$ with $U_5$ by
\begin{equation}\label{eq:cokernel_q_dual}
\begin{aligned}
& (\alpha^2)^\vee \mapsto -e_1,\\
& (\gamma\alpha)^\vee \mapsto e_2,\\
& (\gamma^2)^\vee \mapsto -e_3,\\
& (\alpha\beta)^\vee \mapsto -e_3,\\
& (\beta\gamma)^\vee \mapsto e_4,\\
& (\beta^2)^\vee \mapsto -e_5.
\end{aligned}
\end{equation}
Then the map $\nu_{\varphi_{\spl}}$ followed by the isomorphism $\wedge^2 N_3^{\vee} \simeq N_3$ is $\nu_{\spl}$ (see Claim~\ref{claim:Theta}).

The last assertion follows from the fact that $\Gr(U_5,3) = \Gr(U_5^\vee,2)$ is defined by $5$ pfaffians
\begin{equation}
\label{eqn:Gr}
\begin{cases}
 b_{12} b_{34} - b_{13} b_{24} + b_{14}b_{23}, \\
 b_{12} b_{35} - b_{13} b_{25} + b_{15}b_{34}, \\
 b_{12} b_{45} - b_{14} b_{25} + b_{15}b_{24}, \\
 b_{13} b_{45} - b_{14} b_{35} + b_{15}b_{34}, \\
 b_{23} b_{45} - b_{24} b_{35} + b_{25}b_{34},
\end{cases}
\end{equation}
and that $\nu_{\spl}$ defines three equations $b_{14}=b_{23}$, $b_{15}=b_{24}$, $b_{25}=b_{34}$
(Set $a_{i+j-3} \coloneqq \overline{b_{ij}}$).
Here $\Gr(U_5,3) = \Gr(U_5^\vee,2)$ is embedded into $\P(\wedge^3U_5) =\P(\wedge^2 U_5^\vee)$ with the coordinates given by (the dual of) the basis $\{e_i \wedge e_j\}$ of $\wedge^2 U_5$.
\end{proof}

\begin{rem}
\label{rem:coordinatechange}
If we denote the coordinates in \cite[p.\ 505]{Mukai-Umemura} (written as $a_0, \ldots, a_6$) by $b_0, \ldots, b_6$, then the relation to our coordinates is given by the following:
\[
\trans{(a_0, \ldots, a_6)} =  \diag (8,24,12,8,6,6,1) \trans{(b_0, \ldots, b_6)}.
\]
\end{rem}

\begin{rem}
\label{remark:splitW5}
As we saw in Corollary \ref{cor:W5eqn}, there is a $V_5$-scheme $Y'_{\Z}$ over $\Z$ which is not split, and corresponds to the symmetric bilinear form $\langle 1 \rangle^{\oplus 3}$.
The reason why we say $Y_{\spl}$ is split is that the symmetric bilinear form $\varphi_{\spl}$ corresponds to a split conic $V(z^2-2xy) \subset \P^1_{\Q}$ on the generic fiber $\Q$.
This fact is reflected in the fact that ``the locus of special lines on $Y_{\Q}$'' is a split conic over $\Q$ (see Theorem \ref{thm:linep>2}).
\end{rem}

As a corollary, we have the following:
\begin{cor}\label{cor:correspondence}
Let $\nu \colon \wedge^2 U_5 \to N_3$ be a net of quinary alternating forms of rank $4$, and $G_\nu$ the map as in Construction~\ref{construction:W5}~\ref{construction:W5_5}.
Then the cokernel of $G_{\nu}$ is a line bundle, and the induced symmetric bilinear form $\varphi_{\nu}$ is non-degenerate.

Conversely, if $\varphi \colon S^2N_3 \to L$ is a non-degenerate symmetric bilinear form, then the map $\nu_{\varphi}$ as in Construction~\ref{construction:W5}~~\ref{construction:W5_6} defines a net of alternating forms of rank $4$.
\end{cor}

\begin{proof}
By taking geometric fibers, we may reduce to the case where $B$ is the spectrum of an algebraically closed field.
Then the assertions follow from Definition-Proposition~\ref{defn-prop:splitform}. 
\end{proof}

\subsection{Proof of Theorem~\ref{thm:W5description}}

\begin{lem}\label{lemma:correspondence_V5_net_of_bivectors}
Constructions~\ref{construction:W5}~\ref{construction:W5_1} and \ref{construction:W5_2} are mutually inverse to each other.
\end{lem}

\begin{proof}
This follows from the constructions.
Note that, by Lemma~\ref{lem:net_of_bivectors_from_V5} and Remark~\ref{remark:maximal_rank}, the associated net of alternating forms for a $V_5$-scheme has rank $4$.
Conversely, a net of quinary alternating forms of rank $4$ defines a $V_5$-scheme by the same argument to \cite[Lemma 3.3]{Kuznetsovspinor} and \cite[Definition-Proposition 3.11]{IKTT}.
\end{proof}

\begin{lem}\label{lemma:S_S'}
For an element $(N_3,L,\varphi) \in \sS'$, we have $3 c_1(L) = 2 c_1(N_3)$.
Moreover $\sS \simeq \sS'$.
\end{lem}

\begin{proof}
Since $\varphi$ is non-degenerate, the map $\mu_\varphi \colon N_3 \to N_3^\vee \otimes L$ is an isomorphism.
By taking the determinant, we have an isomorphism $\det \mu_\varphi \colon \det N_3 \to \det (N_3^\vee \otimes L)$, which implies $3 c_1(L) = 2 c_1(N_3)$.
Then $\varphi_{L\otimes \det N_3 ^\vee}$ gives an element in $\sS$, which is similar to $\varphi$. 
\end{proof}

\begin{lem}\label{lemma:symmetric_form}
 $\Phi' \circ \Psi' = \id $ and  $\Phi \circ \Psi = \id$.
\end{lem}

\begin{proof}
 Let $(N_3,L,\varphi) \in \sS'$ be a non-degenerate symmetric bilinear form.
 Then we have a net of alternating forms $\nu_\varphi \colon \wedge^2 (K^\vee \otimes L) \to \wedge^2 (N_3^\vee) \otimes L$ as in Construction~\ref{construction:W5}.
 By Lemma~\ref{lemma:S_S'}, this is an element in $\sA$.
 By applying Construction~\ref{construction:W5}, we have $\varphi_{\nu_\varphi} \colon S^2( \wedge^2 (N_3^\vee) \otimes L) \to \cO$, which is in $\sS$.
 We prove that $\varphi_{L\otimes \det N_3 ^\vee}$ is isomorphic to $\varphi_{\nu_\varphi}$.
 More precisely, we will show that the following is a commutative diagram (up to an automorphism of upright $\cO_B$):
\[
\xymatrix{
 S^2( \wedge^2 (N_3^\vee) \otimes L)\ar[d] \ar[r]^-{\varphi_{\nu_\varphi}} & \cO \ar[d]\\
S^2(N_3 \otimes L \otimes \det N_3^\vee) \ar[r]^-{\varphi_{L\otimes \det N_3^\vee}} & L \otimes   (L \otimes \det N_3^\vee)^{\otimes 2},
}
\]
where the left vertical arrow is given by the natural identification $N_3 \otimes L \otimes \det N_3^\vee \simeq  \wedge^2 (N_3^\vee) \otimes L$ and
the right vertical arrow is induced from $\det \mu_\varphi \colon \det N_3 \to \det (N_3^\vee \otimes L)$.

Since every map is naturally constructed,  we are reduced to proving that the above diagram commutes when $B = \Spec (R)$ for a commutative ring $R$, $N_3=\cO_B^{\oplus 3}$ and $L=\cO_B$.
This can be checked directly as follows:

Let $\alpha$, $\beta$, $\gamma$ be the standard basis of $N_3 = \cO_B^{\oplus 3}$.
Take a symmetric matrix $Q = (q_{ij})$ corresponding to the map $\mu_\varphi$, and write $Q^{-1} = (p_{ij})$.
Then the map
\[
F_\varphi \colon \wedge^2((S^2N_3)^\vee \otimes L) \to \wedge^2 (N_3^\vee) \otimes L
\]
sends the bivector $P \wedge R$ of symmetric matrices $P$ and $R$ to an alternating matrix $\Alt(P Q^{-1} R)\coloneqq P Q^{-1} R - \trans (P Q^{-1} R)$.

Then, by a similar construction to Construction~\ref{construction:W5}~\ref{construction:W5_5}, the map $F_\varphi$ induces a map
\[
\wedge^4 ((S^2N_3)^\vee \otimes L) \to S^2 (\wedge^2 (N_3^\vee) \otimes L)
\]
whose cokernel is isomorphic to $\varphi_{\nu_\varphi}$.
By the definition, this map sends the $4$-vector $P_1 \wedge P_2 \wedge P_3 \wedge P_4$ of symmetric matrices $P_i$ to
\begin{align*}
 &A(P_1 \wedge P_2 \wedge P_3 \wedge P_4) \\
 &=\Alt(P_1Q^{-1}P_2) \Alt(P_3Q^{-1}P_4)-\Alt(P_1Q^{-1}P_3) \Alt(P_2Q^{-1}P_4)+\Alt(P_1Q^{-1}P_4) \Alt(P_2Q^{-1}P_3)
\end{align*}
Here $\Alt(M)\Alt( N)$ is regarded as the image $\Alt(M)\otimes \Alt( N)$ by the map  $ (\wedge^2 (N_3^\vee) \otimes L)^{\otimes 2} \to S^2 (\wedge^2 (N_3^\vee) \otimes L)$.

Note that the map $S^2 (\wedge^2 (N_3^\vee) \otimes L) \to S^2(N_3 \otimes L \otimes \det N_3^\vee) \simeq S^2N_3 \to \cO$
sends $\Alt(M)\Alt( N)$ to
\begin{align*}
& \varphi((\Alt(M)_{23}\alpha + \Alt(M)_{31}\beta +\Alt(M)_{12}\gamma)(\Alt(N)_{23}\alpha + \Alt(N)_{31}\beta +\Alt(N)_{12}\gamma)) \\
=& \Alt(M)_{23}\Alt(N)_{23}q_{11} + \Alt(M)_{31}\Alt(N)_{23}q_{21} +\Alt(M)_{12}\Alt(N)_{23}q_{31} \\
+& \Alt(M)_{23}\Alt(N)_{31}q_{12}+ \Alt(M)_{31}\Alt(N)_{31}q_{22} +\Alt(M)_{12}\Alt(N)_{31}q_{32}\\
+& \Alt(M)_{23}\Alt(N)_{12}q_{13} + \Alt(M)_{31}\Alt(N)_{12}q_{23} +\Alt(M)_{12}\Alt(N)_{12}q_{33}.
\end{align*}
We can check directly that this map sends $A(P_1 \wedge P_2 \wedge P_3 \wedge P_4)$ to $0$, and the assertion follows.
\end{proof}

\begin{lem}\label{lemma:net_of_altaernating_forms}
 $\Psi' \circ \Phi' = \id $ and  $\Psi \circ \Phi = \id$.
\end{lem}
The proof of this lemma is the main technical part of the proof and is divided into several claims.
 Take a net of alternating forms $(U_5,N_3,\nu) \in \sA'$.
 Then we have a non-degenerate symmetric bilinear form $(N_3, L, \varphi_\nu)$ as in Construction~\ref{construction:W5}, where $L$ is the cokernel of the map $\wedge^4 U_5 \to S^2N_3$.

From $\varphi_\nu$, we have the net of alternating forms $\nu_{\varphi_\nu} \colon \wedge^2 ( (\wedge^4 U_5 ) ^\vee \otimes L) \to \wedge^2 (N_3^\vee) \otimes L$ as in Construction~\ref{construction:W5}.
We will show that $\nu_{\varphi_\nu}$ is isomorphic to $\nu_{\det U_5^\vee \otimes L}$.
More precisely, we will show that there is a map
\[
H_\nu \colon N_3 \otimes \det N_3^\vee \otimes L \to N_3 \otimes(\det U_5^\vee \otimes L)^{\otimes 2}
\]
that makes the following diagram commutative:
\[
\xymatrix{
\wedge^2 ( (\wedge^4 U_5 ) ^\vee \otimes L) \ar[d] \ar[r]^-{\nu_{\varphi_\nu}} & \wedge^2 (N_3^\vee) \otimes L \simeq N_3 \otimes \det N_3^\vee \otimes L\ar[d]^{H_\nu}\\
\wedge^2(U_5 \otimes \det U_5^\vee \otimes L) \ar[r]^-{\nu_{\det U_5^\vee \otimes L}} & N_3 \otimes(\det U_5^\vee \otimes L)^{\otimes 2},
}
\]
where the left vertical arrow is given by the natural identification $(\wedge^4 U_5 ) ^\vee \simeq U_5 \otimes \det U_5^\vee$.
Note that, since $G_\nu^\vee$ is surjective, it is enough to prove that this diagram commutes after composing with $ \wedge^2 (G_\nu^\vee \otimes \id_L)$, i.e.,\ the following commutes:
\begin{equation}\label{eq:net_of_bivectors}
  \xymatrix@C+0.6pc{
\wedge^2((S^2N_3)^\vee \otimes L) \ar[d]^-= \ar[r]^-{ \wedge^2 (G_\nu^\vee \otimes \id_L)} &\wedge^2 ( (\wedge^4 U_5 ) ^\vee \otimes L) \ar[r]^-{\nu_{\varphi_\nu}} & \wedge^2 (N_3^\vee) \otimes L \simeq N_3 \otimes \det N_3^\vee \otimes L\ar[d]^{H_\nu}\\
\wedge^2((S^2N_3)^\vee \otimes L) \ar[r]^-{ \wedge^2 (G_\nu^\vee \otimes \id_L)} &\wedge^2 ( (\wedge^4 U_5 ) ^\vee \otimes L )\simeq \wedge^2(U_5 \otimes \det U_5^\vee \otimes L) \ar[r]^-{\nu_{\det U_5^\vee \otimes L}} & N_3 \otimes(\det U_5^\vee \otimes L)^{\otimes 2}.
}
\end{equation}
Note also that, since the horizontal arrows are surjective, $H_\nu$ is unique if it exists.
Thus we are reduced to proving the assertion when $B=\Spec(R)$ for a commutative ring $R$, $U_5 = \cO_B^{\oplus 5}$ and  $N_3 = \cO_B^{\oplus 3}$.

Let $e_i$ be the standard basis of $U_5 = \cO_B^{\oplus 5}$, and $\alpha$, $\beta$, $\gamma$ that of $N_3$.
The net of alternating forms $\nu$ is given by three alternating matrix $A=(a_{ij})$, $B=(b_{ij})$ and $C=(c_{ij})$.
Define  $6 \times 5$ matrix $M=(m_{ij})$ whose $j$-th column is
\[
\begin{pmatrix}
 \Pf (A_j) \\
 \Pf(B_j + C_j) -\Pf(B_j) - \Pf(C_j)\\
  \Pf (B_j) \\
 \Pf(C_j + A_j) -\Pf(C_j) - \Pf(A_j)\\
  \Pf (C_j) \\
 \Pf(A_j + B_j) -\Pf(A_j) - \Pf(B_j)
\end{pmatrix}.
\]
Here $A_j$, $B_j$ and $C_j$ are principal submatrices of $A$, $B$ and $C$ respectively.
Further, we define a symmetric matrix $Q_\nu$ as
\[
Q_\nu \coloneqq
\begin{pmatrix}
 \det M_1 & - \det M_6& - \det M_4\\
 - \det M_6 &  \det M_3& -\det M_2\\
-\det M_4 & -\det M_2&  \det M_5\\
\end{pmatrix},
\]
where $M_i$ is a submatrix of $M$ obtained by removing $i$-th row.

In the following,  for a free module $V = R^{\oplus m}$ with standard basis $\{v_i\}$, we denote by $\omega_{V}$ the volume form with respect to this basis, and we take the basis $\wedge^2 \{v_i\} \coloneqq \{v_i \wedge v_j \}_{i <j}$ of $\wedge^2 V$ with the index set $\{ (i,j) \mid i<j \}$.

\begin{claim}\label{claim:Theta}
The matrix
$Q_{\nu}$ is invertible. Moreover, if we write $Q_{\nu}^{-1} = (p_{ij})$, then
the composite of the upper horizontal maps in \eqref{eq:net_of_bivectors}
 corresponds to
 \setcounter{MaxMatrixCols}{15}
\begin{align*}
& \Theta(p_{ij}) \coloneqq\\
&\begin{pmatrix}
0&0&0&0&0 &-p_{22}&p_{13}&p_{33}&-p_{12}&p_{12}&p_{23}&0&0&-p_{11}&-p_{13} \\
-p_{12}&0&-p_{11}&-p_{13}&0&0&p_{23}&0&p_{22}&0&0&0&-p_{33}&p_{12}&p_{23}\\
p_{13}&p_{12}&0&0&p_{11}&0&-p_{33}&0&-p_{23}&-p_{23}&0&-p_{22}&0&p_{13}&0
\end{pmatrix}
\end{align*}
with respect to the bases
 $\wedge^2 \{(\alpha^2)^\vee, (\beta \gamma)^\vee, (\beta^2)^\vee, (\gamma \alpha)^\vee, (\gamma^2)^\vee,(\alpha \beta)^\vee\}$ and $\{\alpha \otimes \omega_{N_3}^\vee,\beta \otimes \omega_{N_3}^\vee ,\gamma \otimes \omega_{N_3}^\vee \}$.
\end{claim}

\begin{proof}

By the definition, the map $G_\nu \colon \wedge^4 U_5 \to S^2N_3$ with respect to the bases
$\{ e_2\wedge e_3\wedge e_ 4\wedge e_5, e_1 \wedge e_3\wedge e_ 4\wedge e_5,e_1 \wedge e_2 \wedge e_ 4\wedge e_5,e_1 \wedge e_2\wedge e_3 \wedge e_5,e_1 \wedge e_2\wedge e_3\wedge e_ 4 \}$
and
$\{\alpha^2, \beta \gamma, \beta^2, \gamma \alpha, \gamma^2,\alpha \beta\}$
is given by the $6 \times 5$ matrix $M=(m_{ij})$.
Thus the dual map $ G_\nu ^\vee$ is given by $\trans M$, and its kernel is spanned by a vector $((-1)^{i+1}\det M_i)$.
(Here we identify $L=\cO$.)
This says that the symmetric bilinear form $\varphi_\nu$ corresponds to the symmetric matrix $Q_\nu$.

By the definition, $\nu_{\varphi_\nu} \circ \wedge^2 (G_\nu^\vee \otimes \id_L) = F_{\varphi_\nu}$.
Thus the map $\nu_{\varphi_\nu} \circ \wedge^2 (G_\nu^\vee \otimes \id_L) = F_{\varphi_\nu}$ sends a bivector $P \wedge R$ of symmetric matrices $P$ and $R$ to $\Alt(P Q_\nu^{-1} R)$, which is identified with
\[
(\Alt(P Q_\nu^{-1} R)_{23} \alpha +\Alt(P Q_\nu^{-1} R)_{31} \beta +\Alt(P Q_\nu^{-1} R)_{12} \gamma) \otimes (\omega_{N_3})^{\vee}  \in N_3 \otimes \det N_3^ \vee \otimes L.
\]
Thus, by direct computation, we see that this map corresponds to the matrix $\Theta(p_{ij})$.
 \end{proof}

\begin{claim}\label{claim:Theta'}
The map $\wedge^2((S^2N_3)^\vee \otimes L) \to N_3 \otimes(\det U_5^\vee \otimes L)^{\otimes 2}$ corresponds to the matrix whose $(i,j)$-th column ($i<j$) is
\[
\Theta'=
\begin{pmatrix}
 \sum _{k<l}  (-1)^{k+l} (m_{ik}m_{jl} - m_{il}m_{jk}) a_{kl}\\
 \sum _{k<l}  (-1)^{k+l} (m_{ik}m_{jl} - m_{il}m_{jk}) b_{kl}\\
 \sum _{k<l}  (-1)^{k+l} (m_{ik}m_{jl} - m_{il}m_{jk}) c_{kl}
\end{pmatrix}
\]
with respect to the bases $\wedge^2 \{(\alpha^2)^\vee, (\beta \gamma)^\vee, (\beta^2)^\vee, (\gamma \alpha)^\vee, (\gamma^2)^\vee,(\alpha \beta)^\vee\}$ and $\{\alpha \otimes (\omega_{U_5}^\vee)^{\otimes 2},\beta \otimes (\omega_{U_5}^\vee)^{\otimes 2} ,\gamma \otimes (\omega_{U_5}^\vee)^{\otimes 2} \}$.
\end{claim}
\begin{proof}
Recall that $G_\nu^\vee$ corresponds to the matrix $\trans M$.
Thus the map $\wedge^2 (G_\nu^\vee \otimes \id_L)$ corresponds to the matrix $\wedge^2 (\trans M)$, which is given by the $2 \times 2$ minor determinants of $M$.
Note that the dual basis of
\[
\{ e_2\wedge e_3\wedge e_ 4\wedge e_5, e_1 \wedge e_3\wedge e_ 4\wedge e_5,e_1 \wedge e_2 \wedge e_ 4\wedge e_5,e_1 \wedge e_2\wedge e_3 \wedge e_5,e_1 \wedge e_2\wedge e_3\wedge e_ 4 \}
\]
corresponds to the basis $\{e_1, -e_2, e_3, -e_4, e_5\}$ of $U_5$.
Thus the composite 
\[
\wedge^2((S^2N_3)^\vee \otimes L) \to \wedge^2 ( (\wedge^4 U_5 ) ^\vee \otimes L)  \to \wedge^2(U_5 \otimes \det U_5^\vee \otimes L)
\]
corresponds to the matrix
\[
( (-1)^{k+l} (m_{ik}m_{jl} - m_{il}m_{jk}) )_{((k<l),(i<j))}
\]
with respect to the basis $\wedge^2 \{(\alpha^2)^\vee, (\beta \gamma)^\vee, (\beta^2)^\vee, (\gamma \alpha)^\vee, (\gamma^2)^\vee,(\alpha \beta)^\vee\}$ and $\wedge^2\{e_1\otimes \omega_{U_5}^\vee,e_2\otimes \omega_{U_5}^\vee,e_3\otimes \omega_{U_5}^\vee,e_4\otimes \omega_{U_5}^\vee,e_5\otimes \omega_{U_5}^\vee\}$.
Further, if we compose this with $\nu_{\det U_5^\vee \otimes L}$, we have the map $\wedge^2((S^2N_3)^\vee \otimes L) \to N_3 \otimes(\det U_5^\vee \otimes L)^{\otimes 2}$ that corresponds to the matrix $\Theta'$.
\end{proof}

By using the invariant theory for nets of alternating forms  (\cite[]{SK77}, \cite[Section 2]{Gyo90} \cite[Section 3]{Och97}, \cite[Subsection 3.2]{Arithmeticinvariant2}), we have the next claim:
\begin{claim}\label{claim:invariant}
The following hold:
\begin{enumerate}
 \item \label{claim:invariant1}$\Theta ' = \Theta (p'_{ij})$ for a symmetric matrix $P' =(p'_{ij})$.
 \item \label{claim:invariant2}$\det P' = F(a_{ij}, b_{ij}, c_{ij})$ for a polynomial $F(a_{ij}, b_{ij}, c_{ij})$ of multi-degree $(5,5,5)$.

\end{enumerate}
\end{claim}

\begin{proof}
The claim \ref{claim:invariant1} basically follows from the proof of \cite[Lemma~3.1]{Och97} (cf.\ \cite{SK77}, \cite{Arithmeticinvariant2}, \cite{Gyo90}).
We recall the notations there first. 
For a $5 \times 5 $ matrix $X$, we denote by $X_{j}$  the $4 \times 4$ matrix given by deleting the $j$-th row and the $j$-th column, where $j \in \{ 1, \dots , 5 \}$. 
We then define $\widetilde\beta_{j} (X,X)$ and $\widetilde\beta_j (X,Y)$ for  $j \in \{ 1, \dots , 5 \}$ and alternating $5 \times 5$ matrices $X \neq Y$ by 
\begin{equation*}
\begin{split}
(-1)^{j-1} \widetilde\beta_{j} (X, X) 
 &=  \text{Pf} \, (X_j ) ,  \\
(-1)^{j-1} \widetilde\beta_{j} (X, Y) 
 &= \text{Pf} \, ((X + Y)_j) - \text{Pf}\, (X_j) -\text{Pf} \, (Y_j), 
\end{split}
\end{equation*}
where $\text{Pf} (X)$ denotes the Pfaffian of an alternating matrix $X$ as before.
The $\beta_j$ in \cite[Section 3]{Och97} is related to our $\widetilde\beta_j$ as follows:
\[
\beta_j (X,X) = 2 \widetilde\beta_j (X,X), \quad \beta_j (X,Y) = \widetilde\beta_j (X,Y).
\]
We denote by $\widetilde\beta (X,X)$ the vector whose components are $\widetilde\beta_1 (X,X), \widetilde\beta_2 (X,X), \dots , \widetilde\beta_5 (X,X)$; similarly, by $\widetilde\beta (X, Y)$ the vector with components $\widetilde\beta_1 (X,Y), \widetilde\beta_2 (X,Y) , \dots , \widetilde\beta_5 (X,Y)$. 

With the above notations, the $j$-th column of the $6 \times 5$ matrix $M = (m_{ij} )$ before is written as
\[
\begin{pmatrix}
 (-1)^{j-1} \widetilde\beta_{j} (A ,A) \\
 (-1)^{j-1} \widetilde\beta_{j} (B , C)  \\
 (-1)^{j-1} \widetilde\beta_{j} (B ,B) \\
 (-1)^{j-1} \widetilde\beta_{j} (C , A)  \\
 (-1)^{j-1} \widetilde\beta_{j} (C , C) \\
 (-1)^{j-1} \widetilde\beta_{j} (A , B)  \\
\end{pmatrix}.
\]

In addition, Ochiai \cite[Section 3]{Och97} introduces a trilinear form 
 $$ \langle \xi | X | \eta \rangle := \sum_{i, j =1}^{5} \xi_i x_{ij} \eta_{j} $$
for vectors $\xi =(\xi_i)_{1\leq i \leq 5}, \eta =(\eta_j)_{1\leq j \leq 5}$ and an alternating $5 \times 5$ matrix $X =(x_{ij})$. 
Since, for example, the first row of the $(i,j)$-th column of $\Theta'$ is 
$$ \sum_{k < l} (-1)^{k+l} (m_{ik} m_{jl} - m_{il}m_{jk}) a_{kl} =\sum_{k, l =0}^{5} (-1)^{k+l} (m_{ik} m_{jl}) a_{kl} $$ 
(as $A=(a_{kl})$ 
 is an alternating matrix), 
the $(i,j)$-th column $(1 \leq i<j \leq 6)$ of the $3 \times 15$ matrix $\Theta'$ is written in terms of this trilinear form as  
\[
\begin{pmatrix}
 \langle m^{(i)} | \, A \, | \, m^{(j)} \rangle \\
 \langle m^{(i)} | \,  B \, |\,  m^{(j)} \rangle \\
 \langle m^{(i)} | \, C \, |\,  m^{(j)} \rangle \\
\end{pmatrix}, 
\]
where $m^{(i)}$ and $m^{(j)}$ denote the $i$-th and $j$-th row vectors of the matrix $M$, respectively. 
For example, the first column (i.e.,\ $(i,j)=(1,2)$) of $\Theta'$ is:
\[ 
\begin{pmatrix}
 \langle \widetilde\beta (A,A) | \,  A  \,| \, \widetilde\beta (B,C) \rangle \\
 \langle \widetilde\beta (A,A)  |\,  B \, |\,  \widetilde\beta (B,C) \rangle \\
 \langle \widetilde\beta (A,A)  |\,  C \, |\,  \widetilde\beta (B,C)  \rangle \\
\end{pmatrix} .
\]
Then, the claim (1) follows from the proof of formulas in \cite[Lemma~3.1]{Och97}. 
In fact, the symmetric matrix $P'$ is written as
\[
P' = 
\begin{pmatrix}
 \langle \widetilde\beta (A,A) | \,  C  \,| \, \widetilde\beta (A,B) \rangle & \langle \widetilde\beta (A,A) | \,  C  \,| \, \widetilde\beta (B,B) \rangle  &\langle \widetilde\beta (C,C) | \,  B  \,| \, \widetilde\beta (A,A) \rangle \\
\langle \widetilde\beta (A,A) | \,  C  \,| \, \widetilde\beta (B,B) \rangle &  \langle \widetilde\beta (B,B) | \,  C  \,| \, \widetilde\beta (B,C) \rangle  &\langle \widetilde\beta (B,B) | \,  A  \,| \, \widetilde\beta (C,C) \rangle \\
 \langle \widetilde\beta (C,C) | \,  B  \,| \, \widetilde\beta (A,A) \rangle & \langle \widetilde\beta (B,B) | \,  A  \,| \, \widetilde\beta (C,C) \rangle  &  \langle \widetilde\beta (C,C) | \,  B  \,| \, \widetilde\beta (C,A) \rangle \\
\end{pmatrix}.
\]

The claim \ref{claim:invariant2} is then trivial.
\end{proof}

The next claim is a main technical part of the proof:
\begin{claim}\label{claim:discriminant_square}
Let $P'$ be the symmetric matrix in the previous claim, and set $F(a_{ij}, b_{ij}, c_{ij}) = \det P'$.
Then we have
\[
P' Q_{\nu} = Q_\nu P'
= F(a_{ij}, b_{ij}, c_{ij}) E
= \begin{pmatrix}
F(a_{ij}, b_{ij}, c_{ij}) & 0 &0 \\ 0 & F(a_{ij}, b_{ij}, c_{ij}) &0 \\ 0&0 & F(a_{ij}, b_{ij}, c_{ij})
\end{pmatrix},
\]
where $E$ is the identity matrix of size $3$.
\end{claim}

\begin{proof}
The claim can be deduced by observing that the adjugate matrix $\widetilde{P'}$ of $P'$ is $Q_\nu$.
Note that, since the constructions of $P'$ and $Q_{\nu}$ are compatible with any base change, we may assume that our base ring $R$ is the (reduced) polynomial ring $\Z[a_{ij},b_{ij},c_{ij}]$ with indeterminates $a_{ij}$, $b_{ij}$, $c_{ij}$.
Now we prove $\widetilde{P'} = Q_\nu$ in two ways.
The first proof is more direct, but relies on computer calculations:
By using a computer algebra system (e.g.,\ \texttt{Macaulay2} or \texttt{SageMath}), we can check directly the equation $\widetilde{P'} = Q_\nu$.
The second does not rely on computer calculations, but is slightly indirect:
Since $R = \Z[a_{ij},b_{ij},c_{ij}]$ is reduced, we are reduced to proving the equation $\widetilde{P'} = Q_\nu$ when $R = \Spec (k)$ for an algebraically closed filed $k$ (by looking at the geometric points).
If $R = \Spec (k)$, then any net of alternating forms is isomorphic to a split form, and the assertion can be checked directly (cf.\ Definition-Proposition \ref{defn-prop:splitform}).
Thus, the claim follows.
\end{proof}

\begin{proof}[Proof of Lemma~\ref{lemma:net_of_altaernating_forms}]
As we have seen before, it is reduced to proving the assertion for the local case.
 Let $H_{\nu} \colon N_3 \otimes \det N_3^\vee \otimes L \to N_3 \otimes(\det U_5^\vee \otimes L)^{\otimes 2}$ be tha map corresponding to the matrix $F(a_{ij}, b_{ij}, c_{ij}) E$ with respect to the bases 
$\{\alpha \otimes \omega_{N_3}^\vee,\beta \otimes \omega_{N_3}^\vee ,\gamma \otimes \omega_{N_3}^\vee \}$ and  $\{\alpha \otimes (\omega_{U_5}^\vee)^{\otimes 2},\beta \otimes (\omega_{U_5}^\vee)^{\otimes 2} ,\gamma \otimes (\omega_{U_5}^\vee)^{\otimes 2} \}$.
Then the assertion follows from Claims~\ref{claim:invariant} and \ref{claim:discriminant_square}.
\end{proof}

\begin{rem}\label{remark:discriminant_square}
 By Claim \ref{claim:discriminant_square}, we have
 \[
 \det Q_\nu = (F(a_{ij}, b_{ij}, c_{ij}))^2,
 \]
 and hence $\det Q_\nu$ is a complete square.
 
 In the global setting, 
$\det \mu_{\varphi_{\nu}} \colon \det N_3 \to \det (N_3^\vee \otimes L)$ gives a trivialization of $12 c_1(U_5) - 10 c_1(N_3)$, while the above invariant $F$ trivializes $6 c_1(U_5) - 5 c_1(N_3)$.
\end{rem}

\begin{lem}\label{lemma:A_A'}
 $\sA \simeq \sA'$.
\end{lem}

\begin{proof}
We prove that  any  element $(U_5,N_3,\nu) \in \sA'$ is similar to an element in $\sA$.
By Remark~\ref{remark:discriminant_square}, we have $6 c_1(U_5) = 5 c_1(N_3)$.
Then $\nu_{\det U_5 \otimes \det N_3 ^\vee}$ gives an element in $\sA$, which is similar to $\nu$. 
\end{proof}

\section{Automorphism groups of $V_5$-schemes}\label{section:automorphisms}

 Let $Y/B$ be a $V_5$-scheme and $(N_3,L,\varphi)$ the symmetric bilinear form corresponding to $Y$.
 This defines a point $[\varphi] \in \P(S^2N_3)$.
By Theorem~\ref{thm:W5description}, the automorphism group scheme $\Aut_{Y/B}$ is isomorphic to the stabilizer group scheme of $[\varphi] \in \P(S^2N_3)$ in $\Aut_{\P(N_3)/B}$, i.e.,\ $\Aut_{Y/B}$ is isomorphic to the subgroup scheme in $\Aut_{\P(N_3)/B}$ preserving the the symmetric bilinear form.
For the split $V_5$-scheme, we have the following:
\begin{prop}\label{prop:aut_split_W5}
Let $Y$ be the split $V_5$-scheme.
Then the following hold:
\begin{enumerate}
 \item 
 The above embedding 
\[\Aut_{Y/B} \subset \PGL_{3,B} = (\Proj \Z[a_{\alpha\alpha}, a_{\alpha\beta}, a_{\alpha\gamma}, a_{\beta\alpha}, a_{\beta\beta}, a_{\beta\gamma}, a_{\gamma\alpha}, a_{\gamma\beta}, a_{\gamma\gamma}] \setminus (\det (a_{\bullet,\bullet})=0) ) \times _{\Spec \Z} B
\]
is defined by the equations
\begin{equation}\label{eq:aut_Y_B_split}
\begin{cases}
a_{\alpha\gamma}^2-2a_{\alpha\alpha}a_{\alpha\beta}=0,\\
a_{\beta\gamma}^2-2a_{\beta\alpha}a_{\beta\beta}=0,\\
a_{\alpha\gamma}a_{\gamma\gamma}-a_{\alpha\alpha}a_{\gamma\beta}-a_{\gamma\alpha}a_{\alpha\beta}=0,\\
a_{\beta\gamma}a_{\gamma\gamma}-a_{\beta\alpha}a_{\gamma\beta}-a_{\gamma\alpha}a_{\beta\beta}=0,\\
a_{\alpha\gamma}a_{\beta\gamma}-a_{\alpha\alpha}a_{\beta\beta}-a_{\beta\alpha}a_{\alpha\beta}=-a_{\gamma\gamma}^2+2a_{\gamma\alpha}a_{\gamma\beta}.
\end{cases}
\end{equation}

\item Over $B= \Spec \Z[1/2]$, we have $\Aut_{Y/B} \simeq \PGL_{2,B}$. More precisely, we have the following morphism 
\[
\sigma = \sigma_{\Z[\frac{1}{2}]} \colon
\PGL_{2, \Z[\frac{1}{2}]}
\simeq \Aut_{Y_{\Z[\frac{1}{2}]}/\Z[\frac{1}{2}]}
\hookrightarrow \PGL_{7,\Z[\frac{1}{2}]}
\]
that sends
$
\begin{pmatrix}
a & b \\
c & d \\
\end{pmatrix}
$
to
\begin{equation}
\label{eqn:actionPGL2}
\begin{psmallmatrix}
a^6 & 2a^5b & 10a^4b^2 & 20a^3b^3 & 20a^2b^4 & 8ab^5 & 8b^6 \\
3a^5c &  5a^4bc + a^5d & 20a^3b^2c + 10a^4bd & 30a^2b^3c + 30a^3b^2d & 20ab^4c + 40a^2b^3d & 4b^5c + 20ab^4d & 24b^5d \\
\frac{3}{2}a^4c^2 & 2a^3bc^2 + a^4cd & 6a^2b^2c^2 + 8a^3bcd + a^4d^2 & 6ab^3c^2 + 18a^2b^2cd + 6a^3bd^2 & 2b^4c^2 + 16ab^3cd +12 a^2b^2d^2 & 4b^4cd + 8ab^3d^2 & 12b^4d^2\\
a^3c^3 & a^2bc^3 + a^3c^2d & 2ab^2c^3 + 6a^2bc^2d +2a^3cd^2 &  b^3c^3 + 9ab^2c^2d + 9a^2bcd^2 + a^3d^3 & 4b^3c^2d + 12ab^2cd^2 + 4a^2bd^3 & 4b^3cd^2 + 4ab^2d^3 & 8b^3d^3 \\
\frac{3}{4}a^2c^4 & \frac{1}{2}abc^4 + a^2c^3d & \frac{1}{2}b^2c^4 + 4abc^3d + 3a^2c^2d^2 & 3b^2c^3d + 9abc^2d^2 + 3a^2cd^3 & 6b^2c^2d^2 + 8abcd^3 + a^2d^4 & 4b^2cd^3 + 2abd^4 & 6b^2d^4 \\
\frac{3}{4}ac^5 & \frac{1}{4}bc^5 + \frac{5}{4}ac^4d & \frac{5}{2}bc^4 + 5ac^3d^2 & \frac{15}{2}bc^3d^2 + \frac{15}{2}ac^2d^3 & 10bc^2d^3 + 5acd^4 & 5bcd^4 + ad^5 & 6bd^5 \\
\frac{1}{8}c^6 & \frac{1}{4}c^5d & \frac{5}{4}c^4d^2 &  \frac{5}{2}c^3d^3 & \frac{5}{2}c^2d^4 & cd^5 & d^6 \\
\end{psmallmatrix}
\end{equation}
We denote the induced action of $\PGL_{2, \Z[\frac{1}{2}]}$ on $Y_{\Z[\frac{1}{2}]}$ by $\sigma$ too.

\item 
Over $B = \Spec \F_2$, $\Aut_{Y/B}$ is \emph{non-reduced} and we have $\Aut_{Y/B,\red} \simeq \SL_{2,\F_2}$.
More precisely, we have the following morphism
\[\sigma' \colon \SL_{2, \F_2} \simeq \Aut_{Y/B,\red} \rightarrow \PGL_{7,\F_2} = \Aut_{\P^6_{\F_2}}
\]
given by
\begin{equation}
\label{eqn:sigma'}
\begin{pmatrix}
a & b \\
c & d \\
\end{pmatrix}
\mapsto 
\begin{pmatrix}
a^3 & 0 &a^2b &0 & ab^2 &0 &b^3\\
0&a^2&0&ab&0&b^2&0\\
a^2c&0&a^2d&0&b^2c&0&b^2d\\
0&0&0&1&0&0&0\\
ac^2&0&bc^2&0&ad^2&0&bd^2\\
0&c^2&0&cd&0&d^2&0\\
c^3&0&c^2d&0&cd^2&0&d^3\\
\end{pmatrix}.
\end{equation}
We denote the induced action of $\SL_{2,\F_2}$ on $Y_{\F_2}$ by $\sigma'$ too.
See Section~\ref{section:characteristic_two} for the description of the non-reduced group scheme $\Aut_{Y/B}$.
\end{enumerate}
\end{prop}

\begin{proof}
Consider the action of $\PGL_{3,B}$ on $\P(N_3)=\P^2$ given by $(A,x) \mapsto A x$, which induces an action of $\PGL_{3,B}$ on $\P(S^2N_3)$.
Then we see that equations~\eqref{eq:aut_Y_B_split} define the automorphism group scheme preserving $[\varphi_{\spl}] \in \P(S^2N_3)$, and the first assertion follows by Theorem \ref{thm:W5description} and Remark \ref{rem:stacky}.

When $2$ is invertible, then symmetric bilinear forms are equivalent to quadratic forms, and thus $\Aut_{Y/B} \simeq \PGL_{2,B}$.
On $\P^2=\P(N_3)$, the corresponding conic is defined by $-2xy+z^2 =0$ with coordinates $(x,y,z)$ with respect to the basis $\{\alpha,\beta,\gamma\}$.
This conic is the image of the map $\P^1 \to \P^2$ sending
\[
(p,q) \mapsto (2p^2,q^2,-2pq).
\]
Then the action of $\PGL_2$ on $\P^1$ given by 
$(A,\trans(p,q)) \mapsto A \trans(p,q)$
induces a map $\PGL_2 \to \PGL_3$ given by
\begin{equation}\label{eq:PGL2_PGL3}
 \begin{pmatrix}
 a&b\\c&d
\end{pmatrix}
\mapsto
\begin{pmatrix}
 a^2&2b^2&-2ab\\
 \frac{1}{2}c^2&d^2&-cd\\
 -ac&-2bd&ad+bc
\end{pmatrix}.
\end{equation}
Via the map \eqref{eq:cokernel_q_dual}, we have the induced projective representation of $\PGL_2$ on $U_5$, which preserves the kernel of $\nu_{\spl}$.
Note that that the coordinates on $Y_{\spl}$ in \eqref{eqn:W5} is induced from the coordinates of $\Gr(U_5,3)=\Gr(U_5^\vee,2)$ with respect to the basis $\{e_i\wedge e_j\}$.
By a direct computation, we see that the induced action of $\PGL_2$ on $Y_{\spl}$ is given by \eqref{eqn:actionPGL2}.

When $2=0$, then the equation \eqref{eq:aut_Y_B_split} says that the image $\Aut_{Y/B,\red} \subset \PGL_3$ is isomorphic to $\SL_2$
 given by the inclusion $\SL_{2} \to \PGL_3$ sending
\begin{equation}
 \begin{pmatrix}\label{eq:SL2_PGL3}
 a&b\\c&d
\end{pmatrix}
\mapsto
\begin{pmatrix}
 a&b&0\\
 c&d&0\\
 0&0&1
\end{pmatrix}.
\end{equation}
Via the map \eqref{eq:cokernel_q_dual}, we have the induced action of $\SL_2$ on $U_5$, which preserves the kernel of $\nu_{\spl}$.
By a direct computation again, we see that the induced action of $\SL_2$ on $Y_{\spl}$ is given by \eqref{eqn:sigma'}.
\end{proof}

\begin{rem}
The action given in \eqref{eqn:actionPGL2} is compatible with the classical description in \cite[p.\ 505]{Mukai-Umemura} with respect to Remark \ref{rem:coordinatechange}. 
\end{rem}
{
\begin{rem}
Theorem \ref{thm:W5description} and Proposition \ref{prop:aut_split_W5} shows that the minimum $N$ satisfying the statement of \cite[Lemma 3.14]{Javanpeykar-Loughran:GoodReductionFano} is $2$.
\end{rem}
}
\begin{rem}
We put $\cO_{K} \coloneqq \Z[\sqrt{2}]$, and $K \coloneqq \Frac \cO_{K}$.
Let 
\[
c \colon \PGL_{2, \cO_K[\frac{1}{\sqrt 2}]} \simeq \PGL_{2, \cO_K[\frac{1}{\sqrt 2}]}
\]
be the conjugation by the element
\[
\begin{pmatrix}
1 &0\\
0 & \sqrt{2}\\
\end{pmatrix}.
\]
We can show that the composite of $c$ and \eqref{eq:PGL2_PGL3}
extends to $\PGL_{2, \cO_K} \rightarrow \PGL_{3, \cO_K} $ which factors through $\Aut_{Y_{\cO_K}/\cO_K} \subset \PGL_{3, \cO_K}$.
This induces a morphism
\begin{equation}
\label{eqn:ramifiedsigma}
\widetilde{\sigma}\colon  \PGL_{2, \cO_K} \rightarrow \Aut_{Y_{\cO_K}/\cO_K} \subset \PGL_{7, \cO_K}.
\end{equation}
On the special fiber over $\F_2$, the map
\[
\widetilde{\sigma}_{\F_2} \colon \PGL_{2, \F_2} \rightarrow \PGL_{7, \F_2}
\]
factors as
\begin{equation}
\label{eqn:factorsigma}
\PGL_{2, \F_2} = \PSL_{2, \F_2} \xrightarrow{F} \SL_{2,\F_2} \xrightarrow{\sigma'} \PGL_{7, \F_2},
\end{equation}
where $F$ is the morphism induced by the Frobenius morphism
\[
\SL_{2,\F_2} \xrightarrow{F} \SL_{2,\F_2}.
\]
\end{rem}

\section{Hilbert scheme of lines}\label{section:hilbert_scheme_of_lines}

In this section, we describe the Hilbert scheme of lines on $V_5$-schemes as the projective bundle $\P(N_3)$.
Then we will classify the lines on the split $V_5$-varieties, which will be used to describe the orbit decomposition of $V_5$.
See also \cite{FN89} and \cite[Section 2]{Ili94} for the results over $\C$.

\subsection{Relative Hilbert scheme of lines}

\begin{defn}
\label{defn:Sigma}
Let $Y/B$ be a $V_{5}$-scheme.
We denote the relative Hilbert scheme of lines on $Y$ over $B$ by
$\Sigma (Y/B)$, and the universal family of lines over $\Sigma(Y/B)$ by $\mathcal{U}_{Y/B}$.
We denote by
\[
\xymatrix{
\mathcal{U}_{Y/B}  \ar[d]^-{q} \ar[r]^-{p} & Y \\
\Sigma(Y/B)
}
\]
the natural projections.
\end{defn}

Let $Y/B$ be a $V_5$-scheme, $(U_5,N_3,\nu)$ the net of alternating forms, and $(N_3,L,\varphi)$ the corresponding symmetric bilinear form.
Recall that we have an exact sequence
\[
0 \to \wedge^4 U_5 \to S^2N_3 \to L \to 0.
\]
Since $\varphi$ is non-degenerate, the projection from $[\varphi] \in \P(S^2N_3)$ embeds the Veronese surface $\P(N_3) \subset \P(S^2N_3)$ into $\P(N_3) \subset \P(\ker \varphi) \simeq \P(\wedge^4 U_5)$.
Note that the secant variety of $\P(N_3) \subset \P(S^2N_3)$ is the divisor of the degenerate symmetric bilinear forms in $\P(S^2N_3)$, and hence the projection  $\P(N_3) \to  \P(\wedge^4 U_5)$ is an embedding.
Since the automorphism group scheme $\Aut_{Y/B}$ is isomorphic to the stabilizer group scheme of $[\varphi] \in \P(S^2N_3)$ in $\Aut_{\P(N_3)/B}$, we have the $\Aut_{Y/B}$-equivariant embedding
\begin{equation}
 \P(N_3) \subset  \P(\wedge^4 U_5).
\end{equation}

Recall that $Y/B$ is contained in $\Gr(U_5,3)$ and that the Hilbert scheme of lines on $\Gr(U_5,3)$ is naturally isomorphic to $\Fl(U_5;4,2)$.
Thus, we have an embedding $\Sigma(Y/B) \subset \Fl(U_5;4,2)$.
By composing with the projection $\Fl(U_5;4,2) \to \Gr(U_5,4) \simeq \P(\wedge^4 U_5)$, we obtain a map
\begin{equation}
\Sigma(Y/B) \to  \P(\wedge^4 U_5),
\end{equation}
which is also $\Aut_{Y/B}$-equivariant.
In this subsection, we prove:
\begin{prop}\label{proposition:Hilb_of_lines}
$ \P(N_3) \subset  \P(\wedge^4 U_5)$ and $\Sigma(Y/B) \to  \P(\wedge^4 U_5)$ are $\Aut_{Y/B}$-equivariantly isomorphic to each other.
\end{prop}

Here we provide the construction of the isomorphism (see also \cite[Section 2]{Ili94} for the results over $\C$):
Let $g \colon \P(N_3) \to B$ be the projection, and $\cO(\xi)$ the tautological quotient line bundle on $\P(N_3)$.
On $\P(N_3)$, we have the universal surjective map
\[
g^*N_3 \to \cO(\xi).
\]
By composing this with the map $g^*\nu \colon g^* (\wedge^2 U_5) \to g^* N_3$, we have an alternating form
\[
\wedge ^2 g^{\ast} U_5 \to \cO(\xi).
\]
Since $\nu$ is a net of alternating forms of rank $4$, the form $\wedge ^2 g^* U_5 \to \cO(\xi)$ is also of rank $4$.
Thus, we have the exact sequence
\[
0 \to M  \to  g^*U_5 \to g^*U_5^\vee \otimes \cO(\xi ) \to M^\vee \otimes \cO(\xi) \to 0,
\]
where $M$ is a line bundle.
Denote by $C_4$ the image of the map $g^*U_5 \to g^*U_5^\vee \otimes \cO(\xi)$, which is a locally free sheaf of rank $4$.
By taking $\wedge^3$ of the map $g^*U_5 \to C_4$, we have the following sequence
\[
0 \to \wedge^2 C_4 \otimes M \to \wedge^3 (g^*U_5) \to \wedge^3C_4 \to 0,
\]
and a $\P^3$-bundle $\P( \wedge^3C_4) \subset \Gr(g^* U_5,3) \subset \P(\wedge^3 (g^*U_5))$.
The proposition follows from:
\begin{claim}
The following hold:
 \begin{enumerate}
 \item The intersection $\P( \wedge^3C_4) \cap (Y\times_B \P(N_3)) \subset \Gr(g^*U_5,3)$ is a relative line over $\P(N_3)$ in $Y$.
 \item We have the induced map $\P(N_3) \to \Sigma(Y/B)$. The composite of this map with $ \Sigma(Y/B) \to \P(\wedge^4U_5)$ is the projected Veronese surface $\P(N_3) \subset \P(\wedge^4 U_5)$.
 \item $\Sigma(Y/B) \to B$ is a smooth morphism of relative dimension $2$ with connected fibers.
\end{enumerate}
\end{claim}

\begin{proof}
Consider the following diagram with two short exact sequences:
\[
\xymatrix{
0 \ar[r] & \cO(-\xi) \ar[r] \ar@{-->}[d]  & g^*N_3^\vee \ar[r] \ar[d] & T_g(-\xi) \ar[r] \ar@{-->}[d] &0\\
0 \ar[r] & \wedge^2C_4 \otimes M \ar[r] & g^*(\wedge^3 U_5) \ar[r] & \wedge^3C_4 \ar[r] &0.
}
\]
By the construction, the composite $\cO(-\xi) \to \wedge^3C_4$ is a zero map.
Thus, the map $g^*N_3 ^\vee \to \wedge^3 C_4$ factors $T_g(-\xi)$, and hence the intersection $\P( \wedge^3C_4) \cap Y$ is of relative dimension at least one.
Then the first assertion follows from the fact that any $V_5$-variety does not contain $\P^3$ nor $\P^2$ linearly.

Next, consider the following diagram with two short exact sequences:
\[
\xymatrix{
0 \ar[r] & \wedge^3C_4 \otimes M \ar[r] \ar@{-->}[d] & g^*\wedge^4 U_5 \ar[r] \ar[d] & \wedge^4C_4 \ar[r] \ar@{-->}[d] &0\\
0 \ar[r] & \bullet  \ar[r] & g^*(S^2N_3) \ar[r]& \cO(2\xi) \ar[r] &0,\\
}
\]
where the first horizontal sequence is induced from the fourth exterior power of $U_5 \to C_4$, and
the second horizontal sequence is induced from the second symmetric power of $g^*N_3 \to \cO(\xi)$.
By the construction, the induced map $\wedge^3C_4 \otimes M  \to \cO(2\xi)$ is a zero map, which induces an isomorphism $\wedge^4C_4 \to \cO(2\xi)$.
Thus, the second assertion holds.

The last assertion can be proved by looking carefully at the explicit example $Y_{\spl}$.
Alternatively, arguing as in \cite[Proposition 2.2.8]{Kuznetsov-Prokhorov-Shramov}, we see that $\Sigma(Y/B)$ is smooth of relative dimension two, and that lines are unobstructed.
Then the assertion follows from that for the characteristic zero case.
\end{proof}

\begin{rem}\label{rem:projection_veronese}
By \eqref{eq:cokernel_q_dual}, we see that the map $\P(N_3)=\P^2 \to \P(\wedge^4U_5) \simeq \P(U_5^\vee) =\P^4$ for $Y_{\spl}$ is given by
\[
(x,y,z) \mapsto (-x^2,zx,-z^2-xy,yz,-y^2).
\]
\end{rem}

\subsection{Classification of lines on split $V_5$-varieties}

\begin{defn}\label{def:lines_split_W5}
 Let $Y_{\spl}$ be the split $V_5$-variety over an algebraically closed field $k$.
 Then $\Sigma (Y) \simeq \P(N_3) = \P^2$.
 We will denote by $l_{(x,y,z)}$ the line corresponding to the point $(x,y,z) \in \P^2$, and by $G_{l_{(x,y,z)}}$ the stabilizer group (scheme) of the point $[l_{(x,y,z)}] \in \Sigma(Y)$ with respect to the $\Aut(Y)_{\red}$-action.
\end{defn}

\begin{cor}\label{cor:lines_split_W5}
 Let $Y$ be the (split) $V_5$-variety over an algebraically closed field $k$ of characteristic $p \geq 0$.
 We identify $\Sigma(Y)$ to $\P(N_3)$ as in Proposition~\ref{proposition:Hilb_of_lines}.
 Then the orbit decomposition of $\Sigma(Y)$ with respect to $\Aut(Y)_{\red}$ is given as follows:
 \begin{enumerate}
 \item If $p \neq 2$, there are two orbits
 \[
 V(z^2-2xy) \sqcup (\P(N_3) \setminus V(z^2-2xy)).
 \]
 Moreover, the stabilizer groups of the following points are given as follows: 
 \begin{enumerate}
  \item For the point $(0,1,0) \in  V(z^2-2xy)$, the stabilizer group is isomorphic to
  \[
 G_{l_{(0,1,0)}} =   \left\{
\begin{pmatrix}
 a & 0\\c&d
\end{pmatrix} \in \PGL_{2,k}
 \right\}.
  \]
    \item For the point $(0,0,1) \in \P(N_3) \setminus V(z^2-2xy)$, the stabilizer group is isomorphic to
  \[
  G_{l_{(0,0,1)}} =
  \left\{
\begin{pmatrix}
 a & 0\\0&d
\end{pmatrix} \in \PGL_{2,k} 
 \right\} \bigsqcup   \left\{
\begin{pmatrix}
 0 & b\\c&0
\end{pmatrix} \in \PGL_{2,k}
 \right\}.
  \]
\end{enumerate}
 \item If $p =2$, there are three orbits $\{(0,0,1)\}$, $V(z)$ and $\P(N_3) \setminus (V(z) \cup \{(0,0,1)\})$.
 Moreover, the stabilizer groups of the following points are given as follows: 
 \begin{enumerate}
 \item For the point $(0,0,1) \in \P(N_3)$, the stabilizer group $G_{l_{(0,0,1)}}$ is $\Aut(Y)_{\red} \simeq \SL_{2,k}$.
  \item For the point $(0,1,0) \in  V(z)$, the stabilizer group is isomorphic to
  \[
  G_{l_{(0,1,0)}}=
  \left\{
\begin{pmatrix}
 a & 0\\c&1/a
\end{pmatrix} \in \SL_{2,k}
 \right\}.
  \]
  \item For the point $(1,1,1) \in \P(N_3)\setminus V(z)$, the stabilizer group is isomorphic to
  \[
  G_{l_{(1,1,1)}}=
  \left\{
\begin{pmatrix}
 1+c & c \\ c &1+c
\end{pmatrix} \in \SL_{2,k} 
 \right\},
  \]
  which is isomorphic to $\G_a$.
\end{enumerate}
\end{enumerate}
In each case, the orbit contains a rational point over the prime field.
Thus, the above statements hold also over any field $k$.
\end{cor}

\begin{proof}
 This follows from direct computations using \eqref{eq:PGL2_PGL3} and \eqref{eq:SL2_PGL3}. 
\end{proof}

\begin{defn}
Let $Y$ be a $V_5$-variety over an algebraically closed field $k$, which is isomorphic to the split $V_5$-variety.
The Hilbert scheme of lines $\Sigma(Y) \simeq \P^2$ admits the decomposition
\[
\Sigma(Y) \simeq \P^2 
=
\begin{cases}
\text{(open orbit)} \sqcup V(z^2-2xy) & (p\neq 2),\\
\text{(open orbit)} \sqcup V(z) \sqcup \{(0,0,1)\} & (p=2).
\end{cases}
\]
A line that lies in the open orbit is called \emph{ordinary}, whereas a line in the one-dimensional orbit is referred to as \emph{special}.
The line corresponding to the point $(0,0,1) \in \P^2$ with $p=2$ is called \emph{exceptional}.

More generally, a line $l$ in a $V_5$-variety over a field $k$ is called \emph{ordinary} if $l$ corresponds to an ordinary line over $\overline k$.
Exceptional or special lines are defined similarly.
\end{defn}

The next proposition describes the above lines and the action of the above stabilizer groups on the corresponding lines.
\begin{prop}\label{prop:line_stab_action}
 Let $Y_{\spl}$ be the split $V_5$-variety over an algebraically closed field $k$ of characteristic $p \geq 0$.
\begin{enumerate}
 \item Assume $p \neq 2$.
\begin{enumerate}
 \item  We have
 \[
 l_{(0,1,0)} = \{(0,0,0,0,0,s,t)\} \subset Y_{\spl} \subset \P^6.
 \]
 Further, $ G_{l_{(0,1,0)}} \subset \PGL_{2,k}$  acts as 
\[
\begin{pmatrix}
 a & 0\\c&d
\end{pmatrix} \cdot (0,0,0,0,0,s,t) = (0,0,0,0,0,as,cs+dt).
\]
This action fixes the point $(0,0,0,0,0,0,1)$, and acts transitively on the complement.

 \item We have
 \[
 l_{(0,0,1)} = \{(0,s,0,0,0,t,0)\} \subset Y_{\spl} \subset \P^6.
 \]
 Further, $ G_{l_{(0,0,1)}} \subset \PGL_{2,k}$  acts as 
\[
\begin{pmatrix}
 a & 0\\0&d
\end{pmatrix} \cdot (0,s,0,0,0,t,0) = (0,a^4s,0,0,0,d^4t,0)
\]
and
\[
\begin{pmatrix}
 0 & b\\c&0
\end{pmatrix} \cdot (0,s,0,0,0,t,0) = (0,4b^4t,0,0,0,\dfrac{1}{4}c^4s,0).
\]
This action fixes the two point sets $\{(0,1,0,0,0,0,0),(0,0,0,0,0,1,0)\}$, and acts transitively on the complement.

\end{enumerate}
 
 \item Assume $p = 2$.
\begin{enumerate}
 \item We have
 \[
 l_{(0,0,1)} = \{(0,s,0,0,0,t,0)\} \subset Y_{\spl} \subset \P^6.
 \]
 Further, $ G_{l_{(0,0,1)}} = \SL_{2,k}$  acts transitively as 
\[
\begin{pmatrix}
 a & b\\c&d
\end{pmatrix} \cdot (0,s,0,0,0,t,0) = (0,as+bt,0,0,0,cs+dt,0).
\]
 \item We have
 \[
 l_{(0,1,0)} = \{(0,0,0,0,0,s,t)\} \subset Y_{\spl} \subset \P^6.
 \]
 Further, $ G_{l_{(0,1,0)}}  \subset \SL_{2,k}$  acts as 
\[
\begin{pmatrix}
 a & 0\\c&1/a
\end{pmatrix} \cdot (0,0,0,0,0,s,t) = (0,0,0,0,0,\frac{1}{a^2} s,\frac{1}{a^3}t).
\]
This action fixes the two point $(0,0,0,0,0,1,0)$ and $(0,0,0,0,0,0,1)$, and acts transitively on the complement.

 \item We have
 \[
 l_{(1,1,1)} = \{(s,t,t,s+t,t,t,s)\} \subset Y_{\spl} \subset \P^6.
 \]
 Further, $ G_{l_{(1,1,1)}}  \subset \SL_{2,k}$  acts on the line $(s,t,t,s+t,t,t,s)$ via
\[
\begin{pmatrix}
 1+c & c\\c&1+c
\end{pmatrix} \cdot (s,t) = ((c^2+c+1)s+(c^2+c)t, (c^2+c)s+(c^2+c+1)t).
\]
This action fixes the point $(1,1)$, and acts transitively on the complement.
\end{enumerate}
\end{enumerate}
\end{prop}

\begin{proof}
 These follow from direct computations using Proposition~\ref{prop:aut_split_W5} and Corollary~\ref{cor:lines_split_W5}.
Here, we provide the proof only for the first case.
 
 The point $(0,1,0) \in \P(N_3)$ corresponds to the alternating form given by
 \[
 B=\begin{pmatrix}
 0&0&0&-1&0\\
 0&0&1&0&0\\
 0&-1&0&0&0\\
 1&0&0&0&0\\
 0&0&0&0&0\\
\end{pmatrix}.
\]
Thus the cokernel $C_4$ of the map $U_5 \to U_5 ^\vee$ is identified with $U_5 / \langle \trans(0,0,0,0,1) \rangle$.
This implies that the map $\wedge ^3 U_5 \to \wedge^3 C_4$ sends the vectors $e_1 \wedge e_2 \wedge e_3$, $e_1 \wedge e_2 \wedge e_4$, $e_1 \wedge e_3 \wedge e_4$ and $e_2 \wedge e_3 \wedge e_4$ to a basis of $\wedge^3 C_4$, and the other vectors $e_i\wedge e_j\wedge e_k$ to $0$.
Furthermore, the remaining alternating forms corresponding to
  \[
\begin{pmatrix}
 0&0&0&0&0\\
 0&0&0&0&-1\\
 0&0&0&1&0\\
 0&0&-1&0&0\\
 0&1&0&0&0\\
\end{pmatrix},
\begin{pmatrix}
 0&0&0&0&-1\\
 0&0&0&1&0\\
 0&0&0&0&0\\
 0&-1&0&0&0\\
 1&0&0&0&0\\
\end{pmatrix}
 \]
 sends $e_1 \wedge e_3 \wedge e_4$ and $e_2 \wedge e_3 \wedge e_4$ to $0$.
The remaining trivectors $e_1 \wedge e_2 \wedge e_3$ and $e_1 \wedge e_2 \wedge e_4$ define a line 
$(b_{12}, \dots, b_{34}, b_{35}, b_{45})=(0,\dots,0,s,t) \subset \Gr(U_5,3) =\Gr(U_5^\vee,2)$ 
with respect to the Pl\"ucker coordinates on $\Gr(U_5^\vee,2)$ corresponding to the basis $\{e_i\wedge e_j\}$.
 This line corresponds to the line $(0,0,0,0,0,s,t) \subset Y_{\spl}$.
 The assertion on the action of $G_{l_{(0,1,0)}}$ follows from Proposition~\ref{prop:aut_split_W5}.
\end{proof}

\section{Orbit decomposition of $V_5$-varieties}\label{section:orbit}
In this section, we will describe the orbit decomposition of the split $V_5$-variety, and discuss the relations between lines and orbits.
For the original treatment in characteristic zero, see \cite{Mukai-Umemura}, \cite{FN89}, \cite{Ili94} (see also \cite[Section 5]{Kuznetsov-Prokhorov-Shramov}).

\subsection{Orbit decomposition in characteristic $p \neq 2$}
\begin{prop}
\label{prop:generalW5}
Let $Y= Y_{\spl} \subset \P^6$ be the split $V_5$-variety over an algebraically closed field $k$ of characteristic $p \neq 2$.
Then the following hold.
\begin{enumerate}

\item 
The orbit decomposition of $Y$ with respect to $\sigma$ is given by 
\[
O_3^{(p)} \sqcup O_2^{(p)} \sqcup O_1^{(p)},
\]
where $O_{i}^{(p)}$ are $i$-dimensional orbits.
Moreover, the Zariski closure of $O_2^{(p)} \subset Y$ is $O_{2}^{(p)} \sqcup O_{1}^{(p)}$.
\item Let $D$ be the divisor on $Y_{\spl}$ with the equation
\begin{equation}
\label{eqn:D}
    5a_2a_4-4a_1a_5+27a_0a_6,
\end{equation}
which is defined over $Y_{\spl}/\Z$.
Then $O_2^{(p)} \sqcup O_1^{(p)}= D$, which is a prime divisor on $Y_{\spl}/k$.
\item 
Let $\nu$ be the morphism
\[
\P^1_{\Z[\frac{1}{2}]} \times_{\Z[\frac{1}{2}]} \P^1_{\Z[\frac{1}{2}]} \rightarrow D_{\Z[\frac{1}{2}]} 
\]
defined by
\begin{align*}
& ((\alpha, \gamma), (\beta, \delta))\\ 
&\mapsto
(8\alpha\beta^5 ,4\beta^5\gamma + 20\alpha\beta^4\delta,4\beta^4\gamma\delta + 8\alpha\beta^3\delta^2, 4\beta^3\gamma\delta^2 + 4\alpha\beta^2\delta^3, 4\beta^2\gamma\delta^3 + 2\alpha\beta\delta^4 , 5\beta\gamma\delta^4 + \alpha\delta^5 , \gamma\delta^5).
\end{align*}
Then $\nu$ is a $\PGL_{2, \Z[\frac{1}{2}]}$-equivariant morphism, and
$\nu_{k}$ is normalization.
Here, the action of $\PGL_{2, \Z[\frac{1}{2}]}$ on the left-hand side is given by the product of the natural action on $\P^1_{\Z[\frac{1}{2}]}$, i.e.,\ given by
\[
\begin{pmatrix}
    a & b \\
    c & d
\end{pmatrix}
\colon \begin{pmatrix}
    \alpha \\
    \beta
\end{pmatrix}
\mapsto
\begin{pmatrix}
    a & b \\
    c & d
\end{pmatrix}
\begin{pmatrix}
    \alpha\\
    \beta
\end{pmatrix}.
\]
\end{enumerate}

Moreover, each orbit contains a rational point over the prime field.
More precisely,
\begin{align*}
& (0,1,0,0,0,1,0) \in O_3^{(p)}, \\
& (0,0,0,0,0,1,0) \in O_2^{(p)}, \\
& (0,0,0,0,0,0,1) \in O_1^{(p)}.
\end{align*}
 Thus, the above statements hold also over any field $k$ of characteristic $p \neq 2$.
\end{prop}

\begin{proof}
Note that $Y$ is covered by lines (since it lifts to characteristic zero).
 By Corollary~\ref{cor:lines_split_W5} and Proposition~\ref{prop:line_stab_action}, we see that an orbit contains at least one of the three points
 \begin{align*}
& (0,1,0,0,0,1,0),\\
& (0,0,0,0,0,1,0),\\
& (0,0,0,0,0,0,1).
\end{align*}
By Proposition~\ref{prop:aut_split_W5}, we see that the three points define three different orbits.
Moreover, by Proposition~\ref{prop:aut_split_W5} again, we see that the orbit map $\PGL_{2,\Z[\frac{1}{2}]} \simeq \Aut_{Y/\Z[\frac{1}{2}]} \to \Aut_{Y/\Z[\frac{1}{2}]} \cdot (0,0,0,0,0,1,0)$ is given by
\begin{align*}
& \begin{pmatrix}
 a&b\\c&d
\end{pmatrix} \mapsto \\
&
(8ab^5 ,4b^5c + 20ab^4d,4b^4cd + 8ab^3d^2, 4b^3cd^2 + 4ab^2d^3, 4b^2cd^3 + 2abd^4 , 5bcd^4 + ad^5 , cd^5).
\end{align*}
Thus the map $\nu$ is $\PGL_{2,\Z[\frac{1}{2}]}$-equivalent.
Moreover, it is finite and birational, since the scheme theoretic fiber of $(0,0,0,0,0,1,0)$ consists of a point. 
Thus, $\nu$ is a normalization.
Also, it is easy to check that the closure of the two-dimensional orbit is $D = O_2^{(p)} \sqcup O_1^{(p)}$ (cf.\ the description in \cite[Subsection 5.1]{Kuznetsov-Prokhorov-Shramov}).
\end{proof}

\begin{rem}
When characteristic $p=3$, then the orbit map $\PGL_2 \to O_1^{(3)}$ of the point $(0,0,0,0,0,0,1)$ sends
\[
\begin{pmatrix}
    a&b\\c&d
\end{pmatrix}
\mapsto
(-b^6,0,0,-b^3d^3,0,0,d^6),
\]
whose image is a conic.
The blow-up of $Y$ along this conic is a Fano variety of No.\ 2.22 in \cite{Mori-Mukaiclassification}, \cite{Tanaka3}.
The automorphism group of a Fano variety in this class has dimension at most one in characteristic zero \cite{PCS19}, while the above example in characteristic $3$ has $\PGL_2$ as its automorphism group.
The other extremal contraction of this Fano variety is the blow-up of $\P^3$ along a special quartic rational curve in characteristic $3$ that is the image of $\P^1_{(s,t)}$ by the map
\[
(s,t) \mapsto (s^4,s^3t,st^3,t^4).
\]
Note that, in characteristic $3$, the linear system $\langle s^4,s^3t,st^3,t^4 \rangle \subset H^0(\P^1,\cO_{\P^1}(4))$ is preserved by $\PGL_2$.
\end{rem}

\subsection{Orbit decomposition in characteristic $p = 2$}
Next, we proceed to show the orbit decomposition in characteristic two.
Note that the Zariski closure $D$ of $O_2^{(0)} \sqcup O_1^{(0)} \subset Y_{\spl, \Z}$ is defined by the equation
\[
5a_2a_4-4a_1a_5+27a_0a_6.
\]
The divisor $D$ for any $p\neq 2$ is a prime divisor on $Y_{\spl}$.
On the other hand, the divisor $D$ is a double divisor on $Y_{\spl}$ defined by $a_3^2 =0$ when $p=2$.
\begin{prop}
\label{prop:char2act}
Let $Y = Y_{\spl}$ be the split $V_5$-variety over an algebraically closed field $k$ of characteristic $p=2$.
Let $D$ be the closed subscheme of $Y$ given as above.
Then the following hold.
\begin{enumerate}

\item \label{prop:char2act1}
The orbit decomposition of $Y$ with respect to $\sigma'$ is given by
\[
Y = O_{3}^{(2)} \sqcup O_2^{(2)} \sqcup O_1^{(2)} \sqcup O_1^{(2)'},
\]
where $O_i^{(2)}$ and $O_i^{(2)'}$ are $i$-dimensional orbits, $O_{1}^{(2)}$ is a twisted cubic, and $O_{1}^{(2)'}$ is the exceptional line $l_{(0,0,1)}$.

\item \label{prop:char2act2}
We have $D_{k, \red} = O_2^{(2)} \sqcup O_1^{(2)} \sqcup O_1^{(2)'}$.

\item \label{prop:char2act3}
Let $\nu'$ be the morphism 
\[
\mathrm{Bl}_{(0:0:1)} \P^2 \rightarrow (D_{\F_2})_{\red}
\]
defined by 
\[
\P^2 \dashrightarrow D_{k, \red}, (x,y,z) \mapsto (x^3 , x^2 z, x^2y,0,xy^2,y^2z,y^3).
\]
Then $\nu'$ is $\sigma'$-equivariant, and $\nu'$ is the normalization.
Here, the action of $\SL_{2,k}$ on the domain is induced from \eqref{eq:SL2_PGL3}.
In particular, $O_1^{(2)}$ is
the image of 
\[
\nu'|_{\P^1} \colon \P^1 = \{z =0 \} \rightarrow D_{k, \red}, (x:y) \mapsto (x^3,0,x^2y,0,xy^2,0,y^3),
\]
which is a twisted cubic, 
and $O_{1}^{(2)'}$ is the image of
\[
\nu'|_{E} \colon \P^1 \rightarrow D_{k,\red}, (x,y) \mapsto (0,x^2,0,0,0,y^2,0),
\]
which is a line $l_{(0,0,1)}$.
Moreover, $\nu'$ is a bijection.
 \end{enumerate}
Each orbit contains a rational point over $\F_2$.
 More precisely, we have
 \begin{align*}
&(1,0,0,1,0,0,1) \in O_3^{(2)},\\
&(0,0,0,0,0,1,1) \in O_2^{(2)},\\
&(0,0,0,0,0,0,1) \in O_1^{(2)},\\
&(0,0,0,0,0,1,0) \in O_1^{(2)'}.
\end{align*}
Hence, the above statements hold also over any field $k$ of characteristic two.
\end{prop}

\begin{proof}
Note again that $Y$ is covered by lines.
By Corollary~\ref{cor:lines_split_W5} and Proposition~\ref{prop:line_stab_action}, we see that an orbit contains at least one of the four points
\begin{align*}
&(1,0,0,1,0,0,1),\\
&(0,0,0,0,0,1,1),\\
&(0,0,0,0,0,0,1),\\
&(0,0,0,0,0,1,0).
\end{align*}
By Proposition~\ref{prop:aut_split_W5}, we see that these four points define four different orbits $O_{3}^{(2)} \sqcup O_2^{(2)} \sqcup O_1^{(2)} \sqcup O_1^{(2)'}$.

By seeing $5$-th column and $6$-th column of (\ref{eqn:sigma'}) we can show the following:
$O_{1}^{(2)'}$ is the line $l_{(0,0,1)}$.
Also, $O_1^{(2)}$ is a twisted cubic defined by
\begin{equation*}
    \begin{gathered}
        \P^1_{\F_2} \rightarrow Y_{\F_2} \subset \P^6_{\F_2},\\
        (s,t) \mapsto (s^3,0,s^2t,0,st^2,0,t^3).
    \end{gathered}
\end{equation*}
Moreover $D_{\F_2, \red}$ is $O_2^{(2)} \sqcup O_1^{(2)} \sqcup O_1^{(2)'}$.

Next, we show \ref{prop:char2act3}.
By the direct computation, $\nu'$ is a well-defined finite morphism whose restrictions to $\{z=0\}$ and $E$ are as in \ref{prop:char2act3}.
Also, it is $\sigma'$-equivariant.
Moreover, we can check that $\nu'$ is birational since the scheme theoretic fiber of the point $(0,0,0,0,0,1,1)$ consists of a point $(0,1,1)$.
Thus, $\nu'$ is a normalization.
Since we have
\[
\nu' (\mathrm{Bl}_{(0:0:1)} \P^2_{\F_{2}} \setminus (\{z=0\} \cup E)) \subset O_2^{(2)},
\]
and $O_2^{(2)} \subset D_{\F_2, \red}$ is smooth, combining with the description of $\nu'|_{\{z=0\}}$ and $\nu'|_{E}$, $\nu'$ is a bijection.
\end{proof}

\begin{rem}
\label{rem:normalization}
Set $\cO_{K} \coloneqq \Z[\sqrt{2}]$, and $K \coloneqq \Frac \cO_{K}$.
Let 
\[
\nu\colon \P^1_{\cO_K[\frac{1}{2}]} \times_{\cO_K[\frac{1}{2}]} \P^1_{\cO_K[\frac{1}{2}]} \rightarrow D_{\cO_K[\frac{1}{2}]}
\]
be the base change of the morphism defined in Proposition \ref{prop:generalW5}.
After taking a conjugation
$
c \colon \PGL_{2, \cO_K[\frac{1}{\sqrt 2}]} \simeq \PGL_{2, \cO_K[\frac{1}{\sqrt 2}]}
$
by the element
\[
\begin{pmatrix}
1 &0\\
0 & \sqrt{2}\\
\end{pmatrix},
\]
we can show that $\nu$ can be extended to a morphism
$
\widetilde{\nu} \colon \mathcal{P} \rightarrow D,
$
where $\mathcal{P}$ is a smooth projective scheme over $\cO_K$ such that
\[
\mathcal{P}_{\cO_K[\frac{1}{2}]} \simeq \P^1_{\cO_K[\frac{1}{2}]} \times_{\cO_K[\frac{1}{2}]} \P^1_{\cO_K[\frac{1}{2}]},
\]
$\mathcal{P}_{\F_{2}}$ is isomorphic to the Hilzebruch surface $\Sigma_{2}$ over $\F_2$, and $\widetilde{\nu}_{\F_2}$ factors through $\nu'$ via the unramified covering
\[
\Sigma_{2} \rightarrow \Sigma_{1}
\]
induced by the Frobenius morphism on the base scheme $\P^1_{\F_{2}}$.
In this sense, $\nu$ and $\nu'$ are related.
\end{rem}

\begin{rem}
The blow-up of $Y$ along the twisted cubic $O_1^{(2)}$ is a Fano variety of No.\ 2.20 in \cite{Mori-Mukaiclassification}, \cite{Tanaka3}.
Similarly to the case of characteristic three,
the automorphism group of a Fano variety in this class has dimension at most one in characteristic zero \cite{PCS19}, while this example has $\SL_2$ as its automorphism group.
The other extremal contraction of this Fano variety is a conic bundle over $\P^2$ whose discriminant divisor is a triple line.

The blow-up of $Y$ along the line $O_1^{(2)'}$ is a Fano variety of No.\ 2.26 in \cite{Mori-Mukaiclassification}, \cite{Tanaka3}.
The automorphism group of a Fano variety in this class has dimension at most two in characteristic zero \cite{PCS19}, while this example also has $\SL_2$ as its automorphism group.
The other extremal contraction of this Fano variety is the blow-up of the three-dimensional quadric $Q$ along the inseparable ordinary twisted cubic as in Remark~\ref{rem:W5_p=2}.
\end{rem}

\subsection{Lines and orbits on $V_5$-varieties in characteristic $p\neq 2$}

\begin{thm}
\label{thm:linep>2}
Let $k$ be an algebraically closed field of characteristic $p \neq 2$.
Let $Y$, $D$, $\nu$, $O_i^{(p)}$ and $\sigma$ be as in Theorem~\ref{prop:generalW5}.
Let $L \subset Y$ be a line.
Then the following hold.
\begin{enumerate}
\item \label{thm:linep>2_1}
Suppose $L \not\subset D$.
Then $L$ is an ordinary line.
Moreover, the scheme theoretic intersection $L \cap D$ consists of two distinct points $\nu(P,Q)$ and $\nu(Q,P)$ ($P\neq Q$) in $ O_{2}^{(p)} (k)$, and we have a bijection
\[
\begin{gathered}
\{[L] \in \Sigma(Y) (k) \mid L \not\subset D\} \simeq \{\textup{distinct two points $\nu(P,Q)$ and $\nu(Q,P)$ in $O_{2}^{(p)} (k)$}\},\\
[L] \mapsto L \cap D.
\end{gathered}
\]

\item \label{thm:linep>2_2}
Suppose $L \subset D$. Then $L$ is a special line.
Moreover, the set theoretic intersection $L \cap O_1^{(p)}$  is a point of $O_{1}^{(p)} (k)$, and we have a bijection 
\[
\begin{gathered}
\{[L] \in \Sigma(Y) (k) \mid L \subset D\}\simeq O_{1}^{(p)} (k), \\
L \mapsto (L \cap O_{1}^{(p)})_{\mathrm{red}}.
\end{gathered}
\]

\item \label{thm:linep>2_3}
Consider the $\PGL_2$-equivariant isomorphism $S^2\P^1_{[X,Y]} \to \P^2 =\P(N_3)$ given by
\[
(\alpha,\gamma), (\beta, \delta) \mapsto (2 \alpha \beta , \gamma \delta, -(\alpha \delta + \beta \gamma)),
\]
whose converse is 
given by
\begin{eqnarray*}
(x , y , z) \mapsto 
V(y X^2+ zXY + \frac{1}{2}xY^2).
\end{eqnarray*}
Then the composite 
\[
S^2 \P^1_{[X:Y]} \simeq \P^2  \rightarrow \Sigma (Y)
\]
sends
\[
(P, Q) \mapsto \textup{the line passing throuhgh $\nu(P, Q)$ and $\nu (Q, P)$}.
\]
Note that, when $P =Q$, the line on $Y$ passing through $\nu (P,Q)$ is unique.

\item \label{thm:linep>2_4}The locus of lines as in (2) in $\P^2 \simeq S^2\P^1 \simeq \Sigma (Y)$ is equal to the conic
 \begin{equation}
 \label{eqn:speciallinelocus}
 V(z^2-2xy) \subset \P^2_{[x:y:z]},
\end{equation}
which is the set of double points in $S^2\P^1$.
\end{enumerate}
\end{thm}

\begin{proof}
First, we shall show \ref{thm:linep>2_1}.
By Corollary~\ref{cor:lines_split_W5}, we may assume $L=l_{(0,0,1)}$.
In this case, $L \cap W$ consists of two points
\[
\nu ((1,0), (0,1))=(0,1,0,0,0,0,0), \quad  \nu ((0,1), (1,0)) = (0,0,0,0,0,1,0).
\]
Since $\nu$ is $\PGL_{2,k}$-equivariant and we have
\[
\P^1 \times \P^1 \setminus \Delta (\P^1) = \PGL_{2,k} ((1,0), (0,1)),
\]
the remaining part of \ref{thm:linep>2_1} follows.

Next, we shall show \ref{thm:linep>2_2}.
By Corollary~\ref{cor:lines_split_W5}, we may assume that $L$ is equal to $l_{(0,1,0)}$.
In this case, the scheme theoretic intersection $L \cap O_1^{(p)}$ is a double point supported on
\[
  (0,0,0,0,0,0,1).
 \]
 Moreover, by using the explicit equation \eqref{eqn:W5}, the line on $Y$ passing through $(0,0,0,0,0,0,1)$ is unique.
 Therefore, the remaining part of \ref{thm:linep>2_2} follows.

Next, we show \ref{thm:linep>2_3}.
We identify $\P^2$ with $S^2 \P^1$ by the isomorphism in the statement.
Then the given maps are $\PGL_2$-equivariant, and hence the assertion follows from \ref{thm:linep>2_1} and \ref{thm:linep>2_2}.
Now \ref{thm:linep>2_4} follows immediately.
\end{proof}

\begin{prop}
\label{prop:numberoflinep>2}
In the setting of Theorem \ref{thm:linep>2}, the following hold.
\begin{enumerate}
\item \label{prop:numberoflinep>2_1}
Let $P$ be a point in $O_{3}^{(p)}(k)$.
Then the lines passing through $P$ are three ordinary lines.
\item  \label{prop:numberoflinep>2_2}
Let $P$ be a point in $O_{2}^{(p)}(k)$.
Then the lines passing through $P$ are one ordinary line and one special line.
\item  \label{prop:numberoflinep>2_3}
Let $P$ be a point in $O_{1}^{(p)}(k)$.
Then the line passing through $P$ is unique, which is a special line.
\item \label{prop:numberoflinep>2_4}
The ordinary lines $L$ (resp.\ special lines) are characterized by the following property:
\begin{itemize}
    \item 
    There (resp.\ do not) exist distinct lines $L_1$ and $L_2$ such that $L \cap L_1 \cap L_2$ consists of a 1 point.
\end{itemize}
\end{enumerate}
\end{prop}

\begin{proof}
To prove \ref{prop:numberoflinep>2_1}, we may assume that 
\[
P= (0,-4,0,0,0,1,0).
\]
Using equations (\ref{eqn:W5}), we can show that the lines passing through $P$ are the following three lines.
 \begin{itemize}
     \item 
     The line $l_{(0,0,1)}$.
     \item 
    The lines $l_{(2,\mp1,0)}$ defined by
     \begin{equation}
     \label{eqn:line15conj}
         \begin{gathered}
             \P^1 \rightarrow Y \subset \P^6,\\
 (s, t) \mapsto (8s, -4t, \pm 4s, 0, -2s, t, \mp s).
         \end{gathered}
     \end{equation}
 \end{itemize}
Thus, \ref{prop:numberoflinep>2_1} follows.

To prove \ref{prop:numberoflinep>2_2}, we may assume that
\[
P = (0,0,0,0,0,1,0),
\]
By using the explicit equation (\ref{eqn:W5}), we can show that the lines on $Y$ passing through $P$ are the following two lines.
\begin{itemize}
 \item 
 The line $l_{(0,0,1)}$,
 \item 
 The line $l_{(0,1,0)}$.
\end{itemize}
Thus, \ref{prop:numberoflinep>2_2} follows.

\ref{prop:numberoflinep>2_3} follows from Theorem \ref{thm:linep>2}.
To elaborate further, we may assume that
\[
P= (0,0,0,0,0,0,1),
\]
and then the lines passing through $P$ is the line $l_{(0,1,0)}$.
The first statement of \ref{prop:numberoflinep>2_4} follows from \ref{prop:numberoflinep>2_1}--\ref{prop:numberoflinep>2_3}. 
The remaining statement follows from Theorem \ref{thm:linep>2}.
\end{proof}

\begin{rem}
\label{rem:lineclassification}
The point $(0,-4,0,0,0,1,0)$ corresponds to $[\phi_6]$ in \cite[Subsection 5.1]{Kuznetsov-Prokhorov-Shramov}.
\end{rem}

\subsection{Lines and orbits on $V_5$-varieties in characteristic $p=2$}

\begin{thm}

\label{thm:linep=2}
Let $k$ be an algebraically closed field of characteristic $2$.
Let $Y$, $D$, $\nu'$, $O_i^{(2)}$, $O_1^{(2)'}$ and $\sigma'$ be as in Theorem \ref{prop:char2act}.
Let $L \subset Y$ be a line.
Then the following hold.
\begin{enumerate}
\item \label{thm:linep=2_1}
Suppose that $L \not\subset D_{\red}$.
 Then $L$ is ordinary.
Moreover, the intersection $L \cap D_{\red}$ consists of a single point in $ O_{2}^{(2)} (k)$, and we have a bijection 
\[
\begin{gathered}
\{[L] \in \Sigma(Y) (k) \mid L \not\subset D \}\simeq O_2^{(2)} (k), \\
L \mapsto  (L \cap D_{\red}).
\end{gathered}
\]
\item \label{thm:linep=2_2}
Suppose that $L \subset D$. 
Then one of the following holds.
\begin{enumerate}
\item \label{thm:linep=2_2a}
The line $L$ is special.
Moreover, the intersection $L \cap O_{1}^{(2)}$ consists of a single point, and we have a bijection 
\[
\begin{gathered}
\{[L] \in \Sigma(Y)  (k) \mid \textup{line as in (a)}\}\simeq  O_{1}^{(2)} \\
L \mapsto  (L \cap O_{1}^{(2)}).
\end{gathered}
\]
\item \label{thm:linep=2_2b}
The line $L$ is the exceptional line $O_{1}^{(2)'} = l_{(0,0,1)}$.
\end{enumerate}
\end{enumerate}
The isomorphism $\P(N_3) = \P^2 \simeq \Sigma (Y)$ is compatible with the following map:
\[
(x , y , z) \mapsto 
\begin{cases}
\textup{The line $\ni \nu' (x, y , z)$ with $L \not\subset D_{\red}$} & \textup{if $z \neq 0$ and $(x, y)\neq 0$,}\\
\textup{The line $\ni \nu' (x,y,z)$} & \textup{if $z =0$,}\\
\textup{The line $l_{(0,0,1)}$} & \textup{if $(x,y,z) = (0,0,1)$}. 
\end{cases}
\]
\end{thm}

\begin{proof}
First, we shall show \ref{thm:linep=2_1}.
By Corollary~\ref{cor:lines_split_W5}, we may assume that $L$ is $l_{(1,1,1)}$.
In this case, $L \cap D_{\red}$ consists of a point
\[
\nu' (1,1,1) = (1,1,1,0,1,1,1).
\]
Also, we can check that $L$ is the only line that is not contained in $D$ and passing through $(1,1,1,0,1,1,1)$.
Therefore, the remaining part of \ref{thm:linep=2_1} follows.

Next, we shall show \ref{thm:linep=2_2}.
By Corollary~\ref{cor:lines_split_W5}, we may assume that $L$ is equal to $l_{(0,1,0)}$.
In this case, the scheme theoretic intersection $L \cap O_1^{(2)}$ is a single point
\[
\nu' (0,1,0) = (0,0,0,0,0,0,1).
\]
Conversely, we can check that $L$ is the only line on $Y$ passing through $\nu'(0,1,0)$.
Therefore, the remaining part of \ref{thm:linep=2_2} follows.

The final statement follows easily.
\end{proof}

\begin{prop}
\label{prop:numberoflinep=2}
In the setting of Theorem \ref{thm:linep=2}, the following hold.
\begin{enumerate}
    \item \label{prop:numberoflinep=2_1}
    Let $P$ be a point in $O_3^{(2)}$. 
    Then the lines passing through $P$ are three ordinary lines.
    \item \label{prop:numberoflinep=2_2}
    Let $P$ be a point in $O_2^{(2)}$.
    Then the lines passing through $P$ are one ordinary line and one special line.
    \item\label{prop:numberoflinep=2_3}
    Let $P$ be a point in $O_1^{(2)}$.
    Then the line passing through $P$ is unique, which is special. 
    \item \label{prop:numberoflinep=2_4}
    Let $P$ be a point in $O_1^{(2)'}$.
    Then the lines passing through $P$ are one special line and one exceptional line.
    \item \label{prop:numberoflinep=2_5}
    The ordinary lines $L$ 
    are characterized by the following property:
\begin{itemize}
    \item 
    There exist distinct lines $L_1$ and $L_2$ such that $L \cap L_1 \cap L_2$ consists of a single point.
\end{itemize}
The special lines $L$ are characterized by the following property:
\begin{itemize}
\item 
    there exist a $P \in Y(k)$ such that the lines passing through $P$ are only one line $L$.
\end{itemize}
The exceptional line $L$ is characterized by the property not satisfying the above conditions for ordinary and special lines.

\end{enumerate}
\end{prop}

\begin{proof}
To prove \ref{prop:numberoflinep=2_1}, we may assume that
\[
P = (1,0,0,1,0,0,1).
\]
By direct calculation using the explicit equation (\ref{eqn:W5}), we can show that the lines on $Y$ passing through $P$ are the following lines.
\begin{itemize}
\item 
The line $l_{(1,1,1)}$.
\item 
The line $l_{(\omega^2,\omega,1)}$, which is 
defined by
\begin{equation}
\label{eqn:lineomega}
\begin{gathered}
    \P^1 \rightarrow Y \subset \P^6,\\
    (s,t) \mapsto (s,t, t\omega^2 , s+t\omega , t, t \omega^2, s)
\end{gathered}
\end{equation}
\item 
The line $l_{(\omega,\omega^2,1)}$, which is 
defined by
\begin{equation}
\label{eqn:lineomega2}
\begin{gathered}
    \P^1 \rightarrow Y \subset \P^6,\\
    (s,t) \mapsto (s,t, t\omega , s+t\omega^2 , t, t \omega, s)
\end{gathered}
\end{equation}
\end{itemize}
Here, $\omega$ is a fixed primitive cube root of $1$.
Thus, the assertion follows.

To prove \ref{prop:numberoflinep=2_2}, we may assume that
\[
P = (0,0,0,0,0,1,1).
\]
By using the explicit equations (\ref{eqn:W5}), we can show that the lines passing through $\nu' (0,1,1)$ are the following lines.
\begin{itemize}
    \item 
    The line $l_{(0,1,0)}$.
    \item 
    The line $l_{(0,1,1)}$ defined by
    \begin{equation}
    \label{eqn:ssssts+t}
        \begin{gathered}
            \P^1 \rightarrow Y \subset \P^6, \\
            (s,t) \mapsto (0,s,s,s,s,t,s+t).
        \end{gathered}
    \end{equation}
\end{itemize}
The former (resp.\ the latter) is of type Theorem \ref{thm:linep=2} \ref{thm:linep=2_2}\ref{thm:linep=2_2a} (resp.\ \ref{thm:linep=2_1}), and the assertion follows.

To prove \ref{prop:numberoflinep=2_3}, we may assume that
\[
P = (0,0,0,0,0,0,1),
\]
and then we can check that $l_{(0,1,0)}$ is the only line on $Y$ passing through $P= \nu'(0,1,0)$.

Next, we show \ref{prop:numberoflinep=2_4}.
We may assume that 
\[
P = (0,0,0,0,0,1,0).
\]
In this case, by direct calculation using (\ref{eqn:W5}), we can show that the lines passing through $P$ are the following two lines $l_{(0,1,0)}$ and $l_{(0,0,1)}$.

Finally, we show \ref{prop:numberoflinep=2_5}.
The characterization of lines follow from \ref{prop:numberoflinep=2_1}--\ref{prop:numberoflinep=2_4}, and the remaining statements follow from Theorem \ref{thm:linep=2}.
\end{proof}

\subsection{The variety of trisecant lines to the projected Veronese surface}
Recall that, over $\C$, a $V_5$-variety $Y$ is described as the variety of trisecant lines to a projected Veronese surface, or similarly, $Y$ is the variety of sums of powers to a conic \cite{Muk92,Ili94}:
Let $\P^2 \subset \P^5$ be the Veronese surface, and take a projection $\P^2 \subset \P^4$ from a general point in $\P^5$.
Then
\[
\{ l \subset \P^4 \mid \text{$l$ is a trisecant line to $\P^2 \subset \P^4$}\} \subset \Gr(5,2)
\]
(with its reduced structure) is a $V_5$-variety (see also \cite{DK19} for a similar result over non-closed fields of characteristic zero).
Here, we briefly mention that an analogous statement also holds over \emph{reduced} base scheme $B$.
Note that, due to the geometric nature of this description, we can formulate this only for a reduced scheme $B$.

\begin{thm}\label{thm:trisecant}
 Let $(N_3,L,\varphi)$ be a non-degenerate ternary symmetric bilinear form over a \emph{reduced} scheme $B$, and 
 $\P(N_3) \subset \P(\ker \varphi)$ be the projected Veronese surface.
 Then the closure of the following set
\begin{equation}\label{eq:trisecant}
 \{ [L] \in \Gr(\ker \varphi, 2) \mid \text{$\P(L) \cap \P(N_3)$ consists of three points}\} \subset \Gr(\ker \varphi, 2)
\end{equation}
 with its reduced structure gives a $V_5$-scheme over $B$.
 Moreover, the correspondence is compatible with Theorem~\ref{thm:W5description}.
\end{thm}

\begin{proof}
Let $Y$ be the $V_5$-scheme corresponding to $\varphi$ as in Theorem~\ref{thm:W5description}.
Then $Y$ is a reduced (since $B$ is reduced), and is also contained in $\Gr(\ker \varphi,2)$.
Thus, it is enough to see that $Y$ is the closure of \eqref{eq:trisecant} as a set.
 By taking base change over geometric points, we are reduced to proving this for the symmetric bilinear form $\varphi_{\spl}$ over an algebraically closed field $k$.
 Note that a point in \eqref{eq:trisecant}, or the triple of intersection points $\P(L) \cap \P(N_3)$, corresponds to a basis of $N_3$ such that the form $\varphi_{\spl}$ is diagonalized with respect to it.
 Since any two such bases are translated to each other by the group as in Proposition~\ref{prop:aut_split_W5},
 we see that \eqref{eq:trisecant} is irreducible and homogeneous.
Therefore, it is enough to show, for a general point $p \in Y_{\spl}$, the line corresponding to $p$ is in  \eqref{eq:trisecant}.
This has been essentially done in the proof of Propositions~\ref{prop:numberoflinep>2} and \ref{prop:numberoflinep=2}.
The details are as follows.

Assume that the characteristic of $k$ is not $2$.
Take a point $(0,-4,0,0,0,1,0) \in Y_{\spl}$.
This is the point in $\Gr(U_5^\vee,2)$ corresponding to the two-dimensional subspace in $U_5 = k^5$ spanned by vectors
\[
\begin{pmatrix}
 0\\0\\2\\0\\0
\end{pmatrix},
\begin{pmatrix}
 4\\0\\0\\0\\1
\end{pmatrix}.
\]
This defines a line $\{ (4s,0,2t,0,s)\} \subset \P(U_5^\vee)$.
By Remark~\ref{rem:projection_veronese}, we can show that this line meets the projected Veronese surface with three points
\[
(0,0,1,0,0), \quad (4,0,-2,0,1), \quad (4,0,2,0,1).
\]
(These points correspond to the three lines $l_{(0,0,1)}$, $l_{(2,\mp 1,0)}$.)

Assume that the characteristic of $k$ is $2$, and take a point $(1,0,0,1,0,0,1) \in Y_{\spl}$.
This is the point in $\Gr(U_5^\vee,2)$ corresponding to the two-dimensional subspace spanned by
\[
\begin{pmatrix}
 1\\0\\0\\1\\0
\end{pmatrix},
\begin{pmatrix}
 0\\1\\0\\0\\1
\end{pmatrix}.
\]
This defines a line $\{ (s,t,0,s,t)\} \subset \P(U_5^\vee)$.
By Remark~\ref{rem:projection_veronese}, we can show that this line intersects with the projected Veronese surface with three points
\[
(1,1,0,1,1), \quad (\omega^2,\omega,0,\omega^2,\omega), \quad (\omega,\omega^2,0,\omega,\omega^2).
\]
(These points correspond to the three lines $l_{(1,1,1)}$, $l_{(\omega,\omega^2,1)}$, $l_{(\omega^2,\omega,1)}$.)
\end{proof}

\begin{rem}
Here we briefly describe the relation between the trisecant lines to $\P^2 \subset \P^4$, and the orbit decomposition of $Y$.
\begin{enumerate}
 \item Suppose that the characteristic of $k$ is not $2$.
 Then $\P(N_3)$ is decomposed into two orbits; the conic $C$ defined by $\varphi$, and the open orbit.
\begin{enumerate}
 \item A general point $P \in O_3^{(p)}$ corresponds to a trisecant line $l$ such that the intersection $l \cap \P(N_3)$ consists of three points in the open orbit.
\item A point $P \in O_2^{(p)}$ corresponds to a trisecant line $l$ such that the intersection $l \cap \P(N_3)$ consists of a double point whose support is contained in $C$, and  a single point in the open orbit.
\item A point $P \in O_1^{(p)}$ corresponds to a trisecant line $l$ such that the intersection $l \cap \P(N_3)$ consists of a triple point whose support is contained in $C$.
\end{enumerate}
 \item Suppose that the characteristic of $k$ is $2$.
 Note that $\P(N_3)$ is decomposed into three orbits; the point $\{(0,0,1)\}$, the line $\{z=0\}$, and the open orbit.
\begin{enumerate}
 \item A general point $P \in O_3^{(2)}$ corresponds to a trisecant line $l$ such that the intersection $l \cap \P(N_3)$ consists of three points contained in the open orbit.
\item A point $P \in O_2^{(2)}$ corresponds to a trisecant line $l$ such that the intersection $l \cap \P(N_3)$ consists of a double point whose support is contained in the line $\{z=0\}$, and a single point not contained in the line $\{z=0\}$.
\item A point $P \in O_1^{(2)}$ corresponds to a trisecant line $l$ such that the intersection $l \cap \P(N_3)$ consists of a triple point whose support is contained in the line $\{z=0\}$.
\item A point $P \in O_1^{(2)'}$ corresponds to a trisecant line $l$ such that the intersection $l \cap \P(N_3)$ consists of a double point whose support is the point $(0,0,1)$, and a single point in the line $\{z=0\}$.
\end{enumerate}
\end{enumerate}
\end{rem}

\section{Arithmetic applications}
\label{section:arithmetic}

In this section, we give several arithmetic applications of our classification theorem.

\subsection{Good reduction of $V_5$-varieties}

First, we prove a useful proposition that guarantees the irrelevance of completion when using the term “good reduction” for the varieties of interest in this paper.
Note that the second Betti number $b_2(Y)$ is one for a $V_5$-variety $Y$ (since $Y_{\spl}$ lifts to characteristic zero).

\begin{prop}
\label{prop:grcompleteFanobetti1}
Let $R$ be an excellent discrete valuation ring, $K$ the fraction field of $R$, and $k$ the residue field of $R$. 
We denote the completion of $R$ by $\widehat{R}$ and the fraction field of $\widehat{R}$ by $\widehat{K}$.
Let $X$ be a (smooth) Fano variety over $K$ with the second Betti number $b_2 (X) =1$ (e.g., a $V_5$-variety).
Suppose that there exists a smooth projective scheme $\widehat{\mathcal{X}}$ over $\widehat{R}$ such that $\widehat{\mathcal{X}}_{\widehat{K}} \simeq X_{\widehat{K}}$.
Then there exists a unique smooth projective scheme $\mathcal{X}$ over $R$ with $\mathcal{X}_K \simeq X$ such that $\mathcal{X}_{\widehat{R}}$ is isomorphic to $\widehat{\mathcal{X}}$.
Moreover, the special fiber $\mathcal{X}_k = \widehat{\mathcal{X}}_k$ is a smooth Fano variety over $k$.
\end{prop}

\begin{proof}
Since we have $b_2 (\widehat{\mathcal{X}}_k) = b_2 (X_{\widehat{K}})=b_2 (X) =1$, the divisor $-K_{\widehat{\mathcal{X}}_k}$ is ample, i.e.,\ $\widehat{\mathcal{X}}_k$ is a smooth Fano variety.
Therefore, it is sufficient to show the statements other than the last one.
By \cite[Proposition 3.12]{Achinger-Youcis}, there exists a unique smooth proper algebraic space $\pi \colon \mathcal{X} \rightarrow R$ satisfying the conditions in the statement. 
It suffices to show that $\mathcal{X}$ is a smooth projective scheme over $R$.
Let $\omega_{\mathcal{X}/R}$ be $\wedge^n\Omega_{\mathcal{X}/R}$, which is a locally free sheaf on $\mathcal{X}$.
Then $\omega_{\mathcal{X}/R}^{-1}$ is a locally free sheaf that is ample on the generic and the special fiber.
Therefore, $\mathcal{X}$ is a scheme and $\omega_{\mathcal{X}/R}^{-1}$ is ample over $R$. It finishes the proof.
\end{proof}

Next, we define the notion of good reduction for $V_5$-varieties.

\begin{defn}
\label{defn:W5goodreduction}
Let $R$ be a discrete valuation ring, and $K$ the fraction field.
Let $Y$ be a $V_5$-variety over $K$.
We say \emph{$Y$ admits good reduction at $R$} when
there exists a $V_5$-scheme $\mathcal{Y}$ over $R$ such that $\mathcal{Y}_K \simeq Y$.
When $Y_{\Frac \widehat{R}}$ admits good reduction at $\widehat{R}$, we say \emph{$Y$ admits good reduction at $\widehat{R}$} simply.
\end{defn}

\begin{rem}
\label{rem:W5goodreduction}
Let $R, K$ be as in Definition \ref{defn:W5goodreduction} and $k$ the residue field of $R$.
Let $Y$ be a $V_5$-variety over $K$.
Suppose that there exists a smooth projective scheme $\mathcal{Y}$ over $R$ with $\mathcal{Y}_K \simeq Y$.
Then $\mathcal{Y}$ is a $V_5$-scheme over $R$ and hence $Y$ admits good reduction at $R$.
Indeed, by the smooth proper base change theorems, $b_2 (\mathcal{Y}_{k})=1$ and hence $\mathcal{Y}_{k}$ is a Fano variety of Picard rank 1 whose index is divisible by $2$.
Since $(-K_{\mathcal{Y}_{k}})^3 =40$, the index is equal to 2. Therefore, $\mathcal{Y}$ is a $V_5$-scheme over $R$. 
\end{rem}

We study the good reduction of $V_5$-varieties over $p$-adic fields.
When $p$ is odd, $V_5$-varieties share similar properties with Brauer--Severi varieties (e.g.\ conics),
but they behave differently when $p = 2$;
it turns out that every $V_5$-variety over a $2$-adic field has good reduction.
More precisely, we prove the following results.

\begin{prop}
\label{prop:goodreductioncriteriaW5odd}
Let $p$ be a prime number.
Let $K$ be a $p$-adic field (i.e.,\ a finite extension of $\Q_p$), and $\cO_K$ the integer ring of $K$.
\begin{enumerate}
\item There are exactly two $K$-isomorphism classes of $V_5$-varieties over $K$.
\item Assume $p \neq 2$. Then a $V_5$-variety $Y$ over $K$ admits good reduction at $\cO_K$
if and only if $Y$ is a split $V_5$-variety (i.e.,\ $Y \simeq Y_{\spl, K})$.
\item Assume $p = 2$. Then every $V_5$-variety $Y$ over $K$ admits good reduction at $\cO_K$.
\end{enumerate}
\end{prop}

\begin{proof}
(1)
By Theorem \ref{thm:W5description}, the assertion follows from the fact that
there is a unique $K$-similarity class of non-split ternary symmetric bilinear forms over $K$
(see \cite[Theorem 12.3.4]{Voight21}).

(2)
Since the split $V_5$-variety $Y_{\spl,K}$ is defined over $\Z$,
it admits good reduction at $R$.
Conversely, assume that $Y$ admits good reduction at $R$.
There exists a $V_5$-scheme $\mathcal{Y}$ over $\cO_K$ with $\mathcal{Y}_K \simeq Y$.
Let $s$ be the non-degenerate symmetric bilinear form on $\cO_K$ corresponding to $\mathcal{Y}$ by Theorem \ref{thm:W5description}.
Since $s$ is split (see \cite[Proposition 12.3.8]{Voight21}).
we see that $\mathcal{Y}$ is a split $V_5$-scheme, and so is $Y$.

(3)
It is enough to show that every non-split $V_5$-variety admits good reduction at $\cO_K$.

Recall that the Hilbert symbol gives a non-degenerate symmetric bilinear form
\[
  (,)_K \colon K^{\times}/K^{\times 2} \times K^{\times}/K^{\times 2} \to \{ \pm 1 \}.
\]
By \cite[Chapter II, Proposition 5.7]{Neukirch:Algebraicnumbertheory},
we have an isomorphism
\[ K^{\times} \, \simeq \, \Z \oplus \Z/(q-1)\Z \oplus \Z/2^r\Z \oplus \Z_2^{\oplus [K : \Q_2]}, \]
where $q$ denotes the number of elements in the residue field of $K$.
Here $\Z/(q-1)\Z \oplus \Z/2^r\Z$ is isomorphic to the group of torsion elements in $\cO_K^{\times}$.
Since $\{ \pm 1 \} \subset \cO_K^{\times}$, we have $r \geq 1$.
From this, we see that $K^{\times}/K^{\times 2}$ is a $([K : \Q_2] + 2)$-dimensional vector space over $\F_2$,
and $\cO_K^{\times}/\cO_K^{\times 2}$ is a codimension $1$ subspace.
Since the Hilbert symbol $(,)_K$ is non-degenerate and $\dim_{\F_2} ([K : \Q_2] + 2) \geq 3$,
the restriction of $(,)_K$ to $\cO_K^{\times}/\cO_K^{\times 2}$ is non-trivial.
Therefore, there exist $a,b \in \cO^{\times}$ such that $(a,b)_K = -1$.

Let $s$ be the non-degenerate ternary symmetric bilinear form over $K$
defined by the matrix
$\begin{pmatrix}a&0&0 \\ 0& b&0 \\ 0& 0& -1 \end{pmatrix}$.
Since $(a,b)_K = -1$,
the equation $a x^2 + b y^2 = 1$ has no non-trivial solution in $K$.
Hence $s$ is not split over $K$.

Let $Y$ be a non-split $V_5$-variety over $K$.
By (1), $Y$ is isomorphic to the $V_5$-variety over $K$ corresponding to $s$ by Theorem \ref{thm:W5description}.
Since $a,b \in \cO^{\times}$, the symmetric bilinear form $s$ over $K$ extends to
a non-degenerate ternary symmetric bilinear form over $\cO_K$.
Therefore, by Theorem \ref{thm:W5description},
we obtain a $V_5$-scheme $\mathcal{Y}$ over $\cO_K$ with $\mathcal{Y}_K \simeq Y$ as desired.
\end{proof}

\begin{rem}
In the proof of (3), we use the fact that the restriction of $(,)_K$ to $\cO_K^{\times}/\cO_K^{\times 2}$
is non-trivial for a $2$-adic field $K$.
On the other hand, when $p$ is odd,
the restriction of $(,)_K$ to $\cO_K^{\times}/\cO_K^{\times 2}$
is trivial for any $p$-adic field $K$
(see \cite[(12.4.9)]{Voight21}).
\end{rem}

\subsection{Explicit Shafarevich conjecture for $V_5$-varieties}

Recall that a \emph{Brauer--Severi variety} (e.g.\ a conic) over a number field $F$
admits good reduction at a finite place $\p$ if and only if
the associated central simple algebra is split at $\p$.
It is well-known that this leads to the explicit version of the Shafarevich conjecture for Brauer--Severi varieties.
In this section, we discuss analogous results for $V_5$-varieties.

Let $F$ be a finite field, $\cO_F$ the ring of integers of $F$.
For a finite place $\p$ of $F$, $\cO_{F,\p}$ denotes the localization at $\p$.
We say a $V_5$-variety over $F$ \emph{admits good reduction at $\p$} if it admits good reduction
at $\cO_{F,\p}$ (see Definition \ref{defn:W5goodreduction}).
By Proposition \ref{defn:W5goodreduction},
this property does not change if we replace $K$ by the completion $K_{\p}$,
i.e.,\ $Y$ admits good reduction at $\p$ if and only if
$Y_{K_{\p}}$ admits good reduction at
the completion $\widehat{\cO}_{F,\p}$ of $\cO_{F,\p}$.

For a number field $F$ and a finite set $S$ of finite places of $F$,
the set of $F$-isomorphism classes of $V_5$-varieties over $F$ which admit good reduction at every finite place outside $S$ is written as
\[
    \Shaf_{F,S} \coloneqq \left\{
    \, Y \left|
    \begin{array}{l}
      \text{\rm $Y$ is a $V_5$-variety over $F$ that admits} \\
      \text{\rm good reduction at $\p$ for every  $\p \notin S$}
    \end{array}
    \right\}\right.
    /\simeq_F.
\]
Note that we do not fix an integral model in the definition of this set.
Javanpeykar--Loughran proved that $\Shaf_{F,S}$ is a finite set,
but they did not give calculate $\# \Shaf_{F,S}$ explicitly
(see \cite[Proposition 4.10]{Javanpeykar-Loughran:GoodReductionFano}).
As an application of our results in this paper,
we calculate the cardinality $\# \Shaf_{F,S}$ explicitly.

\begin{thm}
\label{thm:explicitShafarevichconjecture}
Let $F$ be a number field, and $S$ a finite set of finite places of $F$.
Let $T$ be the set of finite places of $F$ dividing $2$, and $U$ the set of real places of $F$.
We put $r \coloneqq  \#(S \cup T \cup U)$ (the union is taken within the set of all places of $F$).
Then we have $\# \Shaf_{F,S} = 2^{r-1}$.
\end{thm}

\begin{proof}
We put $G_{F} := \Gal (\overline{F}/F)$.
Let $Y \in \Shaf$ and 
\[
\alpha_Y \in H^1 (G_F, \PGL_{2,F}(\overline{F})) \simeq \Br (F)[2]
\]
the corresponding class.
Recall that we have the exact sequence
\begin{equation}
\label{eqn:Brauer}
0 \rightarrow \Br (F) \rightarrow  \prod_{v : \text{place of} \ F} \Br (F_{v}) \rightarrow \Q/\Z \rightarrow 0,
\end{equation}
where $v$ varies in all the places of $F$.
We denote the image of $\alpha_Y$ in the middle term by $(\alpha_{Y,v})_{v}$.
Note that for a non-Archimedean place $v$, $\alpha_{Y,v}=0$ if and only if $Y_{F_v}$ is split.
Hence for any non-Archimedean place $v$ with $v \notin S \cup T$, we have $\alpha_{Y,v} =0$ by Proposition \ref{prop:goodreductioncriteriaW5odd}.
Hence, we obtain an injection
\begin{equation*}
\begin{aligned}
\Shaf &\rightarrow
\left.\left
\{  
(\alpha_v)_v \in
\prod_{v \in S \cup T \cup U} (\frac{1}{2}\Z/\Z) \, 
\right| \, \sum \alpha_v  =0
\right\},\\
&Y \mapsto (\alpha_{Y,v})_v,
\end{aligned}
\end{equation*}
by (\ref{eqn:Brauer}).
Note that the cardinality of the right-hand side is equal to $2^{r-1}$.
Hence, it suffices to show the surjectivity.
For any element $(\alpha_v)_v$ of the right-hand side, by the exact sequence (\ref{eqn:Brauer}), there exists a $V_5$-variety $Y'$ over $F$ whose corresponding class $\alpha_{Y'}$ satisfying 
\[
\alpha_{Y',v} = 
\begin{cases}
\alpha_{v} & \textup{if $v \in S \cup T \cup U$}, \\
0 & \textup{otherwise}.
\end{cases}
\]
Then $Y' \in \Shaf$ by Propositions \ref{prop:goodreductioncriteriaW5odd}. 
It finishes the proof.
\end{proof}

\subsection{Brauer--Severi curves}

As a corollary of Theorem \ref{thm:W5description}, we can classify $V_5$-varieties over a field of characteristic different from two as follows.

\begin{prop}
\label{prop:field}
Let $k$ be a field of characteristic different from two.
Then there are natural bijections between the following sets:
\begin{enumerate}
\item The set of $k$-isomorphism classes of $V_5$-varieties over $k$.

\item The set of $k$-similarity classes of non-degenerate ternary symmetric bilinear forms over $k$.

\item The set of $k$-similarity classes of non-degenerate ternary quadratic forms over $k$.

\item The set of $k$-isomorphism classes of smooth conics over $k$.

\item The set of $k$-isomorphism classes of quaternion algebras over $k$.
\end{enumerate}
\end{prop}

\begin{proof}
Since $2$ is invertible in $k$,
the existence of natural bijections between (2), (3), (4) is well-known.
A bijection between (3) and (5) is given by the even Clifford algebra and the reduced norm
(see \cite[Theorem 5.1.1]{Voight21} for example).
Hence the assertion follows from Theorem \ref{thm:W5description}.
\end{proof}

The above result can be generalized to arbitrary schemes over $\Z[1/2]$.

\begin{defn}
Let $B$ be a scheme, and $f \colon C \rightarrow B$ a smooth projective morphism.
We say $f$ is a \emph{Brauer-Severi curve} over $B$ when any geometric fiber over $s$ is isomorphic to $\P^1_{k(s)}$.
\end{defn}

\begin{prop}
\label{prop:W5vsP1}
Let $B$ be a scheme with $2 \in \cO_{B}^{\times}$.
Then there is a natural bijection between
\[
\{\textup{Brauer--Severi curves over $B$}
\}/\simeq_B
\]
and
\[
\{\textup{$V_5$-schemes over $B$}
\}/\simeq_B.
\]
Under the correspondence, $\P^1_B$ corresponds to the split $V_5$-scheme $Y \subset \P^6_B$ over $B$.
\end{prop}

\begin{proof}
For a Brauer--Severi curve $f \colon C\rightarrow B$,
$-K_{C/B}$ is a very ample line bundle, and we have an embedding
\[
C \hookrightarrow \P_B( f_\ast (\cO(-K_{C/B}))).
\]
Let $\mathcal{I}_{C}$ be the ideal sheaf of $C$ via the embedding, and $\pi \colon \P_B( f_\ast (\cO(-K_{C/B}))) \rightarrow B$ the natural projection.
Then 
\[
\pi_{\ast} (\mathcal{I}_{C}(2)) \subset \pi_{\ast} \cO(2) \simeq S^2 (f_{\ast} \cO(-K_{C/B}))
\]
defines a ternary quadratic form over $B$.
Here, a \emph{ternary quadratic form} over $B$ is, by definition, a triple $(N_3,L,q)$ of a locally free sheaf $N_3$ of rank $3$, a line bundle $L$, and a locally direct summand morphism $q \colon L \to S^2N_3$.
Since $C$ is a relative conic in $\P_B( f_\ast (\cO(-K_{C/B})))$, the quadratic form is non-degenerate, i.e.,\ the associated symmetric bilinear form is non-degenerate.
This correspondence induces a bijection between
\[
\{\textup{Brauer--Severi curves over $B$}
\}/\simeq_B
\]
and
\[
\{\textup{non-degenerate ternary quadratic forms over $B$}
\}/\textup{$B$-similar}.
\]
Since $2 \in \cO_{B}^{\times}$, we obtain the desired bijection by Theorem \ref{thm:W5description}.
\end{proof}

\subsection{Finiteness of integral models}
 
The following lemma refines \cite[Lemma 3.15]{Javanpeykar-Loughran:GoodReductionFano}.

\begin{lem}
\label{lem:fppflocallysplit}
Let $B$ be a reduced scheme, and $\mathcal{Y}$ a $V_5$-scheme over $B$.
Then $\mathcal{Y}$ is split fppf locally, i.e.,\ there exists a flat surjective locally of finite presentation morphism $B' \rightarrow B$ such that $\mathcal{Y}_{B'}$ is a split $V_5$-scheme over $B'$.
\end{lem}

\begin{proof}
Let $(N_3, L, \varphi)$ be the non-degenerate ternary symmetric bilinear form corresponding to $\mathcal{Y}$.
By taking a Zariski localization, we may assume that $B = \Spec R$ is affine and $N_3$ and $L$ are free.
First, we show that $\varphi$ is diagonalizable Zariski locally.
Since $\Sigma (\mathcal{Y}/B)  \simeq \P (N_3)$, Zariski locally, $\mathcal{Y}$ admits an ordinary line over $B$.
Therefore, by Theorems \ref{thm:linep>2} and \ref{thm:linep=2}, we may take a section $P$ of $\mathcal{Y}$ over $B$ such that $P_{s} \in \mathcal{Y}_{s}$ is contained in an open orbit for any geometric point $s$ on $B$.
By the proof of Theorem \ref{thm:trisecant}, this section corresponds to an orthogonal basis of $\varphi$ over $B$.

By adjoining suitable square roots to $R$, we may assume that $\varphi$ is given by the diagonal matrix $A:=\diag (-1,1,1)$.
We put
\[
P :=
\begin{pmatrix}
    1 & 1 & 1\\
    0 & 1 & 1\\
    1 & 0 & 1
\end{pmatrix}.
\]
Then ${}^t\!P A P$ is equal to the matrix (\ref{eqn:splitmatrix}). It finishes the proof.
\end{proof}

\begin{rem}
In the setting of Lemma \ref{lem:fppflocallysplit}, when $2 \in \cO_{B}^{\times}$, $\mathcal{Y}$ is split \'{e}tale locally (this was proved in \cite[Lemma 3.15]{Javanpeykar-Loughran:GoodReductionFano} after inverting $N\gg0$).
This can be deduced from the proof above or, alternatively, from Proposition \ref{prop:W5vsP1}.
\end{rem}

The following is a strong form of the Shafarevich conjecture, more precisely, the finiteness of models over a fixed ring of $S$-integers, which was proved after inverting $N\gg0$ in the proof of \cite[Proposition 4.10]{Javanpeykar-Loughran:GoodReductionFano}.

\begin{thm}
\label{thm:finitenessofintegral models}
Let $F$ be a number field, and $S$ a finite set of finite places of $F$.
Then the set 
\[
\{
V_{5}\textup{-schemes over $\cO_{F,S}$}
\}
/ \simeq_{\cO_{F,S}}
\]
is a finite set.
\end{thm}
\begin{proof}
Let $\mathcal{Y}_0$ be the split $V_5$-scheme over $\cO_{F,S}$.
By Lemma \ref{lem:fppflocallysplit}, the set appearing in the statement can be seen as a subset of 
\[
H^1_{\mathrm{fppf}} (\cO_{F,S}, \Aut_{\mathcal{Y}_{0}/\cO_{F,S}}).
\]
By Proposition \ref{prop:aut_split_W5}, $\Aut_{\mathcal{Y}_{0}/\cO_{F,S}}$ is affine and of finite type.
Therefore, the desired finiteness follows from \cite[Proposition 7.13]{Gille-Moret-Bailly}.
\end{proof}

\section{Remarks on the automorphism group scheme in characteristic two}
\label{section:characteristic_two}

In this section,  we describe the non-reduced automorphism group scheme of $Y_{\spl}$ in characteristic two, and give some remarks related to this special feature in characteristic two.
In the following, $k$ is a field of characteristic two.
\subsection{The group scheme}
Here we slightly change the notation from Proposition~\ref{prop:aut_split_W5} and write as
\[A=
\begin{pmatrix}
 a_{\alpha\alpha}& a_{\alpha\beta}& a_{\alpha\gamma}\\
  a_{\beta\alpha}& a_{\beta\beta}& a_{\beta\gamma}\\
   a_{\gamma\alpha}& a_{\gamma\beta}& a_{\gamma\gamma}
\end{pmatrix}
=
\begin{pmatrix}
 a&b&e_1\\
 c&d&e_2\\
 f_1&f_2&m
\end{pmatrix}.
\]
Thus $\Aut_{Y/k}$ is isomorphic to
\[
G \subset \PGL_{3,k} = \Proj (k[a,b,c,d,e_1,e_2,f_1,f_2,m]) \setminus (\det =0)
\]
defined by the equations
\[
e_1^2=0,\quad
e_2^2=0,\quad
\begin{pmatrix}
 a&b\\c&d
\end{pmatrix}
\begin{pmatrix}
 f_2\\f_1
\end{pmatrix}
=
\begin{pmatrix}
m e_1\\ m e_2
\end{pmatrix},\quad
\det \begin{pmatrix}
 a&b\\c&d
\end{pmatrix}=m^2+e_1 e_2.
\]
Since
\[
\det A  =m^3+ me_1e_2 =m(m^2+e_1e_2),
\]
we have
\[
m(m^2+e_1e_2) = m\det \begin{pmatrix}
 a&b\\c&d
\end{pmatrix} \neq 0.
\]
Thus $m\neq 0$, and hence we may normalize as $m=1$.
Then the group is written as:
\[
G=\left\{ \left.
 A=
\begin{pmatrix}
 a&b&af_2+bf_1\\
 c&d&cf_2+df_1\\
 f_1&f_2&1
\end{pmatrix} \subset \PGL_{3,k}
\, \right| \,  f_1^2=f_2^2=0, \quad 1+f_1f_2=\det \begin{pmatrix}
 a&b\\c&d
\end{pmatrix}
\right\}.
\]
By the description, we have
\begin{prop}
 The automorphism group is non-reduced, and $\dim T_{[\id]} (\Aut_Y) =5$.
 In particular, we have $h^0(T_Y)=5$ and $h^1(T_Y) =2$ by \cite[(8.13)]{Fuj90}
\end{prop}

Moreover, by the same argument to the proof of Proposition~\ref{prop:aut_split_W5}, we have the following:
\begin{prop}\label{prop:non-reduced_action}
The action of $G$  on $Y_{\spl} \subset \P^6$ is given by
\begin{align*}
&A=\begin{pmatrix}
 a&b&e_1\coloneqq af_2+bf_1\\
 c&d&e_2\coloneqq c f_2+df_1\\
 f_1&f_2&1
\end{pmatrix}
\mapsto \\
&\begin{pmatrix}
a^3 & 0 &a^2b &0 & ab^2 &0 &b^3\\
0&a^2&0&ab&0&b^2&0\\
a^2c&0&a^2d&0&b^2c&0&b^2d\\
0&0&0&1&0&0&0\\
ac^2&0&bc^2&0&ad^2&0&bd^2\\
0&c^2&0&cd&0&d^2&0\\
c^3&0&c^2d&0&cd^2&0&d^3\\
\end{pmatrix}
+
\begin{pmatrix}
0 & a^2e_1 &0 &abe_1 & 0 & b^2e_1 & 0\\
a^2f_1&0&a^2f_2&0&b^2f_1&0&b^2f_2\\
0&a^2e_2&0&e_1+abe_2&0&b^2e_2&0\\
0&0&0&f_1f_2&0&0&0\\
0&c^2e_1&0&cde_1+e_2&0&d^2e_1&0\\
c^2f_1&0&c^2f_2&0&d^2f_1&0&d^2f_2\\
0&c^2e_2&0&cde_2&0&d^2e_2&0
\end{pmatrix}.
\end{align*}
\end{prop}

\subsection{A characteristic non-reduced subgroup}

As we have seen in Proposition~\ref{prop:aut_split_W5}, $G_{\red}$ is isomorphic to $\SL_{2,k}$, which is given by
\[
\left\{ \left. 
 A=
\begin{pmatrix}
 a&b&0\\
 c&d&0\\
 0&0&1
\end{pmatrix}
\, \right| \,  1=\det \begin{pmatrix}
 a&b\\c&d
\end{pmatrix}
\right\}.
 \]
In the following, we write this subgroup as $\SL_{2,k} \subset G$ by abuse of notation.
The following subgroup $H$ compensates for the non-reduced part of $\Aut_{Y/k}$:
 \[ H \coloneqq
 \left\{ \left. 
\begin{pmatrix}
 1&f_1&f_2\\
 f_2&1&f_1\\
 f_1&f_2&1
\end{pmatrix}
\, \right| \,  f_1^2=f_2^2=0
\right\} \subset G.
 \]
Indeed, we have
\begin{prop}\label{prop:non-reduced_subgroup}
 $H$ contains
  \[ K \coloneqq
 \left\{ \left. 
\begin{pmatrix}
 1&f&f\\
 f&1&f\\
 f&f&1
\end{pmatrix}
\, \right| \, f^2=0
\right\}
 \]
 as a normal subgroup, and $K \simeq H/K \simeq \mu_2$.
 Moreover $\SL_{2,k} \cdot H = G$.
\end{prop}

\begin{rem}\label{rem:quotient_V10}
Here we briefly describe the quotient of $Y$ by this subgroup $H$.
By the above description, an element $\begin{pmatrix}
 1&f_1&f_2\\
 f_2&1&f_1\\
 f_1&f_2&1
\end{pmatrix}$ of $H$
acts on $Y \subset \P^6$ by:
\[
\begin{pmatrix}
1 & f_2 &f_1 &f_1f_2 & 0 & 0 & 0\\
f_1&1&f_2&f_1&0&0&0\\
f_2&f_1&1&f_2&0&0&0\\
0&0&0&1+f_1f_2&0&0&0\\
0&0&0&f_1&1&f_2&f_1\\
0&0&0&f_2&f_1&1&f_2\\
0&0&0&f_1f_2&f_2&f_1&1
\end{pmatrix}.
\]
We can show that this action fixes quadratic forms $a_0^2$, $a_1^2$, $a_3^2$, $a_4^2$, $a_5^2$, $a_6^2$ and $a_1a_5+a_0a_3+a_3a_6$ respectively.
Define $V_{10} \subset \P^7_{(a_0,a_1,\dots,a_6,t)}$ by the equations
\begin{equation}
\label{eqn:V10}
\begin{cases}
a_0 a_4 - a_1 a_3 + a_2^2,\\
a_0 a_5 - a_1 a_4 + a_2 a_3,\\
a_0 a_6 - a_2 a_4 + a_3^2,\\
a_1 a_6 - a_2 a_5 + a_3 a_4,\\
a_2 a_6 - a_3 a_5 + a_4^2,\\
t^2-(a_1a_5+a_0a_3+a_3a_6).
\end{cases}
\end{equation}
Then the map $Y_{\spl} \to V_{10}$ given by
\[
(a_0,\dots,a_6) \mapsto (a_0^2,\dots,a_6^2 ,a_1a_5+a_0a_3+a_3a_6)
\]
provides the quotient map $Y \to Y/H$.
Note that $V_{10}$ is a singular Fano $3$-fold of Gushel--Mukai type ($g=6$), which admits a purely inseparable double cover over $Y_{\spl}$.
More precisely, the Frobenius morphism $Y_{\spl} \to Y_{\spl}$ factors as
\[
Y_{\spl } \to Y/H \simeq V_{10} \to Y_{\spl}.
\]
We can show that the singular locus of $V_{10}$ consists of three lines and three twisted cubics.
\end{rem}

\subsection{Non-trivial first order deformations}\label{subsection:deformation_p=2}
Since $h^1(T_Y) = 2$, the variety $Y$ admits a two-dimensional family of non-trivial first-order deformations.
Here we describe these deformations briefly.

Note that, by Theorem~\ref{thm:W5description}, we also have the corresponding two-dimensional family of non-trivial first-order deformations of the symmetric bilinear form $\varphi_{\spl}$.
This is easier to describe: Let $(\xi,\eta) \in k^2$, and define a ternary symmetric bilinear form 
$\varphi_{(\xi,\eta)}$ over the ring of dual numbers $k[t]/(t^2)$ as the form corresponding to the symmetric matrix
\[
\begin{pmatrix}
   t_1 \coloneqq \xi t & 1 & 0\\
    1 & t_2\coloneqq \eta t & 0\\
    0 & 0 & 1 
\end{pmatrix}.
\]
Then these define a two-dimensional family of non-trivial first-order deformations of the symmetric bilinear form $\varphi_{\spl}$.
Indeed, the reduced automorphism group $\SL_{2,k}$, which acts faithfully on $Y_{\spl}$,  acts on the above space of symmetric bilinear forms $\{(\xi, \eta) \in k^2\}$ via 
\[
\begin{pmatrix}a&b\\c&d\end{pmatrix} \cdot (\xi,\eta) = \begin{pmatrix}a^2&b^2\\c^2&d^2 \end{pmatrix} 
\begin{pmatrix}
 \xi \\ \eta
\end{pmatrix}.
\]
Thus, by Theorem~\ref{thm:W5description}, the corresponding $V_5$-schemes $Y_{(\xi,\eta)}$ over $k[t]/(t^2)$ give the two-dimensional family of non-trivial first order deformations.
By a similar argument to the proof of Definition-Proposition~\ref{defn-prop:splitform}, we  see that $Y_{(\xi,\eta)}$ is defined by the three equations
\[
\begin{cases}
    b_{34} = b_{25}+t_2b_{23}+t_1b_{45},\\
    b_{23} = b_{24}+t_2b_{12}+t_1b_{34},\\
    b_{24} = b_{25}+t_2b_{13}+t_1b_{35}
\end{cases}
\]
in $\Gr(U_5^\vee,2)$.
Equivalently, the deformation $Y_{(\xi,\eta)}$ is defined in $\P^6$ by the equations
\begin{equation}
\label{eqn:W5_deformation}
\begin{cases}
a_0 a_4 + a_1 a_3 + a_2^2 + t_1(a_0 a_6 + a_1 a_5 + a_2 a_4) + t_2 a_1^2,\\
a_0 a_5 + a_1 a_4 + a_2 a_3 + t_1 a_3 a_4 + t_2 a_0 a_3,\\
a_0 a_6 + a_2 a_4 + a_3^2 + t_1 a_3 a_5 + t_2 a_1 a_3,\\
a_1 a_6 + a_2 a_5 + a_3 a_4 + t_1 a_3 a_6 + t_2 a_2 a_3,\\
a_2 a_6 + a_3 a_5 + a_4^2 + t_1 a_5^2 + t_2(a_0 a_6 + a_1 a_5 + a_2 a_4).
\end{cases}
\end{equation}
Note that, by the above observation, the reduced automorphism group $\SL_{2,k}$ acts on $H^1(T_{Y})$ via the Frobenius twist of the standard representation, i.e.,\ the representation on $H^1(T_{Y})$ is isomorphic to
\[
\begin{pmatrix}a&b\\c&d\end{pmatrix} \mapsto \begin{pmatrix}a^2&b^2\\c^2&d^2 \end{pmatrix}.
\]

Note that, the $V_5$-scheme $Y$ over $B:=\mathbb{A}^2_{(t_1,t_2)}$ defined by (\ref{eqn:W5_deformation}) 
satisfies the following property:
\begin{itemize}
\item 
 All closed fibers of $Y$ are isomorphic to one another, but $Y$ is not locally trivial in the \'{e}tale topology.
\end{itemize}
Such examples were constructed in \cite[Section 8]{Poczobut_isotrivial} using Enriques surfaces in characteristic two and surfaces of general type in characteristic five.

\printbibliography

@article {Mukaicurve,
    AUTHOR = {Mukai, Shigeru},
     TITLE = {Curves and symmetric spaces. {I}},
   JOURNAL = {Amer. J. Math.},
  FJOURNAL = {American Journal of Mathematics},
    VOLUME = {117},
      YEAR = {1995},
    NUMBER = {6},
     PAGES = {1627--1644},
      ISSN = {0002-9327,1080-6377},
   MRCLASS = {14H45 (14J45 14M17)},
  MRNUMBER = {1363081},
MRREVIEWER = {Raquel\ Mallavibarrena},
       DOI = {10.2307/2375032},
       URL = {https://doi.org/10.2307/2375032},
}

@article {Javanpeykar-Loughran:GoodReductionFano,
    AUTHOR = {Javanpeykar, Ariyan and Loughran, Daniel},
     TITLE = {Good reduction of {F}ano threefolds and sextic surfaces},
   JOURNAL = {Ann. Sc. Norm. Super. Pisa Cl. Sci. (5)},
  FJOURNAL = {Annali della Scuola Normale Superiore di Pisa. Classe di
              Scienze. Serie V},
    VOLUME = {18},
      YEAR = {2018},
    NUMBER = {2},
     PAGES = {509--535},
      ISSN = {0391-173X,2036-2145},
   MRCLASS = {11G35 (14G40 14J45 14K30)},
  MRNUMBER = {3801287},
MRREVIEWER = {Ulrich\ Derenthal},
}

@article {Kuznetsov-Prokhorov-Shramov,
    AUTHOR = {Kuznetsov, Alexander G. and Prokhorov, Yuri G. and Shramov,
              Constantin A.},
     TITLE = {Hilbert schemes of lines and conics and automorphism groups of
              {F}ano threefolds},
   JOURNAL = {Jpn. J. Math.},
  FJOURNAL = {Japanese Journal of Mathematics},
    VOLUME = {13},
      YEAR = {2018},
    NUMBER = {1},
     PAGES = {109--185},
      ISSN = {0289-2316,1861-3624},
   MRCLASS = {14J45 (14C05 14J30 14J50)},
  MRNUMBER = {3776469},
MRREVIEWER = {Alexandr\ V.\ Pukhlikov},
       DOI = {10.1007/s11537-017-1714-6},
       URL = {https://doi.org/10.1007/s11537-017-1714-6},
}

@article {Licht,
    AUTHOR = {Licht, Philipp},
     TITLE = {Hyperbolicity of the moduli of certain {F}ano threefolds},
   JOURNAL = {Acta Arith.},
  FJOURNAL = {Acta Arithmetica},
    VOLUME = {206},
      YEAR = {2022},
    NUMBER = {1},
     PAGES = {75--95},
      ISSN = {0065-1036,1730-6264},
   MRCLASS = {14J20 (14D23 14G05 14J45 32Q45)},
  MRNUMBER = {4515667},
MRREVIEWER = {Su-ion\ Ih},
       DOI = {10.4064/aa220322-22-9},
       URL = {https://doi.org/10.4064/aa220322-22-9},
}

@article {Scholl,
    AUTHOR = {Scholl, A. J.},
     TITLE = {A finiteness theorem for del {P}ezzo surfaces over algebraic
              number fields},
   JOURNAL = {J. London Math. Soc. (2)},
  FJOURNAL = {Journal of the London Mathematical Society. Second Series},
    VOLUME = {32},
      YEAR = {1985},
    NUMBER = {1},
     PAGES = {31--40},
      ISSN = {0024-6107,1469-7750},
   MRCLASS = {14J20 (11G35 14J26)},
  MRNUMBER = {813382},
MRREVIEWER = {Gerd\ Faltings},
       DOI = {10.1112/jlms/s2-32.1.31},
       URL = {https://doi.org/10.1112/jlms/s2-32.1.31},
}

@article {Kuznetsovspinor,
    AUTHOR = {Kuznetsov, A. G.},
     TITLE = {On linear sections of the spinor tenfold. {I}},
   JOURNAL = {Izv. Ross. Akad. Nauk Ser. Mat.},
  FJOURNAL = {Izvestiya Rossiiskoi Akademii Nauk. Seriya Matematicheskaya},
    VOLUME = {82},
      YEAR = {2018},
    NUMBER = {4},
     PAGES = {53--114},
      ISSN = {1607-0046,2587-5906},
   MRCLASS = {14M17 (14C05 14F05 14M15 18E30)},
  MRNUMBER = {3833474},
MRREVIEWER = {Nicolae\ Manolache},
       DOI = {10.4213/im8756},
       URL = {https://doi.org/10.4213/im8756},
}

@incollection {Gille-Moret-Bailly,
    AUTHOR = {Gille, Philippe and Moret-Bailly, Laurent},
     TITLE = {Actions alg\'{e}briques de groupes arithm\'{e}tiques},
 BOOKTITLE = {Torsors, \'{e}tale homotopy and applications to rational
              points},
    SERIES = {London Math. Soc. Lecture Note Ser.},
    VOLUME = {405},
     PAGES = {231--249},
 PUBLISHER = {Cambridge Univ. Press, Cambridge},
      YEAR = {2013},
      ISBN = {978-1-107-61612-7},
   MRCLASS = {14L15 (14G25)},
  MRNUMBER = {3077171},
MRREVIEWER = {Marcin\ Skrzy\'{n}ski},
}

@article{Achinger-Youcis,
  title={Beauville-Laszlo gluing of algebraic spaces},
  author={Achinger, Piotr and Youcis, Alex},
  journal={arXiv preprint arXiv:2410.20500},
  year={2024}
}

@book {Fuj90,
    AUTHOR = {Fujita, Takao},
     TITLE = {Classification theories of polarized varieties},
    SERIES = {London Mathematical Society Lecture Note Series},
    VOLUME = {155},
 PUBLISHER = {Cambridge University Press, Cambridge},
      YEAR = {1990},
     PAGES = {xiv+205},
      ISBN = {0-521-39202-0},
   MRCLASS = {14C20 (14J40 14J60)},
  MRNUMBER = {1162108},
MRREVIEWER = {Elvira\ Laura\ Livorni},
       DOI = {10.1017/CBO9780511662638},
       URL = {https://doi.org/10.1017/CBO9780511662638},
}

@article {DK19,
    AUTHOR = {Dubouloz, Adrien and Kishimoto, Takashi},
     TITLE = {Cylindres dans les fibrations de {M}ori: formes du volume
              quintique de del {P}ezzo},
   JOURNAL = {Ann. Inst. Fourier (Grenoble)},
  FJOURNAL = {Universit\'e{} de Grenoble. Annales de l'Institut Fourier},
    VOLUME = {69},
      YEAR = {2019},
    NUMBER = {6},
     PAGES = {2377--2393},
      ISSN = {0373-0956,1777-5310},
   MRCLASS = {14R25 (14E30 14J45 14R10)},
  MRNUMBER = {4033922},
MRREVIEWER = {Tomasz\ Pe\l ka},
       DOI = {10.5802/aif.3297},
       URL = {https://doi.org/10.5802/aif.3297},
}

@incollection {Mukai-Umemura,
    AUTHOR = {Mukai, Shigeru and Umemura, Hiroshi},
     TITLE = {Minimal rational threefolds},
 BOOKTITLE = {Algebraic geometry ({T}okyo/{K}yoto, 1982)},
    SERIES = {Lecture Notes in Math.},
    VOLUME = {1016},
     PAGES = {490--518},
 PUBLISHER = {Springer, Berlin},
      YEAR = {1983},
      ISBN = {3-540-12685-6},
   MRCLASS = {14J30 (14E35 14M20)},
  MRNUMBER = {726439},
MRREVIEWER = {Mary\ Schaps},
       DOI = {10.1007/BFb0099976},
       URL = {https://doi.org/10.1007/BFb0099976},
}

@article{IKTT,
  title={Arithmetic finiteness of Mukai varieties of genus 7},
  author={Ito, Tetsushi and Kanemitsu, Akihiro and Takamatsu, Teppei and Tanaka, Yuuji},
  journal={arXiv preprint arXiv:2409.20046},
  year={2024}
}

@article {Meg98,
    AUTHOR = {Megyesi, G.},
     TITLE = {Fano threefolds in positive characteristic},
   JOURNAL = {J. Algebraic Geom.},
  FJOURNAL = {Journal of Algebraic Geometry},
    VOLUME = {7},
      YEAR = {1998},
    NUMBER = {2},
     PAGES = {207--218},
      ISSN = {1056-3911,1534-7486},
   MRCLASS = {14J45},
  MRNUMBER = {1620094},
MRREVIEWER = {Yuri\ G.\ Prokhorov},
}

@book {Brion-Kumar,
    AUTHOR = {Brion, Michel and Kumar, Shrawan},
     TITLE = {Frobenius splitting methods in geometry and representation
              theory},
    SERIES = {Progress in Mathematics},
    VOLUME = {231},
 PUBLISHER = {Birkh\"auser Boston, Inc., Boston, MA},
      YEAR = {2005},
     PAGES = {x+250},
      ISBN = {0-8176-4191-2},
   MRCLASS = {14M15 (13A35 14C05 17B10 20G05)},
  MRNUMBER = {2107324},
MRREVIEWER = {Vikram\ B.\ Mehta},
}

@article{Tanaka3,
  title={Fano threefolds in positive characteristic III},
  author={Asai, Masaya and Tanaka, Hiromu},
  journal={arXiv preprint arXiv:2308.08124},
  year={2023}
}

@article {Iskovskih1,
    AUTHOR = {Iskovskih, V. A.},
     TITLE = {Fano threefolds. {I}},
   JOURNAL = {Izv. Akad. Nauk SSSR Ser. Mat.},
  FJOURNAL = {Izvestiya Akademii Nauk SSSR. Seriya Matematicheskaya},
    VOLUME = {41},
      YEAR = {1977},
    NUMBER = {3},
     PAGES = {516--562, 717},
      ISSN = {0373-2436},
   MRCLASS = {14J10 (14M20 14N05)},
  MRNUMBER = {463151},
MRREVIEWER = {Miles\ Reid},
}

@article {Mori-Mukaiclassification,
    AUTHOR = {Mori, Shigefumi and Mukai, Shigeru},
     TITLE = {Classification of {F}ano {$3$}-folds with {$B\sb{2}\geq 2$}},
   JOURNAL = {Manuscripta Math.},
  FJOURNAL = {Manuscripta Mathematica},
    VOLUME = {36},
      YEAR = {1981/82},
    NUMBER = {2},
     PAGES = {147--162},
      ISSN = {0025-2611,1432-1785},
   MRCLASS = {14J30 (14J10)},
  MRNUMBER = {641971},
MRREVIEWER = {Mary\ Schaps},
       DOI = {10.1007/BF01170131},
       URL = {https://doi.org/10.1007/BF01170131},
}

@book {Milnor-Husemoller,
    AUTHOR = {Milnor, John and Husemoller, Dale},
     TITLE = {Symmetric bilinear forms},
    SERIES = {Ergebnisse der Mathematik und ihrer Grenzgebiete [Results in
              Mathematics and Related Areas]},
    VOLUME = {Band 73},
 PUBLISHER = {Springer-Verlag, New York-Heidelberg},
      YEAR = {1973},
     PAGES = {viii+147},
   MRCLASS = {15A63 (10C05 57D65)},
  MRNUMBER = {506372},
MRREVIEWER = {Louis\ H.\ Kauffman},
}

@incollection {Milnor1971,
    AUTHOR = {Milnor, John},
     TITLE = {Symmetric inner products in characteristic {$2$}},
 BOOKTITLE = {Prospects in mathematics ({P}roc. {S}ympos., {P}rinceton
              {U}niv., {P}rinceton, {N}.{J}., 1970)},
    SERIES = {Ann. of Math. Stud.},
    VOLUME = {No. 70},
     PAGES = {59--75},
 PUBLISHER = {Princeton Univ. Press, Princeton, NJ},
      YEAR = {1971},
   MRCLASS = {15A63 (10C05)},
  MRNUMBER = {347866},
MRREVIEWER = {Alex\ Rosenberg},
}

@article {PCS19,
    AUTHOR = {Przhiyalkovski\u i, V. V. and Cheltsov, I. A. and
              Shramov, K. A.},
     TITLE = {Fano threefolds with infinite automorphism groups},
   JOURNAL = {Izv. Ross. Akad. Nauk Ser. Mat.},
  FJOURNAL = {Izvestiya Rossiiskoi Akademii Nauk. Seriya Matematicheskaya},
    VOLUME = {83},
      YEAR = {2019},
    NUMBER = {4},
     PAGES = {226--280},
      ISSN = {1607-0046,2587-5906},
   MRCLASS = {14J45 (14J50)},
  MRNUMBER = {3985696},
MRREVIEWER = {Jaros\l aw\ A.\ Wi\'sniewski},
       DOI = {10.4213/im8834},
       URL = {https://doi.org/10.4213/im8834},
}

@article {Ili94,
%    AUTHOR = {Iliev, Atanas},
%     TITLE = {Lines on the {G}ushel\cprime\ threefold},
%   JOURNAL = {Indag. Math. (N.S.)},
%  FJOURNAL = {Koninklijke Nederlandse Akademie van Wetenschappen.
%              Indagationes Mathematicae. New Series},
%    VOLUME = {5},
%      YEAR = {1994},
%    NUMBER = {3},
%     PAGES = {307--320},
%      ISSN = {0019-3577,1872-6100},
%   MRCLASS = {14J45 (14J05 14J30)},
%  MRNUMBER = {1298777},
%MRREVIEWER = {Yuri\ G.\ Prokhorov},
%       DOI = {10.1016/0019-3577(94)90006-X},
%       URL = {https://doi.org/10.1016/0019-3577(94)90006-X},
%}

@incollection {Muk92,
    AUTHOR = {Mukai, Shigeru},
     TITLE = {Fano {$3$}-folds},
 BOOKTITLE = {Complex projective geometry ({T}rieste, 1989/{B}ergen, 1989)},
    SERIES = {London Math. Soc. Lecture Note Ser.},
    VOLUME = {179},
     PAGES = {255--263},
 PUBLISHER = {Cambridge Univ. Press, Cambridge},
      YEAR = {1992},
      ISBN = {0-521-43352-5},
   MRCLASS = {14J45 (14J30 14J60 14N05)},
  MRNUMBER = {1201387},
MRREVIEWER = {Jaros\l aw\ A.\ Wi\'sniewski},
       DOI = {10.1017/CBO9780511662652.018},
       URL = {https://doi.org/10.1017/CBO9780511662652.018},
}

@article {FN89,
    AUTHOR = {Furushima, Mikio and Nakayama, Noboru},
     TITLE = {The family of lines on the {F}ano threefold {$V_5$}},
   JOURNAL = {Nagoya Math. J.},
  FJOURNAL = {Nagoya Mathematical Journal},
    VOLUME = {116},
      YEAR = {1989},
     PAGES = {111--122},
      ISSN = {0027-7630,2152-6842},
   MRCLASS = {14J30 (14C05 14F05)},
  MRNUMBER = {1029973},
MRREVIEWER = {Jaros\l aw\ A.\ Wi\'sniewski},
       DOI = {10.1017/S0027763000001719},
       URL = {https://doi.org/10.1017/S0027763000001719},
}

@book {Dol12,
    AUTHOR = {Dolgachev, Igor V.},
     TITLE = {Classical algebraic geometry},
      NOTE = {A modern view},
 PUBLISHER = {Cambridge University Press, Cambridge},
      YEAR = {2012},
     PAGES = {xii+639},
      ISBN = {978-1-107-01765-8},
   MRCLASS = {14-02 (14-01)},
  MRNUMBER = {2964027},
MRREVIEWER = {Arnaud\ Beauville},
       DOI = {10.1017/CBO9781139084437},
       URL = {https://doi.org/10.1017/CBO9781139084437},
}

@article {Gyo90,
    AUTHOR = {Gyoja, Akihiko},
     TITLE = {Construction of invariants},
   JOURNAL = {Tsukuba J. Math.},
  FJOURNAL = {Tsukuba Journal of Mathematics},
    VOLUME = {14},
      YEAR = {1990},
    NUMBER = {2},
     PAGES = {437--457},
      ISSN = {0387-4982,2423-821X},
   MRCLASS = {22E47 (20G05 32M10)},
  MRNUMBER = {1085211},
MRREVIEWER = {J.\ S.\ Joel},
       DOI = {10.21099/tkbjm/1496161465},
       URL = {https://doi.org/10.21099/tkbjm/1496161465},
}

@article {Och97,
    AUTHOR = {Ochiai, Hiroyuki},
     TITLE = {Quotients of some prehomogeneous vector spaces},
   JOURNAL = {J. Algebra},
  FJOURNAL = {Journal of Algebra},
    VOLUME = {192},
      YEAR = {1997},
    NUMBER = {1},
     PAGES = {61--73},
      ISSN = {0021-8693,1090-266X},
   MRCLASS = {14L30 (14D25 14M17)},
  MRNUMBER = {1449952},
MRREVIEWER = {Dmitri\ I.\ Panyushev},
       DOI = {10.1006/jabr.1996.6979},
       URL = {https://doi.org/10.1006/jabr.1996.6979},
}

@article {SK77,
    AUTHOR = {Sato, M. and Kimura, T.},
     TITLE = {A classification of irreducible prehomogeneous vector spaces
              and their relative invariants},
   JOURNAL = {Nagoya Math. J.},
  FJOURNAL = {Nagoya Mathematical Journal},
    VOLUME = {65},
      YEAR = {1977},
     PAGES = {1--155},
      ISSN = {0027-7630,2152-6842},
   MRCLASS = {32M10 (20G05)},
  MRNUMBER = {430336},
MRREVIEWER = {A.\ L.\ Onishchik},
       URL = {http://projecteuclid.org/euclid.nmj/1118796150},
}

@book {FGA,
    AUTHOR = {Fantechi, Barbara and G\"ottsche, Lothar and Illusie, Luc and
              Kleiman, Steven L. and Nitsure, Nitin and Vistoli, Angelo},
     TITLE = {Fundamental algebraic geometry},
    SERIES = {Mathematical Surveys and Monographs},
    VOLUME = {123},
      NOTE = {Grothendieck's FGA explained},
 PUBLISHER = {American Mathematical Society, Providence, RI},
      YEAR = {2005},
     PAGES = {x+339},
      ISBN = {0-8218-3541-6},
   MRCLASS = {14-06 (14A15 14D15 14F20)},
  MRNUMBER = {2222646},
MRREVIEWER = {Liam\ O'Carroll},
       DOI = {10.1090/surv/123},
       URL = {https://doi.org/10.1090/surv/123},
}

@article {Poczobut_isotrivial,
    AUTHOR = {Poczobut, Pawe\l},
     TITLE = {An algebraic variant of the {F}ischer--{G}rauert {T}heorem},
   JOURNAL = {Math. Z.},
  FJOURNAL = {Mathematische Zeitschrift},
    VOLUME = {310},
      YEAR = {2025},
    NUMBER = {2},
     PAGES = {Paper No. 33},
      ISSN = {0025-5874,1432-1823},
   MRCLASS = {14F20 (14D06 32C15)},
  MRNUMBER = {4898071},
       DOI = {10.1007/s00209-025-03728-4},
       URL = {https://doi.org/10.1007/s00209-025-03728-4},
}

@article{V22,
Author = {Ito, Tetsushi and Kanemitsu, Akihiro and Takamatsu, Teppei and Tanaka, Yuuji},
Title = {Fano threefolds of genus 12 with large automorphism group in positive and mixed characteristic},
Year = {2025},
Journal = {In preparation},
}

@incollection {Arithmeticinvariant2,
    AUTHOR = {Bhargava, Manjul and Gross, Benedict H. and Wang, Xiaoheng},
     TITLE = {Arithmetic invariant theory {II}: {P}ure inner forms and
              obstructions to the existence of orbits},
 BOOKTITLE = {Representations of reductive groups},
    SERIES = {Progr. Math.},
    VOLUME = {312},
     PAGES = {139--171},
 PUBLISHER = {Birkh\"auser/Springer, Cham},
      YEAR = {2015},
      ISBN = {978-3-319-23442-7; 978-3-319-23443-4},
   MRCLASS = {11E72 (14L24)},
  MRNUMBER = {3495795},
MRREVIEWER = {Igor\ A.\ Rapinchuk},
       DOI = {10.1007/978-3-319-23443-4\_5},
       URL = {https://doi.org/10.1007/978-3-319-23443-4_5},
}

@article {Bhargava-Ho,
    AUTHOR = {Bhargava, Manjul and Ho, Wei},
     TITLE = {Coregular spaces and genus one curves},
   JOURNAL = {Camb. J. Math.},
  FJOURNAL = {Cambridge Journal of Mathematics},
    VOLUME = {4},
      YEAR = {2016},
    NUMBER = {1},
     PAGES = {1--119},
      ISSN = {2168-0930,2168-0949},
   MRCLASS = {20G05 (14H10 14H52 14L30)},
  MRNUMBER = {3472915},
MRREVIEWER = {Toshiyuki\ Tanisaki},
       DOI = {10.4310/CJM.2016.v4.n1.a1},
       URL = {https://doi.org/10.4310/CJM.2016.v4.n1.a1},
}

@article {Bhargava-HighercompositionIV,
    AUTHOR = {Bhargava, Manjul},
     TITLE = {Higher composition laws. {IV}. {T}he parametrization of
              quintic rings},
   JOURNAL = {Ann. of Math. (2)},
  FJOURNAL = {Annals of Mathematics. Second Series},
    VOLUME = {167},
      YEAR = {2008},
    NUMBER = {1},
     PAGES = {53--94},
      ISSN = {0003-486X,1939-8980},
   MRCLASS = {11E76 (11R20 11R29)},
  MRNUMBER = {2373152},
MRREVIEWER = {A.\ G.\ Earnest},
       DOI = {10.4007/annals.2008.167.53},
       URL = {https://doi.org/10.4007/annals.2008.167.53},
}

@article{Kawakami-Tanaka-weak,
  title={Weak quasi-F-splitting and del Pezzo varieties},
  author={Kawakami, Tatsuro and Tanaka, Hiromu},
  journal={Journal of the London Mathematical Society},
  volume={111},
  number={2},
  pages={e70098},
  year={2025},
  publisher={Wiley Online Library}
}

@article {Kuz25,
    AUTHOR = {Kuznetsov, Alexander},
     TITLE = {Derived categories of families of {F}ano threefolds},
   JOURNAL = {Algebr. Geom.},
  FJOURNAL = {Algebraic Geometry},
    VOLUME = {12},
      YEAR = {2025},
    NUMBER = {4},
     PAGES = {519--574},
      ISSN = {2313-1691,2214-2584},
   MRCLASS = {14F08 (14J45 18G80)},
  MRNUMBER = {4927290},
}

@book {Voight21,
    AUTHOR = {Voight, John},
     TITLE = {Quaternion algebras},
    SERIES = {Graduate Texts in Mathematics},
    VOLUME = {288},
 PUBLISHER = {Springer, Cham},
      YEAR = {[2021] \copyright 2021},
     PAGES = {xxiii+885},
      ISBN = {978-3-030-56692-0; 978-3-030-56694-4},
   MRCLASS = {11R52 (11-02 11E12 11F06 16H05 16U60 20H10)},
  MRNUMBER = {4279905},
MRREVIEWER = {Juliusz\ Brzezi\'nski},
       DOI = {10.1007/978-3-030-56694-4},
       URL = {https://doi.org/10.1007/978-3-030-56694-4},
}

@book {Neukirch:Algebraicnumbertheory,
    AUTHOR = {Neukirch, J\"urgen},
     TITLE = {Algebraic number theory},
    SERIES = {Grundlehren der mathematischen Wissenschaften [Fundamental
              Principles of Mathematical Sciences]},
    VOLUME = {322},
      NOTE = {Translated from the 1992 German original and with a note by
              Norbert Schappacher,
              With a foreword by G. Harder},
 PUBLISHER = {Springer-Verlag, Berlin},
      YEAR = {1999},
     PAGES = {xviii+571},
      ISBN = {3-540-65399-6},
   MRCLASS = {11Rxx (11-02 11S15 11S31 14C40)},
  MRNUMBER = {1697859},
MRREVIEWER = {Cornelius\ Greither},
       DOI = {10.1007/978-3-662-03983-0},
       URL = {https://doi.org/10.1007/978-3-662-03983-0},
}
\end{document}